\documentclass[10pt]{article}
\usepackage[dvipdfmx]{graphicx}
\usepackage{amsmath}
\usepackage{amssymb}
\usepackage{amsthm}

\usepackage{amsfonts}
\usepackage[none]{hyphenat}
\usepackage{hyperref}

	\usepackage[mathlines]{lineno}

\usepackage{bm}
\usepackage{url}
\usepackage{float}
\usepackage{mathtools}
\usepackage{subcaption}
\usepackage[usenames]{color}
\usepackage[table,xcdraw]{xcolor}
\usepackage{nicematrix}
\newtheorem{dfn}{Definition}[section]
\newtheorem{thm}[dfn]{Theorem}
\newtheorem{prop}[dfn]{Proposition}
\newtheorem{lem}[dfn]{Lemma}

\newtheorem{rem}[dfn]{Remark}

\newcommand{\R}{\mathbb{R}}
\newcommand{\C}{\mathbb{C}}
\newcommand{\N}{\mathbb{N}}
\newcommand{\Z}{\mathbb{Z}}
\newcommand{\cT}{\mathcal{T}}
\newcommand{\cC}{\mathcal{C}}
\newcommand{\cG}{\mathcal{G}}
\newcommand{\cX}{\mathcal{X}}
\newcommand{\cY}{\mathcal{Y}}
\newcommand{\bx}{\bar{x}}
\newcommand{\tx}{\tilde{x}}
\newcommand{\ba}{\bar{a}}

\newcommand{\bydef}{\stackrel{\mbox{\tiny\textnormal{\raisebox{0ex}[0ex][0ex]{def}}}}{=}} 
\DeclareMathOperator{\Log}{Log}
\DeclareMathOperator{\Arg}{Arg}

\DeclareMathOperator{\Sqrt}{Sqrt}
\newcommand{\norm}[1]{\left\lVert#1\right\rVert}
\newtheorem{theorem}{Theorem}

\numberwithin{equation}{section}
\usepackage[top=1in, bottom=1in, left=1in, right=1in]{geometry}

\begin{document}
\title{
Spatially inhomogeneous two-cycles in an integrodifference equation}
\author{
Kevin Church \thanks{Centre de Recherches Mathématiques, Université de Montréal, 
Pavillon André-Aisenstadt,
2920 chemin de la Tour,
Montréal, QC, H3T 1J4, Canada (\texttt{kevin.church@umontreal.ca})
}
\and
Kevin Constantineau \thanks{Department of Mathematics and Statistics, McGill University, 805 Sherbrooke West, Montreal, QC, H3A 0B9, Canada (\texttt{kevin.constantineau@mail.mcgill.ca})
}
\and
Jean-Philippe~Lessard \thanks{Department of Mathematics and Statistics, McGill University, 805 Sherbrooke West, Montreal, QC, H3A 0B9, Canada (\texttt{jp.lessard@mcgill.ca})
}
}
\date{\today}
\maketitle

\begin{abstract}

In this work, we prove the existence of a 2-cycle in an integrodifference equation with a Laplace kernel and logistic growth function, connecting two non-trivial fixed points of the second iterate of the logistic map in the non-chaotic regime. This model was first studied by Kot (1992), and the 2-cycle we establish corresponds to one numerically observed by Bourgeois, Leblanc, and Lutscher (2018) for the Ricker growth function. We provide strong evidence that the 2-cycle for the Ricker growth function can be rigorously proven using a similar approach. Finally, we present numerical results indicating that both 2-cycles exhibit spectral stability.
\end{abstract}





\section{Introduction} \label{sec:intro}

Understanding how populations or quantities change over space and time is a central problem in ecology, epidemiology, and urban planning. In many cases, this requires not only observing distributions but also predicting their evolution under growth and dispersal mechanisms. For example, an understanding of motor vehicle crash hotspots \cite{BIL201982} can help identify the most dangerous roads in a city, suggesting some candidates for traffic calming measures. Identifying migratory corridors of endangered species, such as the monarch butterfly \cite{TKB2019}, are important for their conservation. Areas with high human population density in the United States were connected with relatively earlier outbreaks of COVID-19 \cite{CFPSRS2020}. In each case, understanding the spatial density--whether of vehicle collisions, butterflies, or infections--can guide decision-making, and having a model that predicts how these populations or quantities evolve across space further improves this process. One way to mathematically formalize this spatial-temporal evolution is through integro-difference equations \cite{KOTSCHAFFER}, which provide a natural framework to describe how populations or other quantities grow and disperse across space in discrete time. These equations are often viewed as discrete-time analogues of reaction-diffusion equations, capturing both local growth and nonlocal dispersal in a unified model. In the scalar case, the object of interest is an equation
\begin{equation} \label{IDE1}\tag{IDE}
N_{t+1}(x) = Q[N_t](x)= \int_\R K(x-y)F(N_t(y))dy,
\end{equation}
where $K$ is the ``dispersal kernel", and $F$ describes single-generation growth dynamics. The dispersal kernel is typically required to be non-negative with unit $L^1$ norm. When a population evolves across discrete generations, a model of this type can be appropriate. For background on these equations in the context of ecological dynamics, we refer the reader to the monograph \cite{Lutscher2019}. These equations can produce a variety of structured solutions, some of which are more analytically tractable than others. One well-studied class consists of biological ``invasions", where travelling wave solutions $N_t(x)=W(x-ct)$ are sought, with $c$ the invasion speed and $W$ a static profile. Extensive work has been devoted to analyzing these solutions from both theoretical \cite{Bourgeois2018,Kot1992,Kot1996,Li2016} and data-driven perspectives \cite{Lewis2016}. More recently, research has explored periodic travelling waves in periodic environments \cite{Lin2020}. 

In this work, we focus on 2-cycles, which are solutions $N_t$ satisfying $N_{t+2} = N_t$ for all $t$, but with $N_{t+1} \neq N_t$. When the non-spatial model $n_{t+1}=F(n_t)$ has a two-cycle alternating between two constants $n_-<n_+$, equation \eqref{IDE1} trivially also has a two-cycle that is spatially homogeneous, alternating between $n_-$ and $n_+$ across generations. Less trivial to identify are 2-cycles that are spatially inhomogeneous -- that is, for which $x\mapsto N_t(x)$ is not constant. In 2018, Bourgeois, Leblanc and Lutscher \cite{Bourgeois2018} numerically found a standing wave of the second iterate operator associated with \eqref{IDE1}, for Laplace dispersal kernel $K(u)=\frac{\sigma}{2}e^{-\sigma|u|}$ and Ricker growth function $F(u)=u\exp(\rho(1-u))$, for parameters $\sigma=10$ and $\rho=2.2$. The standing wave connects the states of the two-cycle $\{n_-,n_+\}$ of the associated discrete-time Ricker map $n_{t+1}=F(n_t)$. This standing wave is equivalent to a 2-cycle in \eqref{IDE1}. To our knowledge, there is no mathematical proof for the existence of this, or any nontrivial (spatially inhomogeneous) 2-cycle in the model \eqref{IDE1}.

Building on these observations, we rigorously prove in the present work the existence of a spatially inhomogeneous 2-cycle using a computer-assisted approach. The precise statement is given in the following theorem.
\begin{thm}\label{thm-logistic}
With the Laplace-type dispersal kernel $K(u)=\frac{\sigma}{2}e^{-\sigma|u|}$ and logistic growth function $F(u)=(1+\rho)u - \rho u^2$, the integrodifference equation \eqref{IDE1} has, for $\rho=2.2$ and any $\sigma>0$, a 2-cycle $N_t$ satisfying:
\begin{itemize}
\item $n_-=\lim_{x\rightarrow -\infty}N_0(x)=\lim_{x\rightarrow +\infty}N_1(x)$
\item $n_+=\lim_{x\rightarrow +\infty}N_0(x)=\lim_{x\rightarrow -\infty}N_1(x)$,
\end{itemize}
where $\{n_-,n_+\}$ is the 2-cycle of the non-spatial logistic growth model $n_{t+1}=F(n_t)$, satisfying $n_-<n_+$.
\end{thm}
See Figure \ref{figure:logistic-solution} for a plot of the 2-cycle. While this theorem does not confirm the existence of the 2-cycle found numerically in \cite{Bourgeois2018}, which used the Ricker growth function, our method of proof could be modified to accomodate different growth functions. Section \ref{Sec3} presents a large portion of the work in this direction, providing extremely compelling numerical evidence that this strategy would work for the 2-cycle in \cite{Bourgeois2018}.
\begin{figure}[h!]
    \centering
    \includegraphics[width=0.65\textwidth, trim={16cm 0cm 0cm 0cm}, clip]{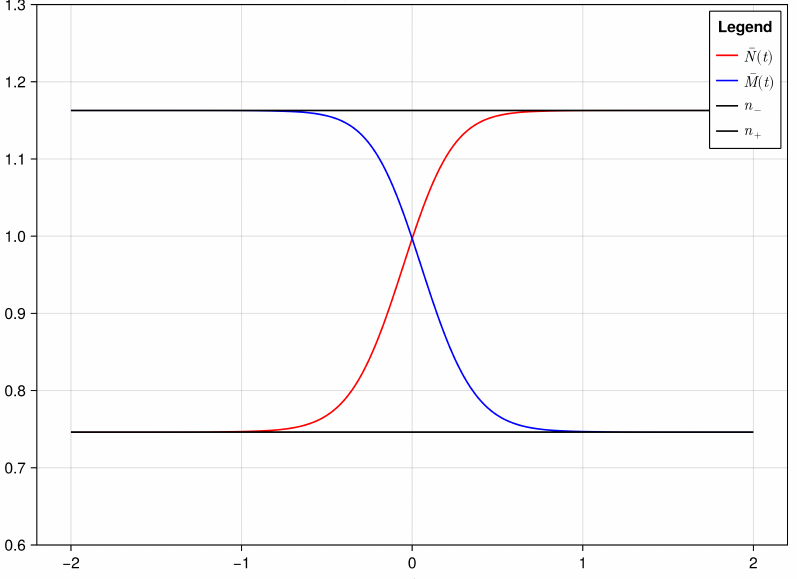}
    \vspace{-.4cm}
    \caption{Two-cycle of \eqref{IDE1}. In red, we have $\bar{N}(t)=N_0(t)$ and in blue $\bar{M}(t)=N_1(t)$. In black, horizontal lines corresponding to $n_\pm$.}\label{figure:logistic-solution}
\end{figure}

The proof of Theorem \ref{thm-logistic} uses a correspondence between 2-cycles of \eqref{IDE1} with Laplace kernel, and connecting orbits of the second-order system of ordinary differential equations
\begin{align}
\label{ODE1}\tag{ODE.1}\ddot x&=\sigma^2(x-F(y))\\
\label{ODE2}\tag{ODE.2}\ddot y&=\sigma^2(y-F(x)).
\end{align}
Essentially, $x=N_0$ and $y=N_1$, so that the limits of each of $x$ and $y$ at $\pm\infty$ correspond to suitably ordered points on the non-spatial 2-cycle of $n_{t+1}=F(n_t)$. From this equation we can also see that the shape parameter $\sigma$ of the Laplace kernel can be scaled out of the differential equations through a reparameterization of the independent variable, which is why our Theorem \ref{thm-logistic} provides connecting orbits for any $\sigma>0$. More fundamentally, the parameter $\sigma$ can be scaled out of \eqref{IDE1} completely by a change of spatial of variables $\hat x = x/\sigma$. 

To establish the existence of the connecting orbit in \eqref{ODE1}--\eqref{ODE2}, we employ the parameterization method \cite{MR1976079,MR1976080,MR2177465} for stable and unstable manifolds of equilibria, combined with a rigorous boundary value problem solver based on Chebyshev series expansions \cite{MR3148084,MR4292534} and a contraction mapping approach. Once this orbit is rigorously validated, the correspondence between the differential equations and \eqref{IDE1} directly implies Theorem \ref{thm-logistic}. Detailed explanations on translating results between the ODE system and the integrodifference equation are provided in Section \ref{sec:2-cycle-to-pvbp}.


The 2-cycle identified numerically in \cite{Bourgeois2018} appears to be (dynamically) stable. In this work, we take initial steps toward a rigorous proof of stability and present compelling numerical evidence supporting it. A central theoretical tool for this analysis is the following theorem.

\begin{thm}\label{thm-stability}
Let $N_t$ be a 2-cycle. Consider the second-iterate map $S=Q\circ Q$, where $Q:C(\R)\rightarrow C(\R)$ is the map associated to \eqref{IDE1}, namely $Q[N](x) = \int_\R K(x-y)F(N(y))dy.$ Let $K$ be the Laplace kernel, where without loss of generality $\sigma=1$. The eigenvalues of $\mathcal{M}=DS[N_0]$ are contained in the disc $\overline{D_\mu(0)}$, with
\begin{align}
\mu&=||F'\circ N_0||_\infty||F'\circ N_1||_\infty.\label{mu-bound}
\end{align}
If $\lambda$ is an eigenvalue of $\mathcal{M}$, then the associated eigenfunction $h$ satisfies
\begin{align}
\label{stability-ode-thm-1-1}\lambda\ddot h&=\lambda h - F'(N_1(x))v\\
\label{stability-ode-thm-1-2}\ddot v&=v-F'(N_0(x))h,
\end{align}
where $v(x)=\int_\R K(x-y)F'(N_0(y))h(y)dy.$ If $|\lambda|\geq1$ and $0<F'(n_+)F'(n_-)<1$, the second-order equation above is asymptotically hyperbolic as $x\rightarrow\pm\infty$.  The asymptotic constant-coefficient systems at $\pm\infty$ share the same eigenvalues and, in particular, have the same intertia. The eigenvalues are $$\xi_{j,k}=(-1)^j \sqrt{1 + (-1)^k\sqrt{F'(n_+)F'(n_-)/\lambda}},$$
for $j,k\in\{0,1\}$, and $\Re(\xi_{1,k})<0<\Re(\xi_{0,k})$ for $k=0,1$.
\end{thm}
When considering the logistic growth function with parameter $\rho=2.2$, a tight upper bound for the radius \eqref{mu-bound} can be computed using the output of the computer-assisted proof of Theorem \ref{thm-logistic}. Since all eigenvalues must lie in the disc $\overline{D_\mu(0)}$, and only eigenvalues with absolute value $\geq 1$ can be destabilizing, we have an annulus in which it suffices to count all eigenvalues. 

One can verify that $DS[N_0]\dot N_0=\dot N_0$, so $1$ is always an eigenvalue of the linearization of the two-steps operator for \eqref{IDE1}. In Section \ref{sec:stability}, we use an Evans function \cite{JonesGardnerAlexander} to numerically verify that $M$ has no other eigenvalues with absolute value $\geq 1$. This holds true for the 2-cycle we have proven for the logistic growth function, and also for a numerical candidate we have identified in the case of the Ricker growth function. We emphasize that neither of these results are proofs, although we believe they could be made rigorous. The Evans function has seen use in studying stability of pulses \cite{Kapitula}, traveling waves \cite{Sandstede2002,Coombes2004} and periodic waves \cite{Oh2010}, for instance. To our knowledge, this is the first work to apply this technique to stability of a 2-cycle in an integrodifference equation. 

\subsection{Overview of the paper}
The structure of this paper is as follows. In Section~\ref{sec:2-cycle-to-pvbp}, we show that proving 2-cycles of \eqref{IDE1} is equivalent to proving the existence of a connecting orbit in a four-dimensional ODE with a symmetry. Section~\ref{Sec3} presents the foundation of the computer-assisted proof: the series methods needed to encode the equivalent connecting orbit.  Section~\ref{sec4} provides the computer-assisted proof. We consider spectral stability of the 2-cycles in Section~\ref{sec:stability}. Concluding remarks follow in Section~\ref{sec-future}.

\subsection{Reproducibility}
The reader may find the codes needed to reproduce the result of this paper at the GitHub repository \cite{github}. 

\section{Analyzing 2-cycles through symmetry-reduced connecting orbits}\label{sec:2-cycle-to-pvbp}
In this section, we provide an equivalence between a 2-cycle of the integrodifference equation \eqref{IDE1} with the Laplace kernel, $K(u)=\frac{\sigma}{2}e^{-\sigma|u|}$, and a connecting orbit in a first-order system of ordinary differential equations. We then provide a simpler characterization, taking advantage of time-reversibility of the system, which shows that one needs only compute one ``half'' of the connecting orbit, subject to an appropriate boundary condition.

We will assume throughout that the growth function $F:\R\rightarrow\R$ is twice continuously differentiable and admits a (non-spatial) 2-cycle: a pair $\{n_+,n_-\}$ with $n_+<n_-$ and satisfying $F(n_{\pm})=n_{\mp}$. We will also assume $0 < F'(n_{-})F'(n_{+}) < 1$, which implies non-oscillatory stability of the fixed points $n_\pm$ of the second-iterate map $x\mapsto F\circ F(x)$. 

A (spatial) \emph{2-cycle of \eqref{IDE1}} is a pair of functions $N,M:\R\rightarrow\R$ with the following properties:
\begin{equation} \label{eq:conditionsC1-C3}
\begin{split}
    &C.1 \quad \lim_{x \to \pm \infty} N(x) = n_{\pm}, \\
    &C.2 \quad \lim_{x \to \pm \infty} M(x) = n_{\mp}, \\
    &C.3 \quad N = Q[M], \quad M = Q[N].
\end{split}
\end{equation}
For example, if $F(u)=(1+\rho)u-\rho u^2$ is the logistic growth function, then for $\rho >2$, the associated (non-spatial) 2-cycle is 
\[
    n_{\pm} = \frac{\rho +2 \pm \sqrt{\rho^2-4}}{2\rho}.
\]
\subsection{An equivalent ordinary differential equation}\label{sec:equivalent_ode}
As suggested in Section \ref{sec:intro}, 2-cycles correspond to connecting orbits of a second-order system of ordinary differential equations. In fact they are equivalent, as we now prove.
\begin{thm} \label{THM:ODE_equivalence}
    Suppose $N,M : \R \rightarrow \R$ are bounded, continuous functions. Then $N,M$ satisfying
    \[
        N = Q[M], \quad
        M = Q[N]
    \]
    implies $M, N$ are twice continuously differentiable and satisfy
    \begin{equation} \label{EQ:ODE}
        \begin{split}
            &\ddot{N} = \sigma^2(N - F(M)), \\
            &\ddot{M} = \sigma^2(M - F(N)),
        \end{split}
    \end{equation}
    where $\ddot{N}(t) = \frac{d^2}{dt^2}N(t)$ and similarly for $\ddot{M}(t)$.   The converse is also true.
\end{thm}
\begin{proof}
    Suppose $N,M : \R \rightarrow \R$ are bounded, continuous functions. If $N=Q[M]$ and $M=Q[N]$, then 
    \[
    {\small
    \begin{split}
        \frac{d}{dx}N(x)
        &= \frac{d}{dx} \int_{-\infty}^{\infty} K(x-y)F(M(y)) dy \\
        &= \frac{d}{dx} \int_{-\infty}^{x} \frac{\sigma}{2}e^{-\sigma(x-y)} F(M(y)) dy + \frac{d}{dx} \int_{x}^{\infty} \frac{\sigma}{2}e^{\sigma(x-y)} F(M(y)) dy \\
        &= \frac{\sigma}{2} F(M(x)) + \int_{-\infty}^{x} \frac{\partial}{\partial x} \frac{\sigma}{2} e^{-\sigma(x-y)} F(M(y)) dy 
         - \frac{\sigma}{2} F(M(x)) + \int_{x}^{\infty} \frac{\partial}{\partial x} \frac{\sigma}{2}e^{\sigma(x-y)} F(M(y)) dy\\
        &= -\sigma\int_{-\infty}^{x}\frac{\sigma}{2}e^{-\sigma(x-y)} F(M(y)) dy + \sigma\int_{x}^{\infty} \frac{\sigma}{2}e^{\sigma(x-y)} F(M(y)) dy,
    \end{split}
    }
    \]
    where the third equality comes from Leibniz integral formula. Also,
    \[
    {\small
    \begin{split}
        \frac{d^2}{dx^2}N(x)
        &= \frac{d}{dx} \left[ -\sigma\int_{-\infty}^{x}\frac{\sigma}{2}e^{-\sigma(x-y)} F(M(y)) dy+ \sigma\int_{x}^{\infty} \frac{\sigma}{2}e^{\sigma(x-y)} F(M(y)) dy\right] \\
        &= -\sigma \left[ \frac{\sigma}{2} F(M(x)) + \int_{-\infty}^{x} \frac{\partial}{\partial x} \frac{\sigma}{2} e^{-\sigma(x-y)} F(M(y)) dy \right] 
         +\sigma \left[ -\frac{\sigma}{2} F(M(x)) + \int_{x}^{\infty} \frac{\partial}{\partial x} \frac{\sigma}{2} e^{\sigma(x-y)} F(M(y)) dy \right] \\
        &= \sigma^2 \left[ \int_{-\infty}^{x} \frac{\sigma}{2} e^{\sigma(x-y)} F(M(y)) dy + \int_{x}^{\infty} \frac{\sigma}{2} e^{\sigma(x-y)} F(M(y)) dy -F(M(x)) \right] \\
        &= \sigma^2 \left[ \int_{-\infty}^{x} \frac{\sigma}{2} e^{\sigma|x-y|} F(M(y)) dy + \int_{x}^{\infty} \frac{\sigma}{2} e^{-\sigma|x-y|} F(M(y)) dy -F(M(x)) \right] \\
        &= \sigma^2 \left[ \int_{-\infty}^{\infty} \frac{\sigma}{2} e^{\sigma|x-y|} F(M(y)) dy -F(M(x)) \right] \\
        &= \sigma^2 \left[ N(x) - F(M(x))\right],
    \end{split}
    }
    \]
    where the third equality is obtained by the Leibniz integral formula and the last equality comes from the assumption. By symmetry the same can be shown for $M$. 
    
    Now, suppose $M, N$ are twice continuously differentiable and satisfy \eqref{EQ:ODE}. Then $F(M) = N - \frac{\ddot{N}}{\sigma^2}$,
    and so
    \[
    {\small
    \begin{split}
        Q[M](x)
        &= \int_{-\infty}^{\infty} \frac{\sigma}{2} e^{-\sigma|x-y|} F(M(y)) dy \\
        &= \int_{-\infty}^{x} \frac{\sigma}{2} e^{-\sigma(x-y)} F(M(y)) dy + \int_{x}^{\infty} \frac{\sigma}{2} e^{\sigma(x-y)} F(M(y)) dy \\
        &= \int_{-\infty}^{x} \frac{\sigma}{2} e^{-\sigma(x-y)} \left( N(y) - \frac{\ddot{N}(y)}{\sigma^2} \right) dy + \int_{x}^{\infty} \frac{\sigma}{2} e^{\sigma(x-y)} \left( N(y) - \frac{\ddot{N}(y)}{\sigma^2} \right) dy \\
        &= -\frac{1}{2\sigma} \int_{-\infty}^{x} e^{-\sigma(x-y)} \frac{\ddot{N}(y)}{\sigma^2} dy + \frac{\sigma}{2} \int_{-\infty}^{x} e^{-\sigma(x-y)} N(y) dy\\
        &\quad -\frac{1}{2\sigma} \int_{x}^{\infty} e^{\sigma(x-y)} \frac{\ddot{N}(y)}{\sigma^2}dy + \frac{\sigma}{2} \int_{x}^{\infty} e^{\sigma(x-y)} N(y) dy \\
        &= -\frac{1}{2\sigma} \left[ \dot{N}(x) - \sigma \int_{-\infty}^{x} e^{-\sigma(x-y)}\dot{N}(y) dy \right] + \frac{\sigma}{2} \left[ \frac{1}{\sigma} N(x) - \frac{1}{\sigma} \int_{-\infty}^{x} e^{-\sigma(x-y)}\dot{N}(y) dy\right] \\
        &\quad -\frac{1}{2\sigma} \left[ -\dot{N}(x) + \sigma \int_{x}^{\infty} e^{\sigma(x-y)} \dot{N}(y) dy\right] + \frac{\sigma}{2} \left[ \frac{1}{\sigma}N(x) + \frac{1}{\sigma} \int_{x}^{\infty} e^{\sigma(x-y)} \dot{N}(y) dy \right] \\
        &= -\frac{1}{2\sigma}\dot{N}(x) +\frac{1}{2} N(x) + \frac{1}{2\sigma}\dot{N}(x)+\frac{1}{2} N(x)\\
        &= N(x),
    \end{split}
    }
    \]
    where the fourth equality follows by integration by parts. By symmetry the same can be shown for $Q[N]$.
\end{proof}

We transform the 2-dimensional second order ordinary differential equation (ODE) \eqref{EQ:ODE} into a 4-dimensional first order ODE in the classical way.
That is, we define $u = \dot{N}$ and $v = \dot{M}$. For notation purposes we denote $(N,u,M,v) = x = (x_1, x_2, x_3, x_4)$
and obtain the following system of equations for $x$:
\begin{equation} \label{EQ:4DODE}
    \dot{x} = f(x) = 
    \begin{pmatrix}
         x_2  \\
         \sigma^2(x_1 - F(x_3))  \\
         x_4  \\
         \sigma^2(x_3 - F(x_1))
    \end{pmatrix}.
\end{equation}

\begin{thm} \label{THM:problem_equivalence}
Suppose $(N,u,M,v)$ is a solution of \eqref{EQ:4DODE} and satisfies 
\begin{equation} 
\begin{split}
    &C.1' \quad \lim_{t \to \pm \infty} u(t), v(t) = 0, \\
    &C.2' \quad \lim_{t \to \pm\infty} N(t) = n_\pm, \\
    &C.3' \quad \lim_{t \to \pm\infty} M(t) = n_\mp.
\end{split}
\end{equation}
Then $N,M$ satisfy C.1-C.3, as defined in \eqref{eq:conditionsC1-C3}.
\end{thm}

\begin{proof}
    $C.1'$ and $C.2'$ imply $C.1$ while $C.1'$ and $C.3'$ imply $C.2$. Then  since $(N,u,M,v)$ is a solution of \eqref{EQ:4DODE}, we have $(N,M)$ a solution of \eqref{EQ:ODE} and hence a 2-cycle of \eqref{IDE1} by Theorem \ref{THM:ODE_equivalence}. Therefore, $M = G[N]$ and $N = G[M]$ so $C.3$ is satisfied.
\end{proof}

By combining Theorems \ref{THM:ODE_equivalence} and \ref{THM:problem_equivalence}, the task of identifying a 2-cycle in \eqref{IDE1} reduces to finding a particular solution of \eqref{EQ:4DODE}.

\begin{dfn}
    In the system $\dot{x} = f(x)$, a connecting orbit between two fixed points $\tilde{x}, \tilde{y}$ is a solution $x(t)$ that satisfies $\lim_{t \to \infty} x(t) = \tilde{x}$ and $\lim_{t \to -\infty} x(t) = \tilde{y}.$
\end{dfn}
One can verify that $\tilde{x}^{(-)}= (n_-, 0, n_+, 0)$ and $\tilde{x}^{(+)} = (n_+, 0, n_-, 0)$ are fixed points of \eqref{EQ:4DODE}. Using this terminology, we can reformulate the problem of identifying a 2-cycle as a search for a connecting orbit between $\tilde{x}^{(-)}$ and $\tilde{x}^{(+)}$ in \eqref{EQ:4DODE}.

\subsection{Invariant manifolds, reversors, and a reduced connecting orbit problem}
Connecting orbits can be characterized by intersections between the stable and unstable manifolds of two different equilibrium points. To formalize this notion, given an equilibrium solution $\tx \in \{ \tx^{(-)},\tx^{(+)} \}$ to \eqref{EQ:4DODE}, denote by $W^s(\tx)$ and $W^u(\tx)$ their associated local stable and unstable manifolds.  
%
%
%
Using these objects, the problem of finding a connecting orbit in \eqref{EQ:4DODE} between $\tilde{x}^{(\pm)}$ is equivalent to having $x(t) = (N(t), u(t), M(t), v(t))$ solution of \eqref{EQ:4DODE} such that, for some $L >0$,
\begin{equation}\label{EQ:manifold-connection-full}
\begin{cases}
    x(-2L) \in W^u(\tilde{x}^{(-)}) \\
    x(2L) \in W^s(\tilde{x}^{(+)}). \\
\end{cases}
\end{equation}
Note that the second-order system \eqref{EQ:ODE} is invariant under time reversal, i.e.\ under the change of variables $t \mapsto -t$. Consequently, it suffices to consider only half of the connecting orbit. This observation allows us to simplify the description of the connection in \eqref{EQ:manifold-connection-full}, since only one of the invariant manifolds is required. We now state this more precisely.
\begin{dfn}
{\em
An \emph{involution} $R: X \to X$ is a map that satisfies $R = R^{-1}$. The vector field $\dot{x} = f(x)$ is said to be \emph{time-reversible} if there exists an involution $R$ such that the flow is invariant under the change of variables
\[
\begin{cases}
    t \mapsto -t,\\
    x \mapsto R(x).
\end{cases}
\]
More precisely, we ask that
\[
\frac{d}{dt} \left[ R(x(-t)) \right] = f(R(x(-t))),
\]
or that $y(t) = R(x(-t))$ be a solution to $\dot{y} = f(y)$. $R$ is called the \emph{reversor} of the system.
}
\end{dfn}

The following result is standard and its proof is omitted. 

\begin{thm} \label{THM:manifold_equivalence}
    Let $\dot{x} = f(x)$ be time-reversible with linear reversor $R$ and fixed point $\tilde{x}$. Then $R\tilde{x}$ is also a fixed point and $W^u(R\tilde{x}) = RW^s(\tilde{x})$.   
\end{thm}

\begin{rem} \label{eq:reversible_manifold}
The system \eqref{EQ:4DODE} is time-reversible with reversor $R(x_1,x_2,x_3,x_4) = (x_3, -x_4, x_1, -x_2)$. Moreover, since $x^{(\pm)}$ are reflections of each other under $R$, Theorem~\ref{THM:manifold_equivalence} implies that $W^u(x^{(-)}) = RW^s(x^{(+)})$.
\end{rem}
\noindent Using the previous remark, the reduced problem consists of searching for a function $x:\R \to \R^4$ solving the projected boundary value problem (BVP)
\begin{equation} \label{eq:reduced_problem}
\begin{cases}
\dot{x} = f(x) \\
x(0) = Rx(0) \\
x(2L) \in W^s(x^{(+)})
\end{cases}
\end{equation}
for some $L>0$ and where $f$ is given in \eqref{EQ:4DODE}. By construction, finding a solution $x:[0,2L] \to \R^4$ to the reduced problem \eqref{eq:reduced_problem} can be used to construct a continuous extension $x:[-2L,2L] \to \R^4$ by setting $x(t) \bydef Rx(-t)$ for $t \in [-2L,0]$. Moreover, by Remark~\ref{eq:reversible_manifold}, $x(-2L) = Rx(2L) \in RW^s(x^{(+)}) = W^u(x^{(-)})$, and therefore the solution can further be extended to $x:\R \to \R^4$ with the property that $\lim_{t \to \pm \infty} x(t) = x^{(\pm)}$. Proving constructively the existence of the connecting orbit therefore boils down to solve the reduced problem \eqref{eq:reduced_problem}, which requires first studying the stable manifold of $x^{(+)}$. To this end, we present explicitly the eigenpairs of $Df(x^{(+)})$.
\begin{lem} \label{LEM: eigenpairs}
    Let $f$ be defined as in \eqref{EQ:4DODE}. Then
    \[
        Df(\tilde{x}^{(+)}) =
    \begin{pmatrix}
         0 & 1 & 0 & 0 \\
         \sigma^2 & 0 & -\sigma^2 F'(n_{-}) & 0 \\
         0 & 0 & 0 & 1 \\
         -\sigma^2 F'(n_{+}) & 0 & \sigma^2 & 0 \\
    \end{pmatrix}.
    \]
Moreover, $Df(\tilde{x}^{(+)})$ has the following eigenpairs
 \[
 {\tiny
    \begin{split}
            \lambda_1 = - \sqrt{\sigma^2 \left( 1 + \sqrt{F'(n_{-})F'(n_{+})} \right)} &,
        \quad \xi_1 = \left( \frac{\sqrt{F'(n_+)}}{\lambda_1 \sqrt{F'(n_-)}}, \frac{\sqrt{F'(n_+)}}{\sqrt{F'(n_-)}}, \frac{1}{\lambda_1}, 1\right). \\
        \lambda_2 = - \sqrt{\sigma^2 \left( 1 - \sqrt{F'(n_{-})F'(n_{+})} \right)} &,
        \quad \xi_2 = \left( -\frac{\sqrt{F'(n_+)}}{\lambda_2 \sqrt{F'(n_-)}}, -\frac{\sqrt{F'(n_+)}}{\sqrt{F'(n_-)}}, \frac{1}{\lambda_2}, 1\right) \\
        \lambda_3 = \sqrt{\sigma^2 \left( 1 + \sqrt{F'(n_{-})F'(n_{+})} \right)} &,
        \quad \xi_3 = \left( \frac{\sqrt{F'(n_+)}}{\lambda_3\sqrt{F'(n_-)}}, \frac{\sqrt{F'(n_+)}}{\sqrt{F'(n_-)}}, \frac{1}{\lambda_3}, 1\right). \\
        \lambda_4 = \sqrt{\sigma^2 \left( 1 - \sqrt{F'(n_{-})F'(n_{+})} \right)} &,
        \quad \xi_4 = \left( -\frac{\sqrt{F'(n_+)}}{\lambda_4\sqrt{F'(n_-)}}, -\frac{\sqrt{F'(n_+)}}{\sqrt{F'(n_-)}}, \frac{1}{\lambda_4}, 1\right). 
\end{split}    
}
\]    
\end{lem}
Since $0 < F'(n_{-})F'(n_{+}) < 1$, all eigenvalues are real, with two positive and two negative. 
It follows that each equilibrium admits both a two-dimensional local stable manifold and a two-dimensional local unstable manifold. For instance, in the case of the logistic map one checks that the condition $0<F'(n_-)F'(n_+)<1$ holds whenever $2 < \rho < \sqrt{5}$.

The full connecting orbit problem \eqref{EQ:manifold-connection-full} is invariant under time translation: if $t\mapsto x(t)$ is a solution, then so is $t\mapsto x(t+s)$ for any $s\in\R$.  By contrast, the reduced problem \eqref{eq:reduced_problem} removes this translational freedom for the class of solutions under consideration. This property, made precise in the next result, plays a key role in formulating a fixed point theorem where local isolation is essential.
\begin{lem}
If $x$ solves \eqref{eq:reduced_problem} for which $t\mapsto x(t+s)$ is also a solution for any $s\in\R$,  then $x$ is constant.
\end{lem}
\begin{proof}
Suppose $x$ is a solution that is translation-invariant. Denote $x_s(t)=x(s+t)$. Then we have $x(0)=Rx(0)$ and $x(s)=x_s(0)=Rx_s(0)=Rx(s)$. Since $R$ is linear,
$$\dot x(0)=\lim_{s\rightarrow 0}\frac{1}{s}(x(s)-x(0))=\lim_{s\rightarrow 0}\frac{1}{s}(Rx(s)-Rx(0))=R\lim_{s\rightarrow 0}\frac{1}{s}(x(s)-x(0))=R\dot x(0).$$
Using the above with the properties of the reversor and the vector field itself,
\begin{align*}
\dot x(0)&=R\dot x(0)=Rf(x(0))=Rf(Rx(0))=R\frac{d}{dt}\left[R(x(-t))\right]\big|_{t=0}=-R^2\dot x(0)=-\dot x(0).
\end{align*}
But this implies $\dot x(0)=0$. Since the vector field $f$ is smooth, it follows that $x$ is a constant solution.
\end{proof}


\section{Series methods for the symmetry-reduced 2-cycle problem}\label{Sec3}

The reversor-reduced problem \eqref{eq:reduced_problem} amounts to finding a solution of \eqref{EQ:4DODE} that lies in the stable manifold of $\tilde x^{(+)}$ at $t=2L$ (that is, $x(2L)\in W^s(\tilde x^{(+)})$) and is fixed by the reversor at $t=0$ (that is, $x(0)=R x(0)$). To this end, we represent a solution of \eqref{EQ:4DODE} on $[0,2L]$ by a Chebyshev series, and impose conditions on its coefficients so that the boundary requirements at $t=0$ and $t=2L$ are met. Since the stable manifold is two-dimensional (with real eigenvalues), the boundary equations naturally involve two real parameters describing directions along $W^s(\tilde x^{(+)})$. Instead of constructing the entire manifold $W^s(\tilde x^{(+)})$, we compute a parameterization of a local portion $W^s_{\text{loc}}(\tilde x^{(+)})$ by combining the Parameterization Method (see Section~\ref{sec:parameterization_method}) with a Taylor expansion.

In this section, we review some machinery from Banach algebras, which will be needed later when we move to computer-assisted proofs, and whose algebraic properties we will require immediately. We then review Taylor and Chebyshev series before moving on to the parameterization method for stable and unstable manifolds, and Chebyshev series methods for the solution of boundary-value problems. We apply both of the latter techniques to our problem, concluding with a recipe to numerically solve the symmetry-reduced 2-cycle problem.

\subsection{Banach algebras and functional analytic background}
Introduce the Banach space
\[
\cT \bydef \left\{ (a_\alpha)_{\alpha \in \N^2}, a_\alpha \in \R: \norm{a}_{\cT} \bydef \sum_{|\alpha|_1 \ge 0} |a_\alpha| < \infty\right\},
\]
the space of two-index sequences of Taylor coefficients that have finite $\ell^1$ norm, where given $\alpha=(\alpha_1,\alpha_2) \in \N^2$, $|\alpha|_1\bydef \alpha_1+\alpha_2$. Moreover,  introduce also
\[
\cC_\nu \bydef \left\{ (u_k)_{k \in \N}, u_k \in \R: \norm{\cdot}_{\cC_\nu} \bydef |u_0| + 2\sum_{k > 0} |u_k| \nu^k < \infty\right\},
\]
where $\nu \ge 1$, the space of sequences of Chebyshev coefficients that have finite weighted $\ell^1$ norm. Note that the norm in $\cC_\nu$ is a weighted $\ell^1$ norm, with weights $\omega \bydef (\omega_k)_{k \ge 0}$ with $\omega_0=1$ and $\omega_k = 2\nu^k$ for $k > 0$, that is $\norm{\cdot}_{\cC_\nu} = \sum_{k \ge 0} |u_k| \omega_k$. Note that $\cT$ and $\cC_\nu$ are commutative Banach algebras under a convolution product. Recall that a \emph{commutative Banach algebra} is a Banach space $X$ equipped with a commutative vector multiplication $*$ such that $||x*y||\leq||x||\cdot||y||$ for all $x,y\in X$. The proofs of the following lemmas may be found for instance at \cite{MR3833658,MR3792791}.
\begin{lem} \label{lem:Banach_algebras}
    $(\cT, \norm{\cdot}_\cT, *_\cT)$ is a commutative Banach algebra, where
    \[
    (x*_\cT y)_\alpha = \sum_{\beta_1=0}^{\alpha_1} \sum_{\beta_2=0}^{\alpha_2} a_\beta b_{\alpha-\beta}.
    \]
The same is true for $(\cC_\nu, \norm{\cdot}_{\cC_\nu}, *_{\cC_\nu})$, where
\begin{equation} \label{eq:Chebyshev_product}
(x*_{\cC_\nu}y)_n = \sum_{k \in \Z} a_{|k|}b_{|n-k|}.
\end{equation}
\end{lem}
From here onward, we will drop the subscript on the multiplication operator $*$ when the context is obvious. We will occasionally make use of the following lemmas, which characterize the support of sequences in these algebras. 
\begin{lem} \label{THM: taylor max degree}
    If $x \in \cT$ satisfies $x_\alpha = 0$ for $|\alpha|_1 < N_1$ and $|\alpha|_1 > N_2$ and $y \in \cT$ satisfies $y_\alpha = 0$ for  $|\alpha|_1 < M_1$ and $|\alpha|_1 > M_2$, then $(x*y)_\alpha = 0$ for $|\alpha|_1 < N_1+M_1$ and $|\alpha|_1 > N_2+M_2$.
\end{lem}
\begin{lem} \label{THM: chebyshev max degree}
    If $u \in \cC_\nu$ satisfies $u_k = 0$ for $k > N$ and $v \in \cC_\nu$ satisfies $v_k = 0$ for $k > M$, then $(x*y)_k = 0$ for $k > N+M$.
\end{lem}

Let $B(X,Y)$ denote the space of bounded linear maps between two normed spaces $X$ and $Y$, and let $B(X)=B(X,X)$. The following lemmas provide expressions for the norms of several linear maps involving the spaces $\cT$ and $\cC_\nu$.
\begin{lem} \label{THM: taylor taylor norm}
    Let $\Gamma \in B(\cT)$ acting on $h \in \cT$ as
    $(\Gamma h)_\alpha \bydef \sum_{|\beta|_1 \ge 0} \Gamma_{\alpha, \beta} h_\beta$ for $|\alpha|_1 \ge 0$. Then
    \begin{equation*}
        \norm{\Gamma}_{B(\cT)} = \sup_{\beta \ge 0} \sum_{\alpha \ge 0} |\Gamma_{\alpha, \beta}|.
    \end{equation*}
\end{lem}

\begin{lem} \label{THM: chebyshev R norm}
    Let $\Gamma \in B(\cC_\nu, \R)$ acting on $h \in \cC_\nu$ as
    $\Gamma h \bydef \sum_{k \ge 0} \Gamma_k h_k$. Then
    \[
    \norm{\Gamma}_{B(\cC_\nu, \R)} = \sup_{k \ge 0} \frac{|\Gamma_k|}{\omega_k}.
    \]
\end{lem}


\begin{lem} \label{THM: R chebyshev norm}
    Let $\Gamma \in B(\R, \cC_\nu)$ acting on $h \in \R$ as
    $(\Gamma h)_k \bydef h\Gamma_k$. Then
    \[
    \norm{\Gamma}_{B(\R, \cC_\nu)} = \norm{(\Gamma_k)_{k \ge 0}}_{\cC_\nu}.
    \]
\end{lem}

\begin{lem} \label{THM: chebyshev chebyshev norm}
    Let $\Gamma \in B(\cC_\nu)$ acting on $h \in \cC_\nu$ as
    $(\Gamma h)_k \bydef \sum_{\ell \ge 0} \Gamma_{k, \ell} h_\ell$.
    Then
    \[
    \norm{\Gamma}_{B(\cC_\nu)} = \sup_{\ell \ge 0} \frac{1}{\omega_\ell} \sum_{k \ge 0}  |\Gamma_{k, \ell}|\omega_k.
    \]
\end{lem}

We conclude this subsection by introducing projection operators, and provide a useful bound on the coefficients of a convolution product involving a Chebyshev sequence of finite support.
\begin{dfn} \label{defn:projections}
    Let $N \in \N^+$ and define $\pi^{(N)} \in B(\cT)$ by its action on $a \in \cT$
    \[
    (\pi^{(N)}a)_\alpha = 
    \begin{cases}
        a_\alpha & 0 \le |\alpha|_\infty \le N, \\
        0 & |\alpha|_\infty > N.
    \end{cases}
    \]
    Similarly, define $\pi^{(N)} \in B(\cC_\nu)$ by its action on $u \in \cC_\nu$
    \[
    (\pi^{(N)}u)_k = 
    \begin{cases}
        u_k & 0 \le k \le N, \\
        0 & k > N.
    \end{cases}
    \]
Now, if $a = (a_1,\dots,a_n)$ is an element of $X^n$ where $X=\cT$ or $X=\cC_\nu$, define the projections $\pi^{(N)}:X^n \to X^n$ and $\pi_i:X^n \to X$ by
\[
\pi^{(N)} a \bydef \left( \pi^{(N)} a_1,\dots,  \pi^{(N)} a_n \right) 
\qquad \text{and} \qquad 
\pi_i a \bydef a_i.
\]
\end{dfn}

The following result (e.g. see \cite{MR3454370}), whose elementary proof is omitted, will be useful when computing one of the bound for the computer-assisted proof argument for the existence of the connecting orbit (see the $Z_1$ bound in Section~\ref{sec:bvp_explicit_bounds}).

\begin{lem} \label{THM: dual estimates}
    Let $N \in \N^+, \nu \ge 1$ and $u,h \in \cC_\nu$. Further suppose that $u = u^{(N)}$. Then for $k = 0, \dots, N+1$ we have
    \[
    \left| (u*h^{(\infty)})_k \right| \le \Psi_k(u) \norm{h}_{\cC_\nu},
    \]
    where
    \[
        \Psi_k(u) = \max_{j = N+1, \dots, k+N} \frac{|u_{|k+j|} + u_{|k-j|}|}{2\nu^j}.
    \]
\end{lem}

The following result is a direct application of the formula presented in Lemma~\ref{THM: chebyshev chebyshev norm}, and will prove useful when deriving bounds for the proof of the connecting orbit using Chebyshev series. 
\begin{lem} \label{lem: upsilon operator norm}
Define the operator $\Upsilon$ that acts on Chebyshev sequences as follows
\begin{equation} \label{eq:Upsilon}
(\Upsilon h)_k = 
\begin{cases}
    0, &k = 0\\
    h_{k-1} - h_{k+1}, &k \ge 1.
\end{cases}
\end{equation}
Let $\nu \ge 1$ and $\omega \bydef (\omega_k)_{k \ge 0}$ given component-wise by $\omega_0=1$ and $\omega_k = 2\nu^k$ for $k > 0$. Let $\Upsilon \in B(\cC_\nu)$ be defined as above. Then
    \[
    \norm{\Upsilon}_{B(\cC_\nu)} = 2\nu.
    \]
\end{lem}

\subsection{Multi-index Taylor series}
Taylor series are an essential tool of single variable calculus and approximation theory. We will use them to parameterize our stable manifold at $\tilde x^{(+)}$. In this section, we briefly review multi-index Taylor series and provide a connection with the convolution product $*_\cT$. 
\begin{dfn}
{\em 
Let $m \in \N^+$. An {\em index $m$ Taylor series} is a series of the form
\[
        \sum_{\alpha_1 \ge 0} \dots \sum_{\alpha_m \ge 0} a_{(\alpha_1, \dots,\alpha_m)} \theta_1^{\alpha_1} \dots \theta_m^{\alpha_m} = \sum_{|\alpha|_1 \geq 0} a_\alpha \theta^\alpha
\]
where $a_\alpha = a_{(\alpha_1, \dots, \alpha_m)} \in \R^n$ for some $n \in \N^+$, and in expression on the right, we employ the multi-index notation. That is, $\alpha = (\alpha_1, \dots, \alpha_m)$,  $\theta = (\theta_1, \dots, \theta_m)$, and 
\[
    |\alpha|_1 \bydef \sum_{k=1}^{m} |\alpha_k|, \quad 
    |\alpha|_\infty \bydef \max_{k=1, \dots, m} |\alpha_k|, \quad 
    \theta^\alpha \bydef \prod_{k=1}^{m} \theta_k^{\alpha_k}.
\]
}
\end{dfn}
\begin{prop}\label{prop-Taylor-multiplication}
    Let $\sum_{|\alpha|_1 \ge 0} a_\alpha \theta^\alpha, \sum_{|\alpha|_1 \ge 0} b_\alpha \theta^\alpha$ be index $m$ Taylor series. Then
    \[
    \left( \sum_{|\alpha|_1 \ge 0} a_\alpha \theta^\alpha \right)\left( \sum_{|\alpha|_1 \ge 0} b_\alpha \theta^\alpha \right) = \left( \sum_{|\alpha|_1 \ge 0} (a*_{\cT}b)_\alpha \theta^\alpha \right).
    \]
\end{prop}

\subsection{Chebyshev series}\label{sec:cheb}
Chebyshev series are well known for their use in interpolation. Their unique properties make them a perfect option to solve the projected boundary problem between the stable manifold at $\tilde x^{(+)}$ and the fixed point of the reversor $R$.
\begin{dfn}
{\em 
The {\em Chebyshev polynomials} are a sequence of polynomials $T_k:[-1,1] \to [-1,1]$ for $k \ge 0$ that satisfy the following recurrence relation $T_0(t) = 1$, $T_1(t) = t$ and $T_k(t) = 2tT_{k-1}(t) - T_{k-2}(t)$ for $k \ge 2$. An equivalent definition constructs the polynomials as $T_k(t) = \cos(k \cos^{-1}(t))$, or $T_k(\cos(\theta)) = \cos(k\theta)$, where $\theta \bydef \cos^{-1}(t)$.
 }
\end{dfn}
\begin{prop}\label{prop:cheb}
    The Chebyshev polynomials $T_k(t)$ enjoy the following properties: 1) $T_k(-1) = (-1)^k$ and $T_k(1) = 1$ for all $k\ge 0$; 2) $2T_k(t) =\frac{d}{dt}\left( \frac{T_{k+1}(t)}{k+1} - \frac{T_{k-1}(t)}{k-1} \right)$ for all $k \ge 2$, and 3) given the Chebyshev series expansions $a(t) = a_0 + 2\sum_{k = 1}^{\infty} a_k T_k(t)$ and $b(t) = b_0 + 2\sum_{k = 1}^{\infty} b_k T_k(t)$, the product $a(t)b(t)$ admits the Chebyshev series expansion
    \[
    a(t)b(t) = (a*_{\cC_\nu}b)_0 + 2\sum_{k = 1}^{\infty} (a*_{\cC_\nu}b)_k T_k(t).
    \]
\end{prop}

\subsection{Local stable manifolds via the parameterization method and Taylor series} \label{sec:parameterization_method}

In this section, we formulate a zero-finding problem of the form $G(a)=0$, with $G$ defined in \eqref{EQ:manifold_map}. Any solution $a$ to this equation provides the Taylor coefficients of a parameterization of the local stable manifold of $\tilde{x}^{(+)}$. Our construction is based on the \emph{Parameterization Method}, first introduced in \cite{MR1976079,MR1976080,MR2177465} and subsequently adapted for rigorous a posteriori analysis in \cite{MR2821596,MR3518609}. For the remainder of this section, we work in the setting $\dot{x}=f(x)$ from \eqref{EQ:4DODE}, where $F$ is either the logistic growth function $F(u)=(1+\rho )u-\rho u^2$ or Ricker’s growth function $F(u)=u\exp(\rho (1-u))$. In both cases, the resulting vector field $f$ in \eqref{EQ:4DODE} is analytic, in fact, entire. We denote the relevant equilibrium simply by $\tilde{x}=\tilde{x}^{(+)}$, and our main objective here is to compute its associated local stable manifold. Recall from Lemma~\ref{LEM: eigenpairs} that $Df(\tilde{x})$ admits precisely two negative eigenvalues $\lambda_1,\lambda_2$ and two positive real eigenvalues $\lambda_3,\lambda_4$. Let $\Lambda$ denote the $2 \times 2$ diagonal matrix of stable (negative) eigenvalues
\[
\Lambda = 
    \begin{pmatrix}
         \lambda_1 &0 \\
         0  & \lambda_2
    \end{pmatrix}.
\]
Recalling the stable eigenvectors $\xi_1$ and $\xi_2$ from Lemma~\ref{LEM: eigenpairs}, we define $V = (\xi_1 | \xi_2) $, that is, the $4 \times 2$ matrix whose columns are precisely the two stable eigenvectors. In the basis of stable eigenvectors, the linearization of \eqref{EQ:4DODE} restricted to the stable subspace takes the form
\begin{equation*}
    \dot{y} = \Lambda y, \quad y \in \R^2
\end{equation*}
with associated flow $\psi(t,x) = e^{\Lambda t}x$. Our objective is to construct a map $P: B^2 \rightarrow \R^4$, where $B^2 \subset \R^2$ is a neighborhood of the origin, such that $P(B^2)$ represents a local stable manifold of $\tilde{x}$. We require that $P$ provides a conjugacy between the flow $\varphi$ of \eqref{EQ:4DODE} restricted to its stable manifold and the linear flow $\psi$. Accordingly, $P$ must satisfy
\[
    P(0) = \tilde{x}, \quad
    DP(0) = V,
\]
together with the conjugacy condition
\begin{equation} \label{EQ:conjugacy}
    \varphi(t,P(\theta)) = P(\psi(t,\theta)) = P(e^{\Lambda t}\theta), \quad \forall \theta \in B^2.
\end{equation}
The first condition centers the parameterization at the equilibrium $\tilde{x}$, the second guarantees tangency to the linearized system at $\tilde{x}$, and the third in \eqref{EQ:conjugacy} encodes the conjugacy between the nonlinear and linear flows. Since the flow $\varphi$ is not explicitly available, we appeal to a standard result (e.g. see \cite{MR1976079,MR1976080,MR2177465}) which reformulates \eqref{EQ:conjugacy} as an equivalent functional equation involving only the vector field $f$ and the matrix $\Lambda$. Solving this functional equation yields the desired parameterization. 
\begin{thm} \label{THM:Parameterization}
Let $P: B^2 \rightarrow \R^4$ be a smooth function such that
\begin{equation} \label{EQ: constraint_1}
P(0) = \tilde{x} \quad \text{and} \quad DP(0) = V.
\end{equation}
Then P satisfies (\ref{EQ:conjugacy}) if and only if P is a solution to the following partial differential equation
    \begin{equation} \label{EQ: PDE_compact}
        DP(\theta)\Lambda\theta = \lambda_1\theta_1\frac{\partial}{\partial\theta_1}P(\theta) + \lambda_2\theta_2\frac{\partial}{\partial\theta_2}P(\theta) = f(P(\theta)), \quad \theta \in B^2.
    \end{equation}    
\end{thm}
Since $f$ is analytic, the solution of \eqref{EQ: PDE_compact} is analytic in a neighborhood of $0 \in \R^2$ and hence we may seek $P:B^2 \to \R^4$ expressed as a power series of the form
%
\begin{equation} \label{eq:Taylor_series_manifold}
P(\theta) = \sum_{|\alpha|_1 \geq 0} a_\alpha \theta^\alpha, \quad a_\alpha = ((a_1)_\alpha,(a_2)_\alpha,(a_3)_\alpha,(a_4)_\alpha)^T \in \R^4,
\end{equation}
on its domain of convergence, where $\alpha = (\alpha_1,\alpha_2) \in \N^2$, $|\alpha|_1 =\alpha_1+\alpha_2$, $\theta = (\theta_1,\theta_2) \in B^2$ and $\theta^\alpha = \theta_1^{\alpha_1} \theta_2^{\alpha_2}$. Since $P(\theta)$ must satisfy constraints \eqref{EQ: constraint_1}, this yields the first Taylor coefficients
\begin{equation*}
    a_{(0,0)} = \tilde{x}, \quad a_{(1,0)} = \xi_1, \quad a_{(0,1)} = \xi_2.
\end{equation*}
To determine the coefficients $a_\alpha$ for $|\alpha|_1 \geq 2$, note that
\[
DP(\theta)\Lambda\theta = \sum_{|\alpha|_1 \geq 0} \alpha_1 \lambda_1 a_{\alpha} \theta ^\alpha + \sum_{|\alpha|_1 \geq 0} \alpha_2 \lambda_2 a_{\alpha} \theta ^\alpha 
        =\sum_{|\alpha|_1 \geq 0} (\alpha \cdot \lambda) a_\alpha \theta^\alpha,
\]
where $\alpha \cdot \lambda \bydef \alpha_1\lambda_1 + \alpha_2\lambda_2$, and the partial derivatives of $P(\theta)$ are obtained by differentiating under the sum.

Recall that in this work we consider $F$ to be either the logistic nonlinearity $F(u)=(1+\rho )u-\rho u^2$ or the Ricker nonlinearity $F(u)=u\exp(\rho (1-u))$, both of which are entire functions (i.e. possess no poles in the complex plane). Hence, if $b(\theta)=\sum_{|\alpha|_1\geq 0}b_\alpha\theta^\alpha$ converges on a disc $D \subset \C$, then the Taylor series of $F\circ b$ also converges on $D$. By a slight abuse of notation, we write $F\circ b(\theta)=\sum_{|\alpha|\geq 1} (F(b))_\alpha\theta^\alpha$, where for the logistic nonlinearity, $(F(b))_\alpha =  (1+\rho) b_\alpha - \rho(b *_{\cT} b)_\alpha$, which follows from Proposition \ref{prop-Taylor-multiplication}.  With this notation, the expression for $f(P(\theta))$ can be written as
\begin{equation} \label{eq:general_vector_field_expansion_Taylor}
f(P(\theta))=
 \sum_{|\alpha|_1 \geq 0} \begin{pmatrix} (\phi_1(a))_\alpha \\ (\phi_2(a))_\alpha  \\ (\phi_3(a))_\alpha  \\ (\phi_4(a))_\alpha \end{pmatrix} \theta^\alpha \bydef 
\sum_{|\alpha|_1 \geq 0}
        \begin{pmatrix}
                (a_2)_\alpha \\
                \sigma^2\left((a_1)_\alpha - (F(a_3))_\alpha \right)\\
                (a_4)_\alpha\\
                \sigma^2\left((a_3)_\alpha - (F(a_1))_\alpha \right)
        \end{pmatrix}
        \theta^\alpha.
\end{equation}
Hence, from \eqref{EQ: PDE_compact}, for $|\alpha|_1 \geq 2$, we have the following constraint on the level of Taylor coefficients
\[
    (\alpha \cdot \lambda) (a_j)_\alpha = (\phi_j(a))_\alpha, \qquad \text{for } j = 1,2,3,4.
\]
Define the map $G(a) = (G_1(a),\dots,G_4(a))$ where each $G_j(a)$ ($j = 1,2,3,4$) is given component-wise by
\begin{equation} \label{EQ:manifold_map}
    (G_j(a))_\alpha = 
    \begin{cases}
        (a_j)_{(0,0)} - \tilde{x}_j,\quad &\alpha = (0,0), \\
        (a_j)_{(1,0)} - (\xi_1)_j, \quad &\alpha = (1,0),\\
        (a_j)_{(0,1)} - (\xi_2)_j, \quad &\alpha = (0,1),\\
        (\alpha \cdot \lambda)(a_j)_\alpha - (\phi_j(a))_\alpha, \quad & |\alpha|_1 \ge 2.
    \end{cases}
\end{equation}
Notice that a solution $a = (a_1, a_2, a_3, a_4)$ to $G(a) = 0$ provides the Taylor coefficients of an analytic function that satisfies the constraints \eqref{EQ: constraint_1} and \eqref{EQ: PDE_compact}. 
Because computers cannot store infinite sequences, we work with a finite-dimensional projection of the map, that is we define the truncated map $G^{(N)}:\cT^4 \to \cT^4$ as
\[
G^{(N)} \bydef \pi^{(N)} G  \pi^{(N)}.
\]
Although $G^{(N)}$ is formally a map from $\cT^4$ to $\cT^4$, it is finite in practice since it depends only on finitely many components of its input and produces finitely many nonzero output components. Applying Newton’s method to $G^{(N)}$ yields an approximate solution $\bar{a} \in \pi^{(N)} \cT^4$ below a prescribed tolerance. Using this $\bar{a}$ we define $\bar{P}(\theta)$ just as $P(\theta)$ but with coefficients $\bar{a}$. More explicitly,
\begin{equation} \label{eq:bar_P}
\bar{P}(\theta) \bydef \sum_{|\alpha|_\infty = 0}^{N} \bar{a}_\alpha \theta^\alpha.
\end{equation}
Figure~\ref{fig:manifolds} illustrates the numerically computed local stable manifold for the logistic growth function.

\begin{figure}[h!] 
\begin{center} \includegraphics[width=.65\textwidth]{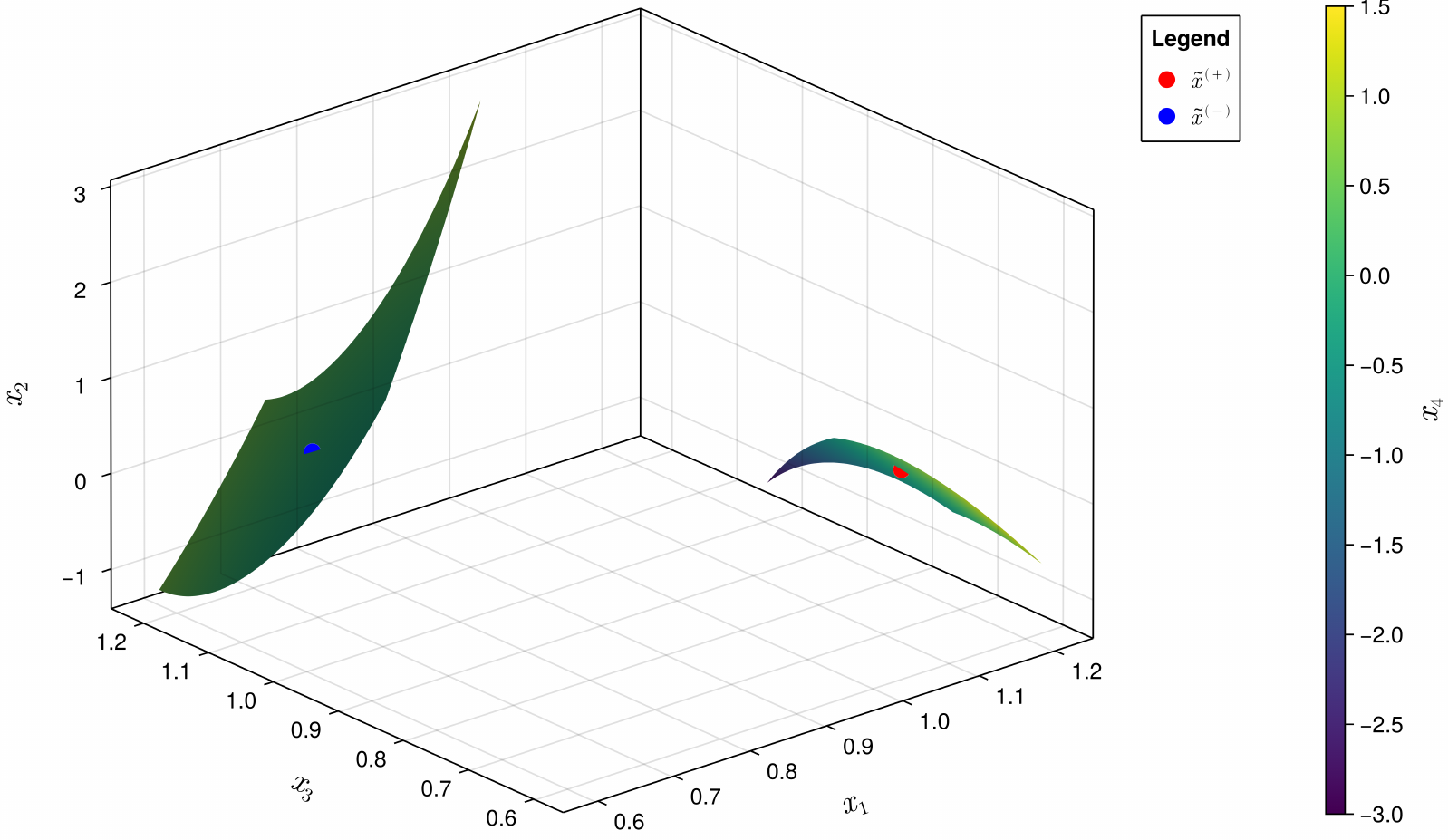} \end{center}
\vspace{-.5cm}
\caption{Plot of the three-dimensional projection of the approximate local stable manifold at $\tilde{x}^{(+)}$ and local unstable manifold at $\tilde{x}^{(-)}$. The fourth dimension is expressed in the color of the plot, specified by the color bar. These are explicitly given by ${P}([-1,1]^2)$ and $R{P}([-1,1]^2)$, respectively. This figure, and all other figures present in this paper, were created using the GLMakie.jl package \cite{Danisch2021}.}
\label{fig:manifolds}
\end{figure}

In order to compute the truncated map $G^{(N)}$ when $F$ is the Ricker growth function, it is necessary to compute a suitable approximation of the Taylor coefficients $F(a_1)$ and $F(a_3)$, for finitely-supported $a_1$ and $a_3$. In our subsequent computer implementation of the numerical (not proven) connecting orbit for the Ricker nonlinearity, these are computed automatically using the RadiiPolynomial.jl library, which uses the Fast Fourier Transform\footnote{Passing from Taylor to Fourier series is done by composing with $e^{i\psi}$; this operation can be inverted.} (FFT), followed by a pointwise operation of the Ricker nonlinearity, and an inverse FFT. As such, none of the terms of the approximations for $F(a_1)$ and $F(a_3)$ are rigorous. See \cite{vanderAalst2025,Breden2024} for some approaches that can be used to get true control on the error. Computer-assisted proofs with the Ricker nonlinearity are not in the scope of this work, so we will not provide further details.

\subsection{Solving the projected BVP using Chebyshev series} \label{sec:BVP_with_Chebyshev}

In this section we formulate a zero-finding problem \eqref{EQ:connecting_orbit_function}, whose solutions yield the Chebyshev coefficients of the solution to the projected BVP \eqref{eq:reduced_problem}. To set the stage, recall that in the previous section we described a procedure for computing a parameterization of $W_{\rm loc}^s(x^{(+)})$, which will later be made rigorous through the computer-assisted method of Section~\ref{sec:bounds_stable_manifold}. For now, suppose that such a parameterization is available, namely a map $P:B^2 \to \R^4$ convergent on $B^2 = [-1,1]^2$. Recalling the boundary value problem \eqref{eq:reduced_problem}, the boundary condition $x(2L) \in W^s(x^{(+)})$ can thus be replaced by $x(2L)=P(\theta)$ for some $\theta \in B^2$. After this substitution, and upon rescaling time from $[0,2L]$ to $[-1,1]$ via $t \mapsto \frac{1}{L}(t-L)$, the BVP \eqref{eq:reduced_problem} takes the form 
\[
\begin{cases}
    x(1) = Rx(1), \\
    x(-1) = P(\theta), \\
    \dot{x} = -Lf(x), \quad t \in [-1, 1],
\end{cases}
\]
where the parameters $\theta=(\theta_1,\theta_2) \in [-1,1]^2$ and $L>0$ remain unknowns to be determined. By Remark~\ref{eq:reversible_manifold}, the system \eqref{EQ:4DODE} is time-reversible with reversor $R(x_1,x_2,x_3,x_4) = (x_3, -x_4, x_1, -x_2)$. Consequently, the condition $x(1) = Rx(1)$ reduces to the two scalar equations $x_1(1) = x_3(1)$ and $x_2(1) = -x_4(1)$. Thus, the three additional unknowns $\theta_1$, $\theta_2$ and $L$ must be balanced with one further equation, which naturally accounts for time-translation invariance. To this end, we impose the normalization $\theta_1^2 + \theta_2^2 = 0.95$, thereby fixing the radius (in parameter space) of the circle in the domain of the parameterization of the local stable manifold. The new balanced resulting projected BVP becomes 
\begin{equation} \label{eq:balanced_BVP}
\begin{cases}
\theta_1^2 + \theta_2^2 = 0.95 \\
x_1(1) = x_3(1) \\
x_2(1) = -x_4(1) \\
x(-1) = P(\theta) \\
\dot{x} = -Lf(x), \quad t \in [-1, 1].
\end{cases}
\end{equation}

Our strategy to rigorously solve the BVP \eqref{eq:balanced_BVP} relies on Chebyshev series expansions for differential equations, following ideas from \cite{MR3148084,MR3865832,MR4068579}. Since $f$ in \eqref{EQ:4DODE} is analytic, the Cauchy–Kovalevskaya theorem (see, e.g., \cite{MR4197075} for an ODE version) ensures that each component $x_j(t)$ is analytic, and thus admits a Chebyshev expansion $x_j(t)=(u_j)_0 + 2\sum_{k \ge 1} (u_j)_k T_k(t)$ with coefficients decaying geometrically. Substituting these series into \eqref{eq:balanced_BVP}, and recalling the tridiagonal operator $\Upsilon$ from \eqref{eq:Upsilon}, we obtain the following system in terms of Chebyshev coefficients
\[
\begin{cases}
    \theta_1^2 + \theta_2^2 = 0.95, \\
    (u_1)_0 + 2\sum_{k \ge 1}(u_1)_k = (u_3)_0 + 2\sum_{k \ge 1}(u_3)_k, \\
    (u_2)_0 + 2\sum_{k \ge 1}(u_2)_k = - (u_4)_0 - 2\sum_{k \ge 1}(u_4)_k, \\
    (u_j)_0 + 2\sum_{k \ge 1} (u_j)_k (-1)^k  = P_j(\theta), &k=0 \\
    2k (u_j)_k = -L(\Upsilon\phi_j(u))_k, &k \ge 1,
\end{cases}
\]
where each $\phi_j(u) = \left( (\phi_j(u))_k \right)_{k \ge 0}$ is given component-wise by
\[
 \begin{pmatrix} (\phi_1(u))_k \\ (\phi_2(u))_k  \\ (\phi_3(u))_k  \\ (\phi_4(u))_k \end{pmatrix} \bydef 
        \begin{pmatrix}
                (u_2)_k \\
                \sigma^2\left((u_1)_k - (F(u_3))_k \right)\\
                (u_4)_k\\
                \sigma^2\left((a_3)_k - (F(u_1))_k \right)
        \end{pmatrix},
\]
where for the logistic map, we have the explicit expression $(F(u))_k = (1+\rho) u_k - \rho(u * u)_k$ with $* = *_\cT$ the discrete convolution defined in \eqref{eq:Chebyshev_product}. Turning this into a zero-finding problem, we obtain 
\begin{equation} \label{EQ:connecting_orbit_function}
\cG(L, \theta, u) \bydef \begin{pmatrix}
    \mathcal{L}(\theta) \\
    \Theta(u) \\
    \mathcal{U}(L, \theta, u) \\
\end{pmatrix}=0,
\end{equation}
where
\begin{align*}
   \mathcal{L}(\theta) &\bydef \theta_1^2 + \theta_2^2 - 0.95, \\
    \Theta_j(u) &\bydef \begin{cases}
        (u_1)_0 + 2\sum_{k \ge 1}(u_1)_k - (u_3)_0 - 2\sum_{k \ge 1}(u_3)_k, & j = 1, \\
        (u_2)_0 + 2\sum_{k \ge 1}(u_2)_k + (u_4)_0 + 2\sum_{k \ge 1}(u_4)_k, & j = 2, \\
    \end{cases} 
    \\
    \mathcal{U}_j(L, \theta, u)_k &\bydef \begin{cases}
        (u_j)_0 + 2\sum_{k \ge 1} (u_j)_k (-1)^k - P_j(\theta), &k=0, \\
        2k (u_j)_k + L(\Upsilon\phi_j(u))_k, &k \ge 1,
        \end{cases}
\end{align*}
for $j = 1,2,3,4$. Define $\mathcal{\bar{U}}$ by replacing $P(\theta)$ with $\bar{P}(\theta)$ in the definition of $\mathcal{U}$. Similarly we define $\mathcal{\bar{G}}$ by replacing $\mathcal{U}$ with $\mathcal{\bar{U}}$ in the definition of $\cG$.

\section{Constructive proof of the 2-cycle for the logistic nonlinearity} \label{sec4}

Recall that in Section~\ref{sec:parameterization_method}, we defined the map $G$ in \eqref{EQ:manifold_map}, where a vector  $a$ satisfying $G(a)=0$ encodes the coefficients of the Taylor expansion of a local stable manifold of $\tx$. In Section~\ref{sec:BVP_with_Chebyshev}, we introduced another map $\cG$ in \eqref{EQ:connecting_orbit_function}, where a vector $(L,\theta,u)$ with $\cG(L,\theta,u)=0$ determines the coefficients of a Chebyshev series representation of a solution to a projected boundary value problem, thereby yielding the desired connecting orbit. As noted earlier, our strategy for addressing these zero-finding problems relies on tools from computer-assisted proofs, in particular through the application of a Newton-Kantorovich type theorem. We open this section by stating such a theorem, which furnishes a constructive framework for proving the existence of zeros of nonlinear operators on Banach spaces.

\begin{thm}[\bf Newton-Kantorovich Theorem] \label{thm:radPolyBanach}
Let $X$ and $Y$ be Banach spaces and $G \colon X \to Y$ be a Fr\'{e}chet differentiable mapping.  
Suppose that $\bx \in X$, $A^{\dagger} \in B(X,Y)$, and $A \in B(Y,X)$. Moreover assume that $A$ is injective. Let $Y_0$, $Z_0$, and $Z_1$ be positive constants and $Z_2 \colon (0, \infty) \to [0, \infty)$ be a non-negative function satisfying 
\begin{align}
\label{eq:Y0_radPolyBanach}
	\| A G(\bar x) \|_X & \leq Y_0,	\\
\label{eq:Z0_radPolyBanach}
	\| I  - A A^\dagger \|_{B(X)} & \leq Z_0, \\
\label{eq:Z1_radPolyBanach}
\| A [ DG(\bar x) - A^\dagger] \|_{B(X)} & \leq Z_1, \\
\label{eq:Z2_radPolyBanach}
\| A [ DG(c) - DG(\bar x) ] \|_{B(X)} &\leq Z_2(r) r, 
\quad \text{for all $c \in \overline{B_r(\bar x)}$ and all  $r > 0$}. 
\end{align}
Define the {\em radii polynomial}
\begin{equation} \label{eq:RPBanach}
p(r) \bydef Z_2(r) r^2 - (1 - Z_0 - Z_1) r + Y_0.
\end{equation}
If there exists $r_0 > 0$ such that $p(r_0) < 0$, then there exists a unique $\tx \in B_{r_0}(\bar x)$ satisfying $G(\tx) = 0$.
\end{thm}

Having introduced the main tool for obtaining constructive existence proofs of solutions to general zero-finding problems $G=0$, we now turn in Section~\ref{sec:bounds_stable_manifold} to the explicit construction of Newton-Kantorovich bounds for the map \eqref{EQ:manifold_map} in the case where $F$ is the logistic growth function. In Section~\ref{sec:manifold_proof}, we apply these bounds to obtain a computer-assisted proof, leading to a rigorous enclosure of the Taylor coefficients in the series expansion \eqref{eq:Taylor_series_manifold}. Next, in Section~\ref{sec:bvp_explicit_bounds}, we derive the Newton-Kantorovich bounds for the map \eqref{EQ:connecting_orbit_function}, again in the logistic growth setting. Section~\ref{sec:bvp proof} is devoted to applying the Newton-Kantorovich theorem to establish a computer-assisted proof for the projected BVP. Finally, in Section~\ref{sec:results}, we combine these results to present a constructive proof of existence for the $2$-cycle of the logistic growth function.

\subsection{Explicit Newton-Kantorovich bounds for \boldmath $W^s_\text{loc}(\tilde{x}^{(+)})$ \unboldmath} \label{sec:bounds_stable_manifold}

In this section, we derive Newton-Kantorovich bounds for the map $G$ defined in \eqref{EQ:manifold_map}, in the case where $F$ is the logistic growth function. We begin by specifying the Banach spaces  $X$ and $Y$ such that $G:X \to Y$. Define
\[
\tilde{\cT} \bydef \left\{ (a_\alpha)_{\alpha \in \N^2}, a_\alpha \in \R: \norm{a}_{\tilde{\cT}} \bydef |a_0| + \sum_{|\alpha|_1 \ge 1} \frac{1}{|\alpha \cdot \lambda|} |a_\alpha| < \infty\right\}.
\]
Setting $X = \cT^4$ and $Y = \tilde{\cT}^4$, it follows from Lemma~\ref{lem:Banach_algebras} and the definition of $\tilde{\cT}$ that $G:X \to Y$. The norm on $X$ is given by
\[
    \norm{a}_X = \max_{i = 1,2,3,4} \mu_i \norm{\pi_i a}_\cT,
\]
where $\mu = (\mu_1, \mu_2, \mu_3, \mu_4)$ is a chosen sequence of weights. Throughout the rest of this section, we fix a truncation dimension $N \in \N$ and as in Section~\ref{sec:parameterization_method}, denote by $\bar{a} \in \pi^{(N)} \cT^4$ the numerical approximation of $G=0$. The next step is to introduce the linear operators $A^{\dagger} \in B(X,Y)$, and $A \in B(Y,X)$ as described in Theorem~\ref{thm:radPolyBanach}. Let
\[
    A^\dagger = 
    \begin{pmatrix}
        A^\dagger_{1,1} & \ldots & A^\dagger_{1,4} \\
        \vdots & \ddots & \vdots \\
        A^\dagger_{4,1} & \ldots & A^\dagger_{4,4}
    \end{pmatrix},
\]
where
\[
    (A^\dagger_{i,j} h_j)_\alpha = 
    \begin{cases}
        (D_{a_j} G_j ^{(N)}(\bar{a})h_j)_\alpha, &|\alpha|_\infty \le N, \\
        \delta_{i,j} (\alpha \cdot \lambda)(h_j)_\alpha, &|\alpha|_\infty > N,
    \end{cases}
\]
for $h = (h_1, h_2, h_3, h_4)\in X$. Note that $A_{i,j}^\dagger \in B(\cT, \tilde{\cT})$. The idea here being that $A^\dagger$ has the following structure. For each block, we apply the derivative of the truncated map to the finite part of the input. Then, for the diagonal blocks, we also apply a tail approximation of the real derivative. As for $A$, we begin by defining $A^{(N)} \approx DG^{(N)}(\bar{a})^{-1}$ with a block decomposition of the form
\[
    A^{(N)} = 
    \begin{pmatrix}
        A^{(N)}_{1,1} & \ldots & A^{(N)}_{1,4} \\
        \vdots & \ddots & \vdots \\
        A^{(N)}_{4,1} & \ldots & A^{(N)}_{4,4}
    \end{pmatrix},
\]
where each block $A^{(N)}_{i,j}$ has size $(N+1)^2 \times (N+1)^2$. 
We can now define
\[
    A = 
    \begin{pmatrix}
        A_{1,1} & \ldots & A_{1,4} \\
        \vdots & \ddots & \vdots \\
        A_{4,1} & \ldots & A_{4,4}
    \end{pmatrix},
\]
where
\[
    (A_{i,j} h_j)_\alpha = 
    \begin{cases}
        (A_{i,j}^{(N)} \pi^{(N)}h_j)_\alpha, &|\alpha|_\infty \le N, \\
        \frac{\delta_{i,j}}{(\alpha \cdot \lambda)} (h_j)_\alpha, &|\alpha|_\infty > N,
    \end{cases}
\]
for $h = (h_1, h_2, h_3, h_4) \in X$. The idea for $A$ is the same as for $A^\dagger$ with $A_{i,j} \in B(\tilde{\cT}, \cT)$. 

With the operators $A$ and $A^\dagger$ in place, we now proceed to derive the bounds $Y_0, Z_0, Z_1,$ and $Z_2$ needed for the application of the Newton-Kantorovich Theorem~\ref{thm:radPolyBanach}. This is done in the specific case of the map $G$ from \eqref{EQ:manifold_map}, where $\phi(a)$ in \eqref{eq:general_vector_field_expansion_Taylor} is defined with the logistic growth nonlinearity.

\subsubsection{The \boldmath $Y_0$ \unboldmath Bound }

Recall first that for the logistic growth nonlinearity, we have $(F(a_j))_\alpha =  (1+\rho) (a_j)_\alpha - \rho(a_j *_{\cT} a_j)_\alpha$. Moreover, since $\bar{a} = \bar{a}^{(N)}$, it follows that $(G(\bar{a}))_\alpha = 0$ for all multi-indices with $|\alpha|_1 > 2N$, and consequently also for all $|\alpha|_\infty > 2N$. We also have $(A_{i,j})_{\alpha,\beta} = 0$ for $|\alpha|_\infty > N$ or $|\beta|_\infty > N$ and $i \ne j$ by definition. Recall that
\[
    \norm{AG(\bar{a})}_X = \max_{i = 1,2,3,4} \mu_i \norm{\pi_iAG(\bar{a})}_{\cT}.
\]
Moreover, 
\[
        \norm{\pi_iAG(\bar{a})}_{\cT} = \sum_{|\alpha|_1 \ge 0} \left|\sum_{j=1}^{4} (A_{i,j} G_j(\bar{a}))_\alpha \right| 
            \le Y_0^{(i)} \bydef \sum_{|\alpha|_\infty = 0}^{N} \left| \sum_{j=1}^{4} (A^{(N)}_{i,j} G_j^{(N)}(\bar{a}))_\alpha \right|
            + \sum_{|\alpha|_\infty = N+1}^{2N} \left| \frac{1}{\alpha \cdot \lambda} (G_i(\bar{a}))_\alpha \right| 
\]
to have
\[
    Y_0 = \max_{i=1,2,3,4} \mu_i Y^{(i)}_0,
\]
which satisfies \eqref{eq:Y0_radPolyBanach}.

\subsubsection{The \boldmath $Z_0$ \unboldmath Bound}

We begin by defining the operator $B \bydef I - AA^\dag$ represented as
\[
    B = 
    \begin{pmatrix}
        B_{1,1} & \cdots & B_{1,4} \\
        \vdots & \ddots & \vdots \\
        B_{4,1} & \cdots & B_{4,4}
    \end{pmatrix}.
\]
Next, let $h = (h_1, h_2, h_3, h_4) \in X$ with $\norm{h}_X = 1$. Notice that $B_{i,j} \in B(\cT)$ have only finitely many non-zero entries. In fact, by construction of the linear operators $A$ and $A^\dag$, $(B_{i,j})_{\alpha,\beta} = 0$ for $|\alpha|_\infty > N$ or $|\beta|_\infty > N$. 
Hence,
\[
\norm{B_{i,j}}_{B(\cT)} = \max_{|\beta|_\infty = 0, \dots, N} \sum_{|\alpha|_\infty = 0, \dots, N} |(B_{i,j})_{\alpha, \beta}|.
\]
and using this formula, we set 
%
\[
    Z_0 = \max_{i=1,2,3,4} \mu_i \left( \sum_{j = 1}^{4} \frac{1}{\mu_j} \norm{B_{i,j}}_{B(\cT)} \right),
\]
which satisfies \eqref{eq:Z0_radPolyBanach}.

\subsubsection{The \boldmath $Z_1$ \unboldmath Bound}
Let $h \in X$ with $||h||_X = 1$, then consider $z \bydef [DG(\bar{a}) - A^\dagger]h$. We begin by showing that $(z_i)_\alpha = 0$ for $|\alpha|_\infty \le N$ and $i = 1,2,3,4$. Indeed, one can see
\[
    (z_i)_\alpha = (\pi_i DG(\bar{a})h)_\alpha - (\pi_i A^\dagger h)_\alpha 
    = \sum_{j = 1}^{4} (D_{a_j} G_i(\bar{a}) h_j)_\alpha - (A_{i,j}^\dagger h)_\alpha 
    = \sum_{j = 1}^{4} (D_{a_j} G_i(\bar{a}) h_j)_\alpha - (D_{a_j}G_i^{(N)}(\bar{a})h)_\alpha, 
\]
for $|\alpha|_\infty \le N$. Each difference in the sum is a linear combination of terms of the form
$\bar{a}_k *_{\cT} h_k - \bar{a}_k *_{\cT} \pi^{(N)}h_k = \bar{a}_k *_{\cT} (h_k - \pi^{(N)}h_k)$ (with $k \in \{ 1,2,3,4\}$), which all have coefficients equal to $0$ when $|\alpha|_\infty \le N$. Meanwhile, for $|\alpha|_\infty > N$, we have
\[
    (z_i)_\alpha = -(D \phi_i(\bar{a})h)_\alpha 
                 = - \sum_{j=1}^{4} (D_{a_j} \phi_i(\bar{a}) h_j)_\alpha.
\]
Since the tail of $A_{i,j}$ is 0 for $i \ne j$, we have $Az = (A_{1,1}z_1, \dots, A_{4,4}z_4)$ and
\[
    (A_{i,i}z_i)_\alpha = -\frac{1}{\alpha \cdot \lambda} \sum_{j=1}^{4} (D_{a_j}\phi_i(\bar{a}) h_j)_\alpha
\]
for $|\alpha|_\infty > N$, and 0 otherwise. Letting
\[
\lambda^*(N) \bydef \min_{|\alpha|_\infty > N} |\alpha \cdot \lambda| = (N+1) \min\{|\lambda_1|,| \lambda_2|\}
\]
yields
\[
{\small
\begin{aligned}
        \norm{A_{i,i}z_i}_{\cT} &= \sum_{|\alpha|_\infty > N} \left|-\frac{1}{\alpha \cdot \lambda} \sum_{j=1}^{4} (D_{a_j}\phi_i(\bar{a}) h_j)_\alpha \right| 
                     \le \frac{1}{\lambda^*(N)} \sum_{|\alpha|_\infty > N} \sum_{j=1}^{4} \left| (D_{a_j}\phi_i(\bar{a}) h_j)_\alpha \right| \\
                     &\le \frac{1}{\lambda^*(N)} \sum_{j=1}^{4} \norm{(D_{a_j}\phi_i(\bar{a}) h_j)}_{\cT} 
 \le Z^{(i)}_1 \bydef \frac{1}{\lambda^*(N)} \sum_{j=1}^{4} \frac{1}{\mu_j} \norm{D_{a_j}\phi_i(\bar{a})}_{B(\cT)}
    \end{aligned}
    }
\]
where the last inequality follows from $\norm{h_j} \le \frac{1}{\mu_j}$. Hence, a bound satisfying \eqref{eq:Z1_radPolyBanach} is given by
\[
    Z_1 = \max_{i=1,2,3,4} \mu_i Z^{(i)}_1.
\]

\subsubsection{The \boldmath $Z_2$ \unboldmath Bound}

Let $r > 0$, $b \in X$ with $\norm{b}_X \le r$, $h \in X$ with $\norm{h}_X = 1$, and let $z \bydef [DG(\bar{a}+b) - DG(\bar{a})]h$. Then
\[
(z_i)_\alpha = \begin{cases}
    0, & 0 \le |\alpha|_1 \le 1, \\
   (\Phi_i)_\alpha \bydef \sum_{j = 1}^{4} D_{a_j} \phi_i(\bar{a}+b)h_j - D_{a_j}\phi_i(\bar{a})h_j, &|\alpha|_1 > 1.
\end{cases}
\]
where
\[
\Phi_i = \begin{cases}
    0, &i = 1, \\
    2\sigma^2  \rho (b_3 *_{\cT} h_3), &i = 2, \\
    0, &i = 3, \\
    2 \sigma^2 \rho (b_1 *_{\cT} h_1), &i = 4.
\end{cases}
\]
Therefore,
{\small
\begin{align*}
    \norm{Az}_X &= \max_{i = 1,2,3,4} \mu_i \norm{\pi_i Az}_\cT \\
    &= \max_{i = 1,2,3,4} \mu_i\norm{\sum_{j = 1}^{4} A_{i,j}z_j}_{\cT} \\
    &\le \max_{i = 1,2,3,4} \mu_i \sum_{j = 1}^{4} \norm{A_{i,j}}_{B(\cT)} \norm{z_j}_{\cT} \\
    &= \max_{i = 1,2,3,4} \mu_i \left(\norm{A_{i,2}}_{B(\cT)} \norm{z_2}_{\cT} + \norm{A_{i,4}}_{B(\cT)} \norm{z_4}_{\cT}\right) \\
    &\le 2\rho r \sigma^2   \max_{i = 1,2,3,4} \mu_i \left( \frac{1}{\mu_3^2}\norm{A_{i,2}}_{B(\cT)} + \frac{1}{\mu_1^2}\norm{A_{i,4}}_{B(\cT)} \right),
\end{align*}
}
where in the last inequality, we applied the Banach algebra property of $(\cT, \norm{\cdot}_\cT, *_\cT)$ (see Lemma~\ref{lem:Banach_algebras}). Therefore, setting $Z_2$ (independent of $r$ here since the logistic growth function is quadratic) as
\[
Z_2 \bydef 2\rho \sigma^2 \max_{i = 1,2,3,4} \mu_i \left(\frac{1}{\mu_3^2}\norm{A_{i,2}}_{B(\cT)} + \frac{1}{\mu_1^2}\norm{A_{i,4}}_{B(\cT)}\right)
\]
we get that is satisfies \eqref{eq:Z2_radPolyBanach}. It remains to compute $\norm{A_{i,j}}_{B(\cT)}$. This is easy for the finite operators. For the infinite ones, namely $A_{j,j}$, we have the computable bound 
\[
{\small
\begin{split}
    \norm{A_{j,j}}_{B(\cT)} &= \sup_{|\beta|_\infty \ge 0} \left( \sum_{|\alpha|_\infty \ge 0} |A_{j,j}|_{\alpha, \beta}\right) \\
    &= \max \left\{ \sup_{0 \le |\beta|_\infty \le N} \left( \sum_{|\alpha|_\infty \ge 0} |A_{j,j}|_{\alpha, \beta}\right), \sup_{|\beta|_\infty > N} \left( \sum_{|\alpha|_\infty \ge 0} |A_{j,j}|_{\alpha, \beta}\right) \right\} \\
    &= \max \left\{ \sup_{0 \le |\beta|_\infty \le N} \left( \sum_{0 \le |\alpha|_\infty \le N} |A_{j,j}^{(N)}|_{\alpha, \beta}\right), \sup_{|\beta|_\infty > N} \frac{1}{|\beta \cdot \lambda|}\right\} \\
    &= \max \left\{ \norm{A_{j,j}^{(N)}}_{B(\pi^{(N)}\cT)}, \sup_{|\beta|_\infty > N} \frac{1}{\lambda^*(N)}\right\}.
\end{split}.
}
\]

\subsection{Computer-assisted proof for \boldmath $W^s_\text{loc}(\tilde{x}^{(+)})$ \unboldmath}
\label{sec:manifold_proof}

In this subsection, we use the computable bounds presented in Section~\ref{sec:bounds_stable_manifold} to compute a rigorous enclosure on the parametrization of $W^s_\text{loc}(\tilde{x}^{(+)})$ using the Newton-Kantorovich Theorem~\ref{thm:radPolyBanach}. To this end, let $N = 30$ and define $\bar{a} \in \pi^{(N)} \cT^4$ where its first few coefficients are given in Table~\ref{tab:manifold_coefficients}.

\begin{table}[h]
    \centerline{
    {\tiny
    \begin{tabular}{|c|c|c|c|c|c|c|c|}
    \hline
    $(\bar{a}_1)_{0,0}$& $1.162844349$ & $(\bar{a}_2)_{0,0}$ & 0 & $(\bar{a}_3)_{0,0}$ & $0.7462465593$ & $(\bar{a}_4)_{0,0}$ & 0\\
    \hline
    $(\bar{a}_1)_{1,0}$& $2.615916931 \times 10^{-2}$ & $(\bar{a}_2)_{1,0}$ & $-0.2026280542$ & $(\bar{a}_3)_{1,0}$ & $-0.1253361100$ & $(\bar{a}_4)_{1,0}$ & $0.9708493338$\\
    \hline
    $(\bar{a}_1)_{2,0}$& $2.694990900\times 10^{-2}$ & $(\bar{a}_2)_{2,0}$ & $-0.4175061949$ & $(\bar{a}_3)_{2,0}$ & $3.796812378\times 10^{-2}$ & $(\bar{a}_4)_{2,0}$ & $-0.5881996444$\\
    \hline
    $(\bar{a}_1)_{3,0}$& $-4.784945637\times 10^{-3}$ & $(\bar{a}_2)_{3,0}$ & $0.1111920885$ & $(\bar{a}_3)_{3,0}$ & $-1.379199302\times 10^{-3}$ & $(\bar{a}_4)_{3,0}$ & $3.204969557\times 10^{-2}$\\
    \hline
    $(\bar{a}_1)_{0,1}$& $-1.720597589\times 10^{-2}$ & $(\bar{a}_2)_{0,1}$ & $0.2035838522$ & $(\bar{a}_3)_{0,1}$ & $-8.243878319\times 10^{-2}$ & $(\bar{a}_4)_{0,1}$ & $0.9754288372$\\
    \hline
    $(\bar{a}_1)_{1,1}$& $1.635296640\times 10^{-2}$ & $(\bar{a}_2)_{1,1}$ & $-0.3201604410$ & $(\bar{a}_3)_{1,1}$ & $1.036356517\times 10^{-2}$ & $(\bar{a}_4)_{1,1}$ & $-0.2028991875$\\
    \hline
    $(\bar{a}_1)_{2,1}$& $-3.025706251\times 10^{-3}$ & $(\bar{a}_2)_{2,1}$ & $8.267467886\times 10^{-2}$ & $(\bar{a}_3)_{2,1}$ & $-9.212498773\times 10^{-4}$ & $(\bar{a}_4)_{2,1}$ & $2.517231728\times 10^{-2}$\\
    \hline
    $(\bar{a}_1)_{3,1}$& $2.440505209\times 10^{-4}$ & $(\bar{a}_2)_{3,1}$ & $-8.558866325\times 10^{-3}$ & $(\bar{a}_3)_{3,1}$ & $2.142513437\times 10^{-4}$ & $(\bar{a}_4)_{3,1}$ & $-7.513807404\times 10^{-3}$\\
    \hline
    $(\bar{a}_1)_{0,2}$& $3.277687706\times 10^{-3}$ & $(\bar{a}_2)_{0,2}$ & $-7.756424789\times 10^{-2}$ & $(\bar{a}_3)_{0,2}$ & $1.507182270\times 10^{-3}$ & $(\bar{a}_4)_{0,2}$ & $-3.566644223\times 10^{-2}$\\
    \hline
    $(\bar{a}_1)_{1,2}$& $-5.197174064\times 10^{-4}$ & $(\bar{a}_2)_{1,2}$ & $1.632447228\times 10^{-2}$ & $(\bar{a}_3)_{1,2}$ & $-2.094285011\times 10^{-4}$ & $(\bar{a}_4)_{1,2}$ & $6.578209081\times 10^{-3}$\\
    \hline
    $(\bar{a}_1)_{2,2}$& $6.594635560\times 10^{-5}$ & $(\bar{a}_2)_{2,2}$ & $-2.582212152\times 10^{-3}$ & $(\bar{a}_3)_{2,2}$ & $8.879511634\times 10^{-5}$ & $(\bar{a}_4)_{2,2}$ & $-3.476883997\times 10^{-3}$\\
    \hline
    $(\bar{a}_1)_{3,2}$& $-1.020718124\times 10^{-5}$ & $(\bar{a}_2)_{3,2}$ & $4.787394524\times 10^{-4}$ & $(\bar{a}_3)_{3,2}$ & $-1.803928884\times 10^{-5}$ & $(\bar{a}_4)_{3,2}$ & $8.460826799\times 10^{-4}$\\
    \hline
    $(\bar{a}_1)_{0,3}$& $-4.733965751\times 10^{-5}$ & $(\bar{a}_2)_{0,3}$ & $1.680391144\times 10^{-3}$ & $(\bar{a}_3)_{0,3}$ & $-2.921282406\times 10^{-5}$ & $(\bar{a}_4)_{0,3}$ & $1.036952387\times 10^{-3}$\\
    \hline
    $(\bar{a}_1)_{1,3}$& $9.161972092\times 10^{-6}$ & $(\bar{a}_2)_{1,3}$ & $-3.961860778\times 10^{-4}$ & $(\bar{a}_3)_{1,3}$ & $1.623223268\times 10^{-5}$ & $(\bar{a}_4)_{1,3}$ & $-7.019214352\times 10^{-4}$\\
    \hline
    $(\bar{a}_1)_{2,3}$& $-2.480914858\times 10^{-6}$ & $(\bar{a}_2)_{2,3}$ & $1.264979091\times 10^{-4}$ & $(\bar{a}_3)_{2,3}$ & $-3.813785350\times 10^{-6}$ & $(\bar{a}_4)_{2,3}$ & $1.944588590\times 10^{-4}$\\
    \hline
    $(\bar{a}_1)_{3,3}$& $5.344086815\times 10^{-7}$ & $(\bar{a}_2)_{3,3}$ & $-3.138816192\times 10^{-5}$ & $(\bar{a}_3)_{3,3}$ & $5.605903060\times 10^{-7}$ & $(\bar{a}_4)_{3,3}$ & $-3.292592337\times 10^{-5}$\\
    \hline
    \end{tabular}
    }
    }
    \caption{First few coefficients (with 10 digits of precision) of the approximate solution $\ba = (\ba_1,\dots,\ba_4)$ to $G(a)=0$ with $G$ defined component-wise in \eqref{EQ:manifold_map}.}
    \label{tab:manifold_coefficients}
\end{table}

\begin{theorem} \label{thm:manif_proof}
    Let $G:X \to Y$ be defined as in \eqref{EQ:manifold_map} with parameter values $\sigma = 10$ and $\rho = 2.2$. Then there exists $\tilde{a} \in X$ such that $G(\tilde{a}) = 0$. Moreover,
    \[
    \norm{\bar{a} - \tilde{a}}_{X} \le r_{\rm manif} \bydef 7.464746738654994 \times 10^{-14},
    \]
	where $\bar{x}$ is as defined above and 
    \[
    \mu=(\sigma, 1, 1, \sigma, 1, \sigma, 1).
    \]
\end{theorem}
\begin{proof}
Using the code provided in \cite{github}, we compute the bounds developed in Section~\ref{sec:bounds_stable_manifold} with $N = 30$ to obtain the numbers found in Table~\ref{tab:manifold_bounds}. More precisely, we simply run the main.jl file. This gives us a range
    \[
    r \in (7.464746738654994 \times 10^{-14}, 0.18615659613392965),
    \]
where the radii polynomial defined in \eqref{eq:RPBanach} is negative. Hence, by the Newton-Kantorovich Theorem~\ref{thm:radPolyBanach}, there exists a unique $\tilde{a} \in X$ such that $\norm{\bar{a} - \tilde{a}}_X \le r_{\rm manif}$ and $G(\tilde{a}) = 0$.    
\end{proof}

\begin{table}[h]
    \centering
    \begin{tabular}{|c|c|} 
        \hline
        $Y_0$& $6.420200972375719 \times 10^{-14}$\\
        \hline
        $Z_0$& $9.666879531520297 \times 10^{-13}$ \\
        \hline
        $Z_1$& $0.13963106975631812$\\
        \hline
        $Z_2$& $4.621748292085127$ \\
         \hline
    \end{tabular}
    \caption{Newton-Kantorovich bounds for the local stable manifold proof.}
    \label{tab:manifold_bounds}
\end{table}
\begin{lem} \label{lem:parameterization_properties}
Let $\tilde{a}, \bar{a}$ be defined as in Theorem~\eqref{thm:manif_proof} and let
\[
P(\theta) \bydef \sum_{|\alpha|_1 \ge 0} \tilde{a}_\alpha \theta^\alpha \qquad \text{and} \qquad \bar{P}(\theta) \bydef \sum_{|\alpha|_\infty = 0}^{N} \bar{a}_\alpha \theta^\alpha.
\]
Then $P([-1,1]^2)$ is a local stable manifold at $\tilde{x}^{(+)}$ of \eqref{EQ:4DODE} and $\zeta(t) \bydef P(e^{\Lambda t}\theta)$ solves the IVP
    \[
    \begin{cases}
        \zeta(0) = P(\theta), \\
        \dot{\zeta} = f(\zeta), &t \ge 0,
    \end{cases}
    \]
    for $\theta \in [-1,1]^2$. Moreover, for $j=1,2,3,4$ and $\theta \in [-1,1]^2$,
    \[
    |P_j(\theta) - \bar{P}_j(\theta)| \leq r_{\rm manif}.
    \]
    Finally, we have $\zeta(t) \to \tilde{x}^{(+)}$ as $t \to \infty$.
\end{lem}
\begin{proof}
    We begin by ensuring that $P(\theta)$ is well-defined on $[-1,1]^2$. Indeed, for $j = 1,2,3,4$ and $0 \le |\theta|_1, |\theta|_2 \le 1$, we have
    \[
        \sum_{|\alpha|_1 \ge 0} |(\tilde{a}_j)_\alpha| |\theta^\alpha|
        \le \sum_{|\alpha|_1 \ge 0} |(\tilde{a}_j)_\alpha|
        = \norm{(\tilde{a}_j)}_\cT
        \le \norm{\tilde{a}_j - \bar{a}_j}_\cT + \norm{\bar{a}_j}_\cT \le r_{\rm manif} + \norm{\bar{a}_j}_\cT < \infty.
    \]
So $P_j(\theta)$ is absolutely convergent and hence converges. 
By construction of $G$, $P(\theta)$ solves \eqref{EQ: PDE_compact} with constraints \eqref{EQ: constraint_1}, $P([-1,1]^2)$ is a local stable manifold at $\tilde{x}^{(+)}$ of \eqref{EQ:4DODE} and $\zeta(t)$ solves the IVP above. As for the inequality, for $j=1,2,3,4$ and $\theta \in [-1,1]^2$, we have
    \[
        |P_j(\theta) - \bar{P}_j(\theta)| 
        \le \sum_{|\alpha|_\infty \ge 0} |(\tilde{a}_j)_\alpha - (\bar{a}_j)_\alpha| |\theta^\alpha| 
        \le \sum_{|\alpha|_\infty \ge 0} |(\tilde{a}_j)_\alpha - (\bar{a}_j)_\alpha| 
        = \norm{\tilde{a}_j - \bar{a}_j}_\cT \le r_{\rm manif},
    \]
    where the last inequality follows from Theorem~\ref{thm:manif_proof}. As for the last statement,
    \[
    \lim_{t \to \infty} \zeta(t) = P(0) = \tilde{a}_{0,0} = \tilde{x}^{(+)},
    \]
    since $G(\tilde{a}) = 0$ by Theorem~\ref{thm:manif_proof}.
\end{proof}

\subsection{Newton-Kantorovich bounds for solution to the projected BVP}
\label{sec:bvp_explicit_bounds}

In this section, we derive Newton-Kantorovich bounds for the projected BVP map $\cG$ defined in \eqref{EQ:connecting_orbit_function}, in the case where $F$ is the logistic growth function. We begin by specifying the Banach spaces $\cX$ and $\cY$ such that $\cG: \cX \to \cY$. Setting
\[
\tilde \cC_\nu \bydef \left\{ (u_k)_{k \in \N}, u_k \in \R: \norm{\cdot}_{\cC_\nu} \bydef |u_0| + 2 \sum_{k > 0} \frac{|u_k|}{2k} \nu^k < \infty\right\},
\]
$\cX = \R^3 \times (\cC_\nu)^4$ and $\cX = \R^3 \times (\tilde \cC_\nu)^4$, it follows from Lemma~\ref{lem:Banach_algebras} and the definition of the sequence space $\tilde \cC_\nu$ that $\cG: \cX \to \cY$. In addition to the Banach spaces, it will be important to define projections on this space for notation later. First, noticing that the unknowns in $\cX$ are $x=(L,\theta_1,\theta_2,u_1,u_2,u_3,u_4)$, we define
\[
\pi_L \bydef \pi_1, \quad \pi_{\theta_1} \bydef \pi_2, \quad \pi_{\theta_2} \bydef \pi_3, \quad \pi_{u_1} \bydef \pi_4, \quad \pi_{u_2} \bydef \pi_5, \quad \pi_{u_3} \bydef \pi_6 \quad \text{and} \quad \pi_{u_4} \bydef \pi_7.
\]
With some abuse of notation, we further define
\[
\begin{split}
    \pi_L \cX, \pi_{\theta_1}\cX, \pi_{\theta_2}\cX &\bydef \R, \\
    \pi_{u_1}\cX, \pi_{u_2}\cX, \pi_{u_3}\cX, \pi_{u_4}\cX &\bydef \cC_\nu.
\end{split}
\]
Then, for $h = (h_L, h_{\theta_1}, h_{\theta_2}, h_{u_1}, h_{u_2}, h_{u_3}, h_{u_4}) \in \cX$, we define
\[
h_\theta \bydef (h_{\theta_1}, h_{\theta_2}), \quad h_u \bydef (h_{u_1}, h_{u_2}, h_{u_3}, h_{u_4}),
\]
so we write $h = (h_L, h_\theta, h_u) \in \cX$. Finally, we define a norm on our space
\[
    \norm{x}_\cX = \max_{\alpha \in \{L, \theta, u\}} \mu_\alpha \norm{\pi_\alpha x}_{\pi_\alpha \cX},
\]
where $\max_{\alpha \in \{L, \theta, u\}}$ is shorthand for $ \max_{\alpha \in \{L, \theta_1, \theta_2, u_1, u_2, u_3, u_4 \}}$ and $\mu = (\mu_L, \mu_{\theta_1}, \mu_{\theta_2}, \mu_{u_1}, \mu_{u_2}, \mu_{u_3}, \mu_{u_4})$ is a chosen sequence of weights that will minimize certain bounds from the Newton-Kantorovich Theorem. Let $\bar{x} = (\bar{L}, \bar{\theta}, \bar{u})$ denote an approximate solution obtained by applying Newton’s method to the finite-dimensional projection
\[
\cG^{(N)} \bydef \pi^{(N)} \cG \pi^{(N)}
\]
for some $N \in \N$. Note that we define $\bar{\cG}^{(N)}$ similarly. Recalling Definition~\ref{defn:projections}, the projection $\pi^{(N)}:\cX \to \cX$ is defined, for $(L,\theta,u) \in \cX$, by
\[
\pi^{(N)}(L,\theta,u) \bydef \left(L,\theta, \pi^{(N)} u \right).
\]
The next step is to introduce the linear operators $A^{\dagger} \in B(\cX,\cY)$, and $A \in B(\cY,\cX)$ as described in Theorem~\ref{thm:radPolyBanach}. Let
\renewcommand\arraystretch{1.5}
\[
{\small 
A^\dagger = \begin{pmatrix}
    A^\dagger_{L,L}           & A^\dagger_{L,\theta_1}           & A^\dagger_{L,\theta_2}           &        A^\dagger_{L, u_1} & A^\dagger_{L, u_2} & A^\dagger_{L, u_3} & A^\dagger_{L, u_4} \\
    A^\dagger_{\theta_1,L}      & A^\dagger_{\theta_1,\theta_1} & A^\dagger_{\theta_1,\theta_2} & A^\dagger_{\theta_1, u_1} &  A^\dagger_{\theta_1, u_2} & A^\dagger_{\theta_1, u_3} & A^\dagger_{\theta_1, u_4} &   \\ A^\dagger_{\theta_2,L}      & A^\dagger_{\theta_2,\theta_1} & A^\dagger_{\theta_2,\theta_2} & A^\dagger_{\theta_2, u_1} &  A^\dagger_{\theta_2, u_2} & A^\dagger_{\theta_2, u_3} & A^\dagger_{\theta_2, u_4} &   \\
    A^\dagger_{u_1,L}           & A^\dagger_{u_1,\theta_1}  & A^\dagger_{u_1,\theta_2}         & A^\dagger_{u_1, u_1} & A^\dagger_{u_1, u_2} &  A^\dagger_{u_1, u_3} & A^\dagger_{u_1, u_4} &   \\
    A^\dagger_{u_2,L}           & A^\dagger_{u_2,\theta_1}  & A^\dagger_{u_2,\theta_2}         & A^\dagger_{u_2, u_1} & A^\dagger_{u_2, u_2} &  A^\dagger_{u_2, u_3} & A^\dagger_{u_2, u_4} &   \\
    A^\dagger_{u_3,L}           & A^\dagger_{u_3,\theta_1}  & A^\dagger_{u_3,\theta_2}         & A^\dagger_{u_3, u_1} & A^\dagger_{u_3, u_2} &  A^\dagger_{u_3, u_3} & A^\dagger_{u_3, u_4} &   \\
    A^\dagger_{u_4,L}           & A^\dagger_{u_4,\theta_1}  & A^\dagger_{u_4,\theta_2}         & A^\dagger_{u_4, u_1} & A^\dagger_{u_4, u_2} &  A^\dagger_{u_4, u_3} & A^\dagger_{u_4, u_4} &   \\
\end{pmatrix}
}
\]
where
\[
\begin{split}
    A^\dagger_{L,\theta_j} h_{\theta_j} &= D_{\theta_j} \mathcal{L}(\bar{\theta}) h_{\theta_j}, \\
    A^\dagger_{\theta_i,u_j} h_{u_j} &= D_{u_j}\Theta^{(N)}_i(\bar{u}) h_{u_j}, \\
    (A^\dagger_{u_i,L} h_L)_k &= \begin{cases}
        (D_L \bar{\mathcal{U}}_i^{(N)}(\bar{x}) h_L)_k, & 0 \le k \le N, \\
        0, &k > N,
    \end{cases}\\
    (A^\dagger_{u_i,\theta_j} h_{\theta_j})_k &= \begin{cases}
        (D_{\theta_j} \bar{P}_i(\bar{\theta}) h_{\theta_j}), & k = 0, \\
        0, &k \ge 1,
    \end{cases}\\
    (A^\dagger_{u_i,u_j} h_{u_j})_k &= \begin{cases}
        (D_{u_j} \bar{\mathcal{U}}_i^{(N)}(\bar{x}) h_{u_j})_k, &0 \le k \le N, \\
        \delta_{i,j} 2k (h_{u_j})_k, &k > N,
    \end{cases}
\end{split}
\]
for $h = (h_L, h_\theta, h_u) \in \cX$. The rest of the operators are identically zero. Note that $A^\dagger_{\alpha, \beta} \in B(\pi_\beta \cX, \pi_\alpha \cY).$ From now on, we use the more compact notation
\[
A^\dagger = \begin{pmatrix}
        A_{L,L}^\dagger & A_{L,\theta}^\dagger & A_{L, u}^\dagger \\
        A_{\theta,L}^\dagger & A_{\theta,\theta}^\dagger & A_{\theta, u}^\dagger      \\
        A_{u,L}^\dagger & A_{u,\theta}^\dagger & A_{u, u}^\dagger
    \end{pmatrix}
\]
instead. The idea is the same as the manifold proof. We want $A^\dagger$ to be the truncated version of the derivative with some infinite decaying tail in the diagonal terms. As for $A$, we begin by defining $A^{(N)} \approx D\cG^{(N)}(\bar{x})^{-1}$. Then let
\[
A^{(N)} = \begin{pmatrix}
    A^{(N)}_{L,L} & A^{(N)}_{L,\theta} & A^{(N)}_{L, u} \\
    A^{(N)}_{\theta,L} & A^{(N)}_{\theta,\theta} & A^{(N)}_{\theta, u} \\
    A^{(N)}_{u,L} & A^{(N)}_{u,\theta} & A^{(N)}_{u, u} \\
\end{pmatrix},
\]
with $A^{(N)}_{\alpha, \beta} \in B(\pi^{(N)} \pi_\beta \cY, \pi^{(N)} \pi_\alpha \cX)$. Finally, we let
\[
    A = \begin{pmatrix}
        A_{L,L}           & A_{L,\theta}           & A_{L, u} \\
        A_{\theta,L}      & A_{\theta,\theta}      & A_{\theta, u}      \\
        A_{u,L}           & A_{u,\theta}           & A_{u, u}      \\
    \end{pmatrix},
\]
in the same fashion as $A^\dagger$ such that $A_{\alpha, \beta} \in B(\pi_\beta \cY, \pi_\alpha \cX)$ where the blocks are defined below. Firstly, 
\[
    A_{\alpha, \beta} h_\beta = A^{(N)}_{\alpha, \beta} \pi^{(N)}h_\beta
\]
for $\alpha \in \{L, \theta_1, \theta_2\}$ and $\beta \in \{L, \theta_1, \theta_2, u_1, u_2, u_3, u_4\}$. Next, 
\[
(A_{\alpha, \beta} h_\beta)_k = \begin{cases}
    (A^{(N)}_{\alpha, \beta} h_\beta)_k & 0 \le k \le N, \\
    0 & k > N,
\end{cases}
\]
for $\alpha \in \{u_1, u_2, u_3, u_4\}$ and $\beta \in \{L, \theta_1, \theta_2\}$. Finally,
\[
(A_{\alpha, \beta} h_\beta)_k = \begin{cases}
    (A^{(N)}_{\alpha, \beta} \pi^{(N)} h_\beta)_k & 0 \le k \le N, \\
     \delta_{\alpha, \beta} \frac{1}{2k} (h_\beta)_k & k > N,
\end{cases}
\]
for $\alpha, \beta \in \{u_1, u_2, u_3, u_4\}$. Having introduced the operators $A, A^\dagger$, we now turn to the derivation of the bounds $Y_0, Z_0, Z_1,$ and $Z_2$ required to apply the Newton-Kantorovich Theorem~\ref{thm:radPolyBanach}.

\subsubsection{The \boldmath $Y_0$ \unboldmath Bound}

Note that $(\mathcal{U}(\bar{x}))_k = (\mathcal{U}(\bar{L},\bar{\theta},\bar{u}))_k = 0$ for $k > 2N$ since $\mathcal{U}$ has max degree 2 and $\bar{u} = \pi^{(N)} \bar{u}$. We also have $(A_{u_i,u_j})_{n,m} = 0$ for $|n| > N$ or $|m| > N$ for $i \ne j$. We now need to bound
\[
    \norm{A\cG(\bar{x})}_\cX = \max_{\alpha \in \{L,\theta, u\}}\mu_\alpha \norm{\pi_\alpha A\cG(\bar{x})}_{\pi_\alpha \cX}.
\]
To begin, for $\alpha \in \{L, \theta_1, \theta_2 \}$ we have
\[
{\small
\begin{split}
        \norm{\pi_\alpha A\cG(\bar{x})}_{\pi_\alpha \cX} &= \left| \sum_{\beta \in \{L, \theta, u\}} A^{(N)}_{\alpha, \beta} \pi_\beta \cG^{(N)}(\bar{x}) \right| \\
        &\le \sum_{\beta \in \{L, \theta, u\}} \left| A^{(N)}_{\alpha, \beta} \pi_\beta (\cG^{(N)}(\bar{x})- \bar{\cG}^{(N)}(\bar{x}) + \bar{\cG}^{(N)}(\bar{x})) \right|\\
        &\le \sum_{\beta \in \{L, \theta, u\}} \left| A^{(N)}_{\alpha, \beta} \pi_\beta (\cG^{(N)}(\bar{x})- \bar{\cG}^{(N)}(\bar{x})) \right| + \sum_{\beta \in \{L, \theta, u\}} \left| A^{(N)}_{\alpha, \beta} \pi_\beta \bar{\cG}^{(N)}(\bar{x}) \right|\\
        &= \sum_{j = 1}^{4} \left| (A^{(N)}_{\alpha, u_j})_0 (P_j(\bar{\theta}) - \bar{P}_j(\bar{\theta})) \right| + \sum_{\beta \in \{L, \theta, u\}} \left| A^{(N)}_{\alpha, \beta} \pi_\beta \bar{\cG}^{(N)}(\bar{x}) \right|\\
        &\le \sum_{j = 1}^{4} \left| (A^{(N)}_{\alpha, u_j})_0 \right| \left|P_j(\bar{\theta}) - \bar{P}_j(\bar{\theta})\right| + \sum_{\beta \in \{L, \theta, u\}} \left| A^{(N)}_{\alpha, \beta} \pi_\beta \bar{\cG}^{(N)}(\bar{x}) \right|\\
        &\le Y_0^{(\alpha)} \bydef r_{\rm manif}\sum_{j = 1}^{4} \left| (A^{(N)}_{\alpha, u_j})_0 \right| + \sum_{\beta \in \{L, \theta, u\}} \left| A^{(N)}_{\alpha, \beta} \pi_\beta \bar{\cG}^{(N)}(\bar{x}) \right|,
\end{split}
}
\]
where the last inequality comes from Lemma~\ref{lem:parameterization_properties}. Next, for $\alpha \in \{u_1, u_2, u_3, u_4\}$, we have
\[
{\small
\begin{split}
    \norm{\pi_\alpha A\cG(\bar{x})}_{\pi_\alpha \cX} &= \sum_{k = 0}^{N} \left| \sum_{\beta \in \{L, \theta, u \}} (A^{(N)}_{\alpha, \beta} \pi_\beta \cG^{(N)}(\bar{x}))_k\right| \omega_k
    + 2 \sum_{k = N+1}^{2N} \frac{1}{2k} \left|(\pi_\alpha \cG(\bar{x}))_k\right| \nu^k \\
    &\le \sum_{k = 0}^{N}  \sum_{\beta \in \{L, \theta, u \}} \left|\left[A^{(N)}_{\alpha, \beta} \pi_\beta (\cG^{(N)}(\bar{x})- \bar{\cG}^{(N)}(\bar{x}) + \bar{\cG}^{(N)}(\bar{x}))\right]_k\right| \omega_k \\
    &+ 2 \sum_{k = N+1}^{2N} \frac{1}{2k} \left|\left[\pi_\alpha (\cG(\bar{x})- \bar{\cG}(\bar{x}) + \bar{\cG}(\bar{x}))\right]_k\right| \nu^k \\
    &\le \sum_{k = 0}^{N} \sum_{\beta \in \{L, \theta, u \}} \left| \left[A^{(N)}_{\alpha, \beta} \pi_\beta (\cG^{(N)}(\bar{x})- \bar{\cG}^{(N)}(\bar{x}))\right]_k\right| \omega_k 
    + \sum_{k = 0}^{N} \sum_{\beta \in \{L, \theta, u \}} \left| (A^{(N)}_{\alpha, \beta} \pi_\beta  \bar{\cG}^{(N)}(\bar{x}))_k\right| \omega_k \\
    &+ 2 \sum_{k = N+1}^{2N} \frac{1}{2k} \left|\left[\pi_\alpha (\cG(\bar{x})- \bar{\cG}(\bar{x}))\right]_k\right| \nu^k 
    + 2 \sum_{k = N+1}^{2N} \frac{1}{2k} \left|(\pi_\alpha \bar{\cG}(\bar{x}))_k\right| \nu^k.
\end{split}
}
\]
We will simply the first and third term in this expression individually. For the first, we have
\[
{\small
\begin{split}
    \sum_{k = 0}^{N} \sum_{\beta \in \{L, \theta, u \}} \left| \left[A^{(N)}_{\alpha, \beta} \pi_\beta (\cG^{(N)}(\bar{x})- \bar{\cG}^{(N)}(\bar{x}))\right]_k\right| \omega_k
    &= \sum_{k = 0}^{N} \sum_{\beta \in \{L, \theta, u \}} \left| \left[A^{(N)}_{\alpha, \beta} \pi_\beta (\cG^{(N)}(\bar{x})- \bar{\cG}^{(N)}(\bar{x}))\right]_k\right| \omega_k \\
    &= \sum_{k = 0}^{N} \sum_{j=1}^{4} \left| \left[(A^{(N)}_{\alpha, u_j})_{k,0} (P_j(\bar{\theta}) - \bar{P}_j(\bar{\theta}))\right]_k\right| \omega_k \\
    &\le \sum_{j=1}^{4} \sum_{k = 0}^{N} \left| (A^{(N)}_{\alpha, u_j})_{k,0} (P_j(\bar{\theta}) - \bar{P}_j(\bar{\theta}))\right| \omega_k \\
    &\le \sum_{j=1}^{4} \sum_{k = 0}^{N} \left| (A^{(N)}_{\alpha, u_j})_{k,0}\right| \left| (P_j(\bar{\theta}) - \bar{P}_j(\bar{\theta}))\right| \omega_k \\
    &\le r_{\rm manif}\sum_{j=1}^{4} \sum_{k = 0}^{N} \left| (A^{(N)}_{\alpha, u_j})_{k,0}\right| \omega_k 
    = r_{\rm manif}\sum_{j=1}^{4} \norm{(A^{(N)}_{\alpha, u_j})_{k,0}}_{\pi^{(N)}\cC_\nu},
\end{split}
}
\]
where the last inequality comes from Lemma~\ref{lem:parameterization_properties}. As for the third term, it is zero since $\pi_\alpha(\cG(\bar{x}))_k = \pi_\alpha(\bar{\cG}(\bar{x}))_k$ when $k \ge 1$. Bringing everything together, we obtain
\begin{align*}
    \norm{\pi_\alpha A\cG(\bar{x})}_{\pi_\alpha \cX} \le Y_0^{(\alpha)} &\bydef
    r_{\rm manif}\sum_{j=1}^{4} \norm{(A^{(N)}_{\alpha, u_j})_{k,0}}_{\pi^{(N)}\cC_\nu} 
    + \sum_{k = 0}^{N} \sum_{\beta \in \{L, \theta, u \}} \left| (A^{(N)}_{\alpha, \beta} \pi_\beta  \bar{\cG}^{(N)}(\bar{x}))_k\right| \omega_k \\
& \qquad + 2 \sum_{k = N+1}^{2N} \frac{1}{2k} \left|(\pi_\alpha \bar{\cG}(\bar{x}))_k\right| \nu^k,
\end{align*}
for $\alpha \in \{u_1, u_2, u_3, u_4\}$. Now simply let
\[
    Y_0 = \max_{\alpha \in \{ L, \theta, u \}} \mu_\alpha Y^{(\alpha)}_0.
\]
By construction, we then have
\[
\norm{A\cG(\bar{x})}_\cX \le Y_0.
\]
\subsubsection{The \boldmath $Z_0$ \unboldmath Bound}
To begin, we define $B = I - AA^\dagger$ block-wise as
\[
    B = \begin{pmatrix}
        B_{L,L}           & B_{L,\theta}           & B_{L, u} \\
        B_{\theta,L}      & B_{\theta,\theta}      & B_{\theta, u}      \\
        B_{u,L}           & B_{u,\theta}           & B_{u, u}      \\
    \end{pmatrix},
\]
in the same fashion as $A^\dagger, A$. Note that each block has only finite nonzero part by definition of the tails of $A_{u_j,u_j}$ and $A^\dagger_{u_j, u_j}$ in the same way as the manifold proof in Section \ref{sec:bounds_stable_manifold}. Now let $h = (h_L, h_\theta, h_u) \in \cX$ with $\norm{h}_\cX \le 1$ and consider $\norm{Bh}_\cX$.
\[
\norm{\pi_\alpha Bh}_{\pi_\alpha \cX} = \norm{\sum_{\beta \in \{ L, \theta, u\}} B_{\alpha, \beta} h_\beta}_{\pi_\alpha \cX} 
\le \sum_{\beta \in \{ L, \theta, u\}} \norm{B_{\alpha, \beta} h_\beta}_{\pi_\alpha \cX} 
\le Z^{(\alpha)}_0 \bydef \sum_{\beta \in \{ L, \theta, u\}} \frac{1}{\mu_\beta}\norm{B_{\alpha, \beta}}_{B(\pi_\beta \cX, \pi_\alpha \cX)}
\]
since $||h_\beta|| \le \frac{1}{\mu_\beta}$. We can simply compute
\[
    Z_0 = \max_{\alpha \in \{L, \theta, u\}} \mu_\alpha Z^{(\alpha)}_0,
\]
where the infinite sums and suprema are replaced by finite sums from 0 to $N$ and maximums over $0, \dots, N$. By construction, we then have
\[
    \norm{I - AA^\dagger}_{B(\cX)} \le Z_0.
\]
\subsubsection{The \boldmath $Z_1$ \unboldmath Bound}

The begin, let $h = (h_L, h_\theta, h_u) \in \cX$ with $\norm{h}_\cX \le 1$. We then let $z \bydef [D\cG(\bar{x}) - A^\dagger]h$. Note that

\[
{\small
    z_\alpha = \begin{cases}
        0, & \alpha = L, \\
        2 \sum_{k \ge N+1} (h_{u_1})_k - 2 \sum_{k \ge N+1} (h_{u_3})_k, & \alpha = \theta_1, \\
        2 \sum_{k \ge N+1} (h_{u_2})_k + 2 \sum_{k \ge N+1} (h_{u_4})_k, & \alpha = \theta_2, \\
    \end{cases}
    }
\]
and
\[
{\small
    (z_{u_j})_k = \begin{cases}
         2 \sum_{k \ge N+1} (-1)^k (h_{u_j})_k + \left(D_{\theta_1}P(\bar{\theta}) - D_{\theta_1}\bar{P}(\bar{\theta})\right)h_{\theta_1} \\+ \left(D_{\theta_2}P(\bar{\theta}) - D_{\theta_2}\bar{P}(\bar{\theta})\right)h_{\theta_2}, &k = 0, \\
        \bar{L}(\psi_{u_j})_k, &1 \le k \le N,\\
         (\Upsilon\varphi_{u_j})_k, & k > N,
    \end{cases}
    }
\]
for $j = 1,2,3,4$, where
\[
{\small
    (\psi_{u_j})_k = \begin{cases}
        \begin{cases}
            0, &1 \le k < N,\\
            -(h_{u_2})_{k+1}, &k = N, \\
        \end{cases} & j = 1,\\
        \begin{cases}
            2\sigma^2 \rho (\Upsilon(\bar{u}_3*h_{u_3}^{(\infty)}))_k, &1 \le k < N,\\
            \sigma^2 \left[(-h_{u_1})_{k+1} + (1+\rho)(h_{u_3})_{k+1} + 2\rho (\Upsilon(\bar{u}_3*h_{u_3}^{(\infty)}))_k \right], &k = N,\\
        \end{cases} & j = 2,\\
        \begin{cases}
            0, & 1 \le k < N,\\
            -(h_{u_2})_{k+1},& k = N, \\
        \end{cases} & j = 3,\\
        \begin{cases}
            2\sigma^2\rho (\Upsilon(\bar{u}_1*h_{u_1}^{(\infty)}))_k, & 1 \le k < N,\\
            \sigma^2 \left[(-h_{u_3})_{k+1} + (1+\rho)(h_{u_1})_{k+1} + 2\rho (\Upsilon(\bar{u}_1*h_{u_1}^{(\infty)}))_k \right],& k = N,\\
        \end{cases} & j = 4,\\
    \end{cases}  
    }  
\]
and
\[
{\small
    \varphi_{u_j} = \begin{cases}
        u_2h_L + \bar{L}h_{_2},& j = 1,\\
        \sigma^2 (\left[\bar{u}_1 - (1+\rho)\bar{u}_3 + \rho \bar{u}_3^2 \right]h_L+ \bar{L}\left[h_{u_1} - (1+\rho)h_{u_3} + 2\rho \bar{u}_3*h_{u_3}\right]),& j = 2,\\
        \bar{u}_4h_L + \bar{L}h_{u_4},& j = 3,\\
        \sigma^2(\left[\bar{u}_3 - (1+\rho)\bar{u}_1 + \rho \bar{u}_1^2 \right]h_L + \bar{L}\left[h_{u_3} - (1+\rho)h_{u_1} + 2\rho \bar{u}_1*h_{u_1}\right]),& j = 4,\\
    \end{cases}  
    }  
\]
These forms can be found by going term by term comparing $D\cG(\bar{x}) h$ and $A^\dagger h$. Indeed, notice that the $2k u_k$ term is exactly copied by definition of $A^\dagger$. Now, one can bound $|(z_{u_j})_k|$ for $k = 0, \dots, N$ as follows. First, for $k = 0$, we have
\[
{\small
\begin{split}
    \left|(z_{u_j})_0\right| &= \left|2 \sum_{k \ge N+1} (-1)^k (h_{u_j})_k + \left(D_{\theta_1}P_j(\bar{\theta}) - D_{\theta_1}\bar{P}_j(\bar{\theta})\right)h_{\theta_1} + \left(D_{\theta_2}P_j(\bar{\theta}) - D_{\theta_2}\bar{P}_j(\bar{\theta})\right)h_{\theta_2}\right| \\
    &\le \left| 2 \sum_{k \ge N+1} (-1)^k (h_{u_j})_k \right| + \left|D_{\theta_1}P_j(\bar{\theta}) - D_{\theta_1}\bar{P}_j(\bar{\theta})\right||h_{\theta_1}| + \left|D_{\theta_2}P(\bar{\theta}) - D_{\theta_2}\bar{P}_j(\bar{\theta})\right| |h_{\theta_2}|.
\end{split}
}
\]
We will handle each of these terms individually. First, since $\nu \ge 1$,
\[
{\small
\begin{split}
    \left| 2 \sum_{k \ge N+1} (-1)^k (h_{u_j})_k \right| &\le 2 \sum_{k \ge N+1} \left|(h_{u_j})_k \right| \le \frac{1}{\nu^{N+1}}2 \sum_{k \ge N+1} \left|(h_{u_j})_k \right| \nu^k \\
    &\le \frac{1}{\nu^{N+1}} \left[ \left| (h_{u_j})_0 \right| + 2 \sum_{k \ge N+1} \left|(h_{u_j})_k \right| \nu^k \right] 
    = \frac{1}{\nu^{N+1}} \norm{h_{u_j}}_{\cC_\nu} 
    \le \frac{1}{\nu^{N+1}} \frac{1}{\mu_{u_j}}.
\end{split}
}
\]
Second,
\[
{\small
\begin{split}
    \left|D_{\theta_1}P_j(\bar{\theta}) - D_{\theta_1}\bar{P}_j(\bar{\theta})\right||h_{\theta_1}| &= |h_{\theta_1}| \left |\sum_{|\alpha|_1 \ge 0} \alpha_1\tilde{a}_\alpha \bar{\theta}^\alpha - \sum_{|\alpha|_1 \ge 0} \alpha_1\bar{a}_\alpha \bar{\theta}^\alpha\right|
    \le \frac{1}{\mu_{\theta_1}} \sum_{|\alpha|_1 \ge 0} \alpha_1|\tilde{a}_\alpha - \bar{a}_\alpha| |\bar{\theta}|^\alpha \\
    &\le \frac{r_{\rm manif}}{\mu_{\theta_1}} \sum_{|\alpha|_1 \ge 0} \alpha_1 |\bar{\theta}|^\alpha 
    = \frac{r_{\rm manif}}{\mu_{\theta_1}} \sum_{\alpha_1 \ge 0} \sum_{\alpha_2 \ge 0} \alpha_1 |\bar{\theta}_1|^{\alpha_1} |\bar{\theta}_2|^{\alpha_2}\\
    &= \frac{r_{\rm manif}}{\mu_{\theta_1}} \left(\sum_{\alpha_1 \ge 0} \alpha_1 |\bar{\theta}_1|^{\alpha_1} \right) \left(\sum_{\alpha_2 \ge 0} |\bar{\theta}_2|^{\alpha_2} \right)
    = \frac{r_{\rm manif}}{\mu_{\theta_1}} \frac{|\bar{\theta}_1|}{(1-\bar{|\theta}_1|)^2} \frac{1}{1-|\bar{\theta}_2|},
\end{split}
}
\]
where the last equality comes from studying various geometric series. In a similar fashion, one obtains
\[
\left|D_{\theta_2}P_j(\bar{\theta}) - D_{\theta_2}\bar{P}_j(\bar{\theta})\right||h_{\theta_1}| \le \frac{r_{\rm manif}}{\mu_{\theta_2}} \frac{|\bar{\theta}_2|}{(1-|\bar{\theta}_2|)^2} \frac{1}{1-|\bar{\theta}_1|}.
\]
We then define
\[
(\hat{z}_{u_j})_0 \bydef \frac{1}{\nu^{N+1}} \frac{1}{\mu_{u_j}} + \frac{r_{\rm manif}}{\mu_{\theta_1}} \frac{\bar{\theta}_1}{(1-\bar{\theta}_1)^2} \frac{1}{1-\bar{\theta}_2} + \frac{r_{\rm manif}}{\mu_{\theta_2}} \frac{\bar{\theta}_2}{(1-\bar{\theta}_2)^2} \frac{1}{1-\bar{\theta}_1},
\]
such that $|(z_{u_j})_0| \le (\hat{z}_{u_j})_0$. Using Lemma~\ref{THM: dual estimates}, and recalling the definition of $\Psi_k$, for $k = 1, \dots, N$, we have
\[
    |(z_{u_j})_k| \le (\hat{z}_{u_j})_k \bydef L(\hat{\psi}_{u_j})_k
\]
where
\[
    (\hat{\psi}_{u_j})_k = \begin{cases}
        \begin{cases}
            0, & 1 \le k < N,\\
            \frac{1}{2 \mu_{u_2}\nu^{k+1}},& k = N, \\
        \end{cases} & j = 1,\\
        \begin{cases}
            \frac{2\sigma^2 \rho }{\mu_{u_3}} (|\Upsilon|\Psi(\bar{u}_3))_k, & 1 \le k < N,\\
            \sigma^2 \left[\left(\frac{1}{\mu_{u_1}} + \frac{1+\rho}{\mu_{u_3}} \right) \frac{1}{2\nu^{k+1}} + \frac{2 \rho}{\mu_{u_3}} (|\Upsilon|\Psi(\bar{u}_3))_k \right], & k = N,\\
        \end{cases} & j = 2,\\
        \begin{cases}
            0, & 1 \le k < N,\\
            \frac{1}{2\mu_{u_4}\nu^{k+1}},& k = N, \\
        \end{cases} & j = 3,\\
        \begin{cases}
            \frac{2\sigma^2 \rho}{\mu_{u_1}} (|\Upsilon|\Psi(\bar{u}_1))_k, & 1 \le k < N,\\
            \sigma^2 \left[\left(\frac{1}{\mu_{u_3}} + \frac{1+\rho }{\mu_{u_1}} \right) \frac{1}{2\nu^{k+1}} + \frac{2\rho }{\mu_{u_1}} (|\Upsilon|\Psi(\bar{u}_1))_k \right], & k = N,\\
        \end{cases} & j = 4\\
    \end{cases}    
\]
where $|\Upsilon|$ is the operator $\Upsilon$ with all coefficients replaced by their absolute value, and $\Psi \in B(\cC_\nu)$ is defined element-wise by
\[
\Psi(u)_k \bydef \Psi_k(u), \quad 0 \le k \le N+1.
\]
Then we bound $|z_L|, |z_\theta|$ by
\[
    \hat{z}_L = 0, \quad \hat{z}_{\theta_1} = \left(\frac{1}{\mu_{u_1}} + \frac{1}{\mu_{u_3}}\right)\frac{1}{\nu^{(N+1)}}, \quad \hat{z}_{\theta_2} = \left(\frac{1}{\mu_{u_2}} + \frac{1}{\mu_{u_4}}\right) \frac{1}{\nu^{(N+1)}}
\]
in the same fashion as $|(z_{u_j})_0|$. We can now build $\hat{z}$ such that it bounds $z$ component-wise. Recall that $z = [D\cG(\bar{x}) - A^\dagger]h$ and now define $w \bydef Az$ as
\[
    w = (w_L, w_\theta, w_u).
\]
Now for $\alpha \in \{L, \theta_1, \theta_2 \}$ we have
\[
    |w_\alpha| = |\pi_\alpha A^{(N)}z^{(N)}| \le Z_1^{(\alpha)} \bydef \pi_\alpha|A^{(N)}|\hat{z}^{(N)},
\]
since their operators do not have an infinite tail. For $j \in \{1,2,3,4\}$, we have
\[
{\small
    \begin{aligned}
        \norm{w_{u_j}}_{\cC_\nu} &= \sum_{k = 0}^{N} |(A^{(N)}z^{(N)})_k| \omega_k + \sum_{k \ge N+1} \frac{1}{2k} |(z_{u_j})_k| \omega_k \\
         &\le \sum_{k = 0}^{N} (|A^{(N)}| \hat{z}^{(N)})_k \omega_k + \frac{1}{2(N+1)}\sum_{k \ge N+1} |(\Upsilon\varphi_{u_j})_k| \omega_k \\
         &\le \sum_{k = 0}^{N} (|A^{(N)}| \hat{z}^{(N)})_k \omega_k + \frac{1}{2(N+1)}\norm{\Upsilon\varphi_{u_j}}_{\cC_\nu} \\
         &\le \sum_{k = 0}^{N} (|A^{(N)}| \hat{z}^{(N)})_k \omega_k + \frac{1}{2(N+1)}\norm{\Upsilon}_{B(\cC_\nu)}\norm{\varphi_{u_j}}_{\cC_\nu} \\
         &= \sum_{k = 0}^{N} (|A^{(N)}| \hat{z}^{(N)})_k \omega_k + \frac{\nu}{(N+1)}\norm{\varphi_{u_j}}_{\cC_\nu},
    \end{aligned}
    }
\]
where the last equality is obtained by Lemma~\ref{lem: upsilon operator norm}. Using the triangle inequality and the Banach algebra property, we obtain
\[
{\small
    \norm{\varphi_{u_j}}_{\cC_\nu} \le 
    \varphi_{u_j}^{(\infty)} = \begin{cases}
        \frac{\norm{\bar{u}_2}_{\cC_\nu}}{\mu_L} + \frac{\bar{L}}{\mu_{u_2}} &j = 1,\\
        \sigma^2 \left[ \frac{1}{\mu_L}\left(\norm{\bar{u}_1}_{\cC_\nu} + (1+\rho)\norm{\bar{u}_3}_{\cC_\nu} + \rho \norm{\bar{u}_3^2}_{\cC_\nu}\right) + \bar{L} \left(\frac{1}{\mu_{u_1}} + \frac{1+\rho }{\mu_{u_3}} + \frac{2 \rho \norm{\bar{u}_3}_{\cC_\nu}}{\mu_{u_3}}\right)) \right]&j = 2,\\
        \frac{\norm{\bar{u}_4}_{\cC_\nu}}{\mu_L} + \frac{\bar{L}}{\mu_{u_4}} &j = 3,\\
        \sigma^2 \left[ \frac{1}{\mu_L}\left(\norm{\bar{u}_3}_{\cC_\nu} + (1+\rho)\norm{\bar{u}_1}_{\cC_\nu} + \rho \norm{\bar{u}_1^2}_{\cC_\nu}\right) + \bar{L}\left(\frac{1}{\mu_{u_3}} + \frac{1+\rho }{\mu_{u_1}} + \frac{2\rho \norm{\bar{u}_1}_{\cC_\nu}}{\mu_{u_1}}\right) \right]&j = 4.
    \end{cases}
    }
\]
Therefore
\[
\norm{w_{u_j}}_{\cC_\nu} \le Z_1^{(u_j)} \bydef \sum_{k = 0}^{N} (|A^{(N)}| \hat{z}^{(N)})_k \omega_k + \frac{\nu}{(N+1)}\varphi_{u_j}^{(\infty)}.
\]
Hence, we set
\[
    Z_1 = \max_{\alpha \in \{ L, \theta, u\}} \mu_\alpha Z_1^{(\alpha)}.
\]
\subsubsection{The \boldmath $Z_2$ \unboldmath Bound}
Now we aim to compute some $Z_2( r )$ such that
\[
    \norm{A[D\cG(\bar{x} + b) - D\cG(\bar{x})]}_{B(\cX)} \le Z_2( r ) r , \quad \forall b \in B_r(0),  \quad 0 \le r \le r_*.
\]
To begin, let 
$h,b \in \cX$ with $\norm{h}_\cX = 1$ and $\norm{b}_\cX \le r$ consider $z \bydef [DG(\bar{x} + b) - DG(\bar{x})]h$. Then
\[
    \begin{aligned}
        z_L &= 2b_{\theta_1}h_{\theta_1} + 2b_{\theta_2}h_{\theta_2}, \\
        z_{\theta_1} &=  z_{\theta_2} = 0.
    \end{aligned}
\]
Moreover,
\[
    (z_{u_j})_k = \begin{cases}
        \left[D_\theta P(\bar{\theta}+b_\theta) - D_\theta P(\bar{\theta})\right]h_\theta,& k = 0, \\
        (\Upsilon(\Phi_{u_j}h_L + L\psi_{u_j}))_k & k > 0, \\
    \end{cases}
\]
for $j = 1,2,3,4$, where
\[
    \Phi_{u_j} = \begin{cases}
        b_{u_2} & j = 1, \\
        \sigma^2 \left[ b_{u_1} - (1+\rho)b_{u_3} + 2u_3b_{u_3} + b_{u_3}^2 \right]& j = 2, \\
        b_{u_4}& j = 3, \\
        \sigma^2 \left[ b_{u_3} - (1+\rho)b_{u_1} + 2u_1b_{u_1} + b_{u_1}^2 \right]& j = 4, 
    \end{cases}
    \qquad \text{and} \qquad
        \psi_{u_j} = \begin{cases}
        0 & j = 1, \\
        2\sigma^2 \rho b_{u_3}h_{u_3} & j = 2, \\
        0 & j = 3, \\
        2\sigma^2 \rho b_{u_1}h_{u_1} & j = 4. 
    \end{cases}
\]
Now, we compute an upper bound for $\norm{Az}_\cX$. Below, we will compute a bound for $\norm{z}_\cX$ which implies that $z \in \cX$. Therefore, we have $\norm{Az}_\cX \le \norm{A}_{B(\cX)} \norm{z}_\cX$. To begin, we have
\[
|z_L| \le 2\left( \frac{1}{\mu_{\theta_1}^2} + \frac{1}{\mu_{\theta_2}^2}\right)r  \quad \text{and} \quad |z_{\theta_1}| = |z_{\theta_2}| = 0 .
\]
So, we define
\[
    Z_2^{(L)} \bydef 2\left( \frac{1}{\mu_{\theta_1}^2} + \frac{1}{\mu_{\theta_2}^2}\right) \quad \text{and} \quad
    Z_2^{(\theta_1)}, Z_2^{(\theta_1)} \bydef 0.
\]
Next, for $j = 1, 2, 3, 4$ we have
\[
    \begin{aligned}
        \norm{z_{u_j}}_{\cC_\nu} &= |(z_{u_j})_0| + \sum_{k > 0} |(z_{u_j})_k| \omega^k \\
         &= |(z_{u_j})_0| + \sum_{k > 0} |(\Upsilon(\Phi_{u_j}h_L + L\psi_{u_j}))_k| \omega^k \\
         &\le |(z_{u_j})_0| + \norm{\Upsilon(\Phi_{u_j}h_L + L\psi_{u_j})}_{\cC_\nu} \\
         &\le |(z_{u_j})_0| + \norm{\Upsilon}_{B(\cC_\nu)} \left[\norm{\Phi_{u_j}h_L}_{\cC_\nu} + \norm{L\psi_{u_j}}_{\cC_\nu} \right]\\
         &\le |(z_{u_j})_0| + 2\nu \left[\frac{1}{\mu_L}\norm{\Phi_{u_j}}_{\cC_\nu} + L\norm{\psi_{u_j}}_{\cC_\nu} \right],
    \end{aligned}
\]
where the last inequality comes from the Banach algebra property of $\cC_\nu$.
It remains to find upper bounds for $|(z_{u_j})_0|, \norm{\Phi_{u_j}}_{\cC_\nu}, \norm{\psi_{u_j}}_{\cC_\nu}$. Recalling that $\norm{b}_\cX \le r \le r_*$, one can see that
\[
    \norm{\Phi_{u_j}}_{\cC_\nu} \le r \hat{\Phi}_{u_j} \bydef r  \begin{cases}
        \frac{1}{\mu_{u_2}}, & j = 1, \\
        \sigma^2 \left[ \frac{1}{\mu_{1}} + \frac{1+\rho}{\mu_{u_3}} + 2\frac{\norm{\bar{u}_3}_{\cC_\nu}}{\mu_{u_3}} + \frac{r_*}{\mu_{u_3}^2} \right],& j = 2, \\
        \frac{1}{\mu_{u_4}}, & j = 3, \\
        \sigma^2 \left[ \frac{1}{\mu_{3}} + \frac{1+\rho}{\mu_{u_1}} + 2\frac{\norm{\bar{u}_1}_{\cC_\nu}}{\mu_{u_1}} + \frac{r_*}{\mu_{u_1}^2} \right],& j = 4, \\
    \end{cases}
\]
and
\[
    \norm{\psi_{u_j}}_{\cC_\nu} \le \gamma\hat{\psi}_{u_j} \bydef r \begin{cases}
        0, & j = 1, \\
        \frac{2\sigma^2r}{\mu_{u_3}^2},& j = 2, \\
        0, & j = 3, \\
        \frac{2\sigma^2r}{\mu_{u_1}^2},& j = 4. \\
    \end{cases}
\]
As for $|(z_{u_j})_0|$, this will require more work. By the mean value inequality and letting 
\[
\varepsilon \bydef \frac{r}{\min\{\mu_{\theta_1}, \mu_{\theta_2}\}},
\]
we have
\[
{\small
\begin{split}
    |D_\theta P_j(\theta + b_\theta) - D_\theta P_j(\theta)| &\le \sup_{c \in B_{r \varepsilon}(\bar{\theta})} \norm{D^2 P_j(c; \tilde
    a)}_{B(\mathbb{R}^2, B(\mathbb{R}^2, \mathbb{R}))} ||b_\theta||_{\infty} \\
     &\le r \varepsilon \sup_{c \in B_{r \varepsilon}(\bar{\theta})} \norm{D^2_\theta P_j(c)}_{B(\mathbb{R}^2, B(\mathbb{R}^2, \mathbb{R}))}.\\
\end{split}
}
\]
Now
\[
{\small
\begin{split}
    \norm{D^2_\theta P_j(c)}_{B(\mathbb{R}^2, B(\mathbb{R}^2, \mathbb{R}))} &= \sup_{\norm{t}_\infty = 1} \sup_{\norm{s}_\infty = 1} |D^2_\theta P_j(c)ts|
      = \sup_{\norm{t}_\infty,  \norm{s}_\infty = 1} \left| \sum_{i,k = 1}^{2} \frac{\partial^2P_j(c)}{\partial \theta_i \partial\theta_k} t_k s_i\right| \\
      &\le  \left| \frac{\partial^2P_j(c)}{\partial \theta_1 ^2} + 2\frac{\partial^2P_j(c)}{\partial \theta_1 \partial\theta_2} + \frac{\partial^2P_j(c)}{\partial \theta_2 ^2}\right| \\
      &\le \left| \frac{\partial^2\bar{P}_j(c)}{\partial \theta_1 ^2} + 2\frac{\partial^2\bar{P}_j(c)}{\partial \theta_1 \partial\theta_2} + \frac{\partial^2\bar{P}_j(c)}{\partial \theta_2 ^2}\right| 
      + \left| \frac{\partial^2P_j(c)}{\partial \theta_1 ^2} - \frac{\partial^2\bar{P}_j(c)}{\partial \theta_1 ^2} \right| 
      + 2\left| \frac{\partial^2P_j(c)}{\partial \theta_1 \partial \theta_2} - \frac{\partial^2\bar{P}_j(c)}{\partial \theta_1 \partial \theta_2} \right| \\
      &\quad +\left| \frac{\partial^2P_j(c)}{\partial \theta_2 ^2} - \frac{\partial^2\bar{P}_j(c)}{\partial \theta_2 ^2} \right|.
\end{split}
}
\]
Note that the second derivatives of $\bar{P}_j(c)$ are computable since $\bar{a}$ is finite. As for the other differences, we start with 
\[
{\small
\begin{split}
    \left| \frac{\partial^2P_j(c)}{\partial \theta_1 ^2} - \frac{\partial^2\bar{P}_j(c)}{\partial \theta_1 ^2} \right| 
    &= \left| \sum_{|\alpha|_1 \ge 0} (\alpha_1 + 1)(\alpha_1 + 2) (\tilde{a}_{(\alpha_1 + 2, \alpha_2)}- \bar{a}_{(\alpha_1 + 2, \alpha_2)}) c^\alpha\right| \\
    &\le \sum_{|\alpha|_1 \ge 0} (\alpha_1 + 1)(\alpha_1 + 2) |\tilde{a}_{(\alpha_1 + 2, \alpha_2)}- \bar{a}_{(\alpha_1 + 2, \alpha_2)}| |c|^\alpha \\
    &\le r_{\rm manif} \sum_{|\alpha|_1 \ge 0} (\alpha_1 + 1)(\alpha_1 + 2) |c|^\alpha \\
    &\le r_{\rm manif} \sum_{\alpha_1 \ge 0} \sum_{\alpha_2 \ge 0}(\alpha_1 + 1)(\alpha_1 + 2) |c_1|^{\alpha_1}  |c_2|^{\alpha_2} \\
    &\le r_{\rm manif} \left(\sum_{\alpha_1 \ge 0} (\alpha_1 + 1)(\alpha_1 + 2) |c_1|^{\alpha_1}\right) \left( \sum_{\alpha_2 \ge 0}  |c_2|^{\alpha_2}  \right) 
    = r_{\rm manif} \frac{2}{(1-|c_1|)^3} \frac{1}{1-|c_2|},
\end{split}
}
\]
where the last equality comes from studying various geometric series. For the others, one obtains the bounds
\[
\begin{split}
    2\left| \frac{\partial^2P_j(c)}{\partial \theta_1 \partial \theta_2} - \frac{\partial^2\bar{P}_j(c)}{\partial \theta_1 \partial \theta_2} \right| &\le 2r_{\rm manif} \frac{1}{(1-|c_1|)^2} \frac{1}{(1-|c_2|)^2}, \\
    \left| \frac{\partial^2P_j(c)}{\partial \theta_2 ^2} - \frac{\partial^2\bar{P}_j(c)}{\partial \theta_2 ^2} \right| & \le r_{\rm manif} \frac{1}{1-|c_1|} \frac{2}{(1-|c_2|)^3},
\end{split}
\]
in a similar fashion. Using these calculations, we define the constants
\[
\begin{split}
     \zeta_1^{(u_j)}&\bydef \left| \frac{\partial^2\bar{P}_j(c)}{\partial \theta_1 ^2} + 2\frac{\partial^2\bar{P}_j(c)}{\partial \theta_1 \partial\theta_2} + \frac{\partial^2\bar{P}_j(c)}{\partial \theta_2 ^2}\right|, \\
     \zeta_2^{(u_j)}&\bydef \sup_{c \in B_{r \varepsilon}(\bar{\theta})} 2r_{\rm manif} \left( \frac{1}{(1-|c_1|)^3} \frac{1}{1-|c_2|} + \frac{1}{(1-|c_1|)^2} \frac{1}{(1-|c_2|)^2} + \frac{1}{(1-|c_1|)} \frac{1}{(1-|c_2|)^3}\right).
\end{split}
\]
With these constants, we let
\[
Z_2^{(u_j)} \bydef \varepsilon \left(\zeta_1^{(u_j)} + \zeta_2^{(u_j)}\right) + 2\nu \left[ \frac{1}{\mu_L} \hat{\Phi}_{u_j} + L \hat{\psi}_{u_j}\right],
\]
for $j = 1,2,3,4$. Finally, we let
\[
Z_2 \bydef \norm{A}_\cX \max_{\alpha \in \{L, \theta, u \}} \mu_\alpha Z_2^{(\alpha)}.
\]
It remains to compute $\norm{A}_{B(\cX)}$. To do this, recall that
\[
\norm{A}_{B(\cX)} \le \max_{\alpha \in \{L, \theta, u\}}  \mu_\alpha\sum_{\beta \in \{ L, \theta, u\}} \frac{1}{\mu_\beta}\norm{A_{\alpha, \beta}}_{B(\pi_\beta \cX, \pi_\alpha \cX)}.
\]
When the block $A_{\alpha, \beta}$ is finite, computing the norm is easy. When it is infinite, it requires a little bit of bookkeeping. Let us consider the infinite blocks $A_{u_j, u_j}$ for $j = 1,2,3,4$.
\[
{\small
\begin{split}
    \norm{A_{u_j,u_j}}_{B(\cC_\nu)} &= \sup_{\ell \ge 0} \left(\frac{1}{\omega_\ell}\sum_{k \ge 0} |(A_{u_j,u_j})_{k, \ell}| \omega_k \right) \\
    &= \max\left\{ \max_{0 \le \ell \le N} \left( \frac{1}{\omega_\ell}\sum_{k \ge 0} |(A_{u_j,u_j})_{k, \ell}| \omega_k \right), \sup_{\ell > N} \left(\frac{1}{\omega_\ell}\sum_{k \ge 0} |(A_{u_j,u_j})_{k, \ell}| \omega_k \right)\right\} \\
    &= \max\left\{ \max_{0 \le \ell \le N} \left(\frac{1}{\omega_\ell} \sum_{0 \le k \le N} |(A_{u_j,u_j}^{(N)})_{k, \ell}| \omega_k \right), \sup_{\ell > N} \frac{1}{2\ell}\right\} \\
    &= \max\left\{ \norm{A_{u_j,u_j}^{(N)}}_{B(\pi^{(N)}\cC_\nu)}, \frac{1}{2(N+1)}\right\}.
\end{split}
}
\]
The other infinite blocks $A_{u_i, u_j}$ for $i \ne j$ do not have the infinite tail, so are treated as finite blocks.

\subsection{Computer-assisted proof for the projected boundary value problem}
\label{sec:bvp proof}
In this subsection, we use the computable bounds presented in Section~\ref{sec:bvp_explicit_bounds} to compute a rigorous error bound on the solution $(\tilde{L}, \tilde{\theta}, \tilde{u})$ of $\cG(L, \theta, u) = 0$ as defined in \eqref{EQ:connecting_orbit_function}. To this end, let $N = 500$ and define the approximate solution $\bar{x} \bydef (\bar{L}, \bar{\theta}, \bar{u}) \in \R^3 \times \pi^{(N)} (\cC_\nu)^4$ whose first coefficients are given by Table~\ref{tab:bvp_coefficients}.
\begin{table}[h]
    \centerline{
    {\tiny
    \begin{tabular}{|c|c|c|c|c|c|c|c|}
    \hline
    $\bar{L}$ & $0.1762980548$ & $\bar{\theta}_1$ & $-0.7294203070$ & $\bar{\theta}_2$ & $0.6464874443$ &&\\
    \hline
    $(\bar{u}_1)_0$& $1.081354585$ & $(\bar{u}_2)_0$ & $0.4104976985$ & $(\bar{u}_3)_0$ & $0.8901579593$ & $(\bar{u}_4)_0$ & $-0.5394882617$\\
    \hline
    $(\bar{u}_1)_1$& $-3.621204648\times 10^{-2}$ & $(\bar{u}_2)_1$ & $0.1338049704$ & $(\bar{u}_3)_1$ & $4.980712484\times 10^{-2}$ & $(\bar{u}_4)_1$ & $-9.358238514\times 10^{-2}$\\
    \hline
    $(\bar{u}_1)_2$& $-6.202797003\times 10^{-3}$ & $(\bar{u}_2)_2$ & $-3.071344384\times 10^{-4}$ & $(\bar{u}_3)_2$ & $4.447448886\times 10^{-3}$ & $(\bar{u}_4)_2$ & $2.554491345\times 10^{-2}$\\
    \hline
    $(\bar{u}_1)_3$& $-1.718834821\times 10^{-5}$ & $(\bar{u}_2)_3$ & $-6.929356063\times 10^{-3}$ & $(\bar{u}_3)_3$ & $-7.613899801\times 10^{-4}$ & $(\bar{u}_4)_3$ & $7.325112475\times 10^{-3}$\\
    \hline
    $(\bar{u}_1)_4$& $1.560628361 \times 10^{-4}$ & $(\bar{u}_2)_4$ & $-8.921102022\times 10^{-4}$ & $(\bar{u}_3)_4$ & $-1.678731718\times 10^{-4}$ & $(\bar{u}_4)_4$ & $-3.676803394\times 10^{-4}$\\
    \hline
    $(\bar{u}_1)_5$& $1.652754362\times 10^{-5}$ & $(\bar{u}_2)_5$ & $1.524162831\times 10^{-4}$ & $(\bar{u}_3)_5$ & $6.253534298\times 10^{-6}$ & $(\bar{u}_4)_5$ & $-2.925857256\times 10^{-4}$\\
    \hline
    $(\bar{u}_1)_6$& $-2.243112874\times 10^{-6}$ & $(\bar{u}_2)_6$ & $4.536716506\times 10^{-5}$ & $(\bar{u}_3)_6$ & $4.409649493\times 10^{-6}$ & $(\bar{u}_4)_6$ & $-1.296659605\times 10^{-5}$\\
    \hline
    $(\bar{u}_1)_7$& $-5.879251489\times 10^{-7}$ & $(\bar{u}_2)_7$ & $-2.646668330\times 10^{-7}$ & $(\bar{u}_3)_7$ & $1.755548932\times 10^{-7}$ & $(\bar{u}_4)_7$ & $7.563892997\times 10^{-6}$\\
    \hline
    $(\bar{u}_1)_8$& $1.724141013\times 10^{-9}$ & $(\bar{u}_2)_8$ & $-1.320542829\times 10^{-6}$ & $(\bar{u}_3)_8$ & $-8.467653327\times 10^{-8}$ & $(\bar{u}_4)_8$ & $9.743887558\times 10^{-7}$\\
    \hline
    $(\bar{u}_1)_9$& $1.317980398\times 10^{-8}$ & $(\bar{u}_2)_9$ & $-1.081917304\times 10^{-7}$ & $(\bar{u}_3)_9$ & $-9.879385589\times 10^{-9}$ & $(\bar{u}_4)_9$ & $-1.209594179\times 10^{-7}$\\
    \hline
    \end{tabular}
    }
    }
    \caption{First few coefficients of approximate solution to \eqref{EQ:connecting_orbit_function} with 10 digits of precision.}
    \label{tab:bvp_coefficients}
\end{table}
\begin{theorem} \label{thm:bvp_proof}
    Let $\nu = 1.05$ and $\cG: \cX \to \cY$ be defined as in \eqref{EQ:connecting_orbit_function} with parameter values $\sigma = 10$ and $\rho = 2.2$. Then there exists $\tilde{x} \bydef (\tilde{L}, \tilde{\theta}, \tilde{u}) \in \cX$ such that $\cG(\tilde{x}) = 0$. Moreover,
    \[
    \norm{\tilde{x} - \bar{x}}_\cX \le r_{\rm bvp} \bydef 4.189559197816045 \times 10^{-12},
    \]
    where $\bar{x}$ is as defined above and 
    \[
    \mu=(\sigma, 1, 1, \sigma, 1, \sigma, 1).
    \]
\end{theorem}
\begin{proof}
    We compute the bounds developed in section \ref{sec:bvp_explicit_bounds} with $N = 500$ and $r_* = 1 \times 10^{-5}$ to obtain the numbers found in Table~\ref{tab:bvp_bounds}.
    This gives us a range     
    \[
    r \in (2.2580382640246346 \times 10^{-12}, 1 \times 10^{-5}),
    \]
    where the Newton-Kantorovich is less than 0 and $r \le r_*$. Hence, by the Newton-Kantorovich Theorem, there exists a unique $\tilde{x} \in \cX$ such that $\norm{\tilde{x} - \bar{x}}_X \le 4.189559197816045 \times 10^{-12}$ and $\cG(\tilde{x}) = 0$.    
\end{proof}
\begin{table} [h]
\begin{center}
    \begin{tabular}{|c|c|} 
        \hline
        $Y_0$& $1.778137191602884 \times 10^{-12}$\\
        \hline
        $Z_0$& $1.5654082895539637 \times 10^{-13}$ \\
        \hline
        $Z_1$& $0.21252973518094204$ \\
        \hline
        $Z_2$& $163.0173122663094$ \\
         \hline
    \end{tabular}
    \caption{Newton-Kantorovich bounds for the projected boundary problem proof.}
    \label{tab:bvp_bounds}
\end{center}
\end{table}
\begin{lem} \label{lem:bvp_domain_of_convergence}
    Let $(\tilde{L}, \tilde{\theta}, \tilde{u}), (\bar{L}, \bar{\theta}, \bar{u})$ be as defined in Theorem~\ref{thm:bvp_proof} and 
    \[
    \Gamma_j(t) \bydef (\tilde{u}_j)_0 + 2 \sum_{k \ge 1} (\tilde{u}_j)_k T_k\left( 1 - \frac{t}{\tilde{L}}\right), \quad \bar{\Gamma }_j(t) \bydef (\bar{u}_j)_0 + 2 \sum_{k = 1}^{N} (\bar{u}_j)_k T_k\left( 1 - \frac{t}{\tilde{L}}\right).
    \]
    Then $\Gamma (t)$ solves \eqref{eq:reduced_problem} with $L$ set to $\tilde{L}$. Moreover, for $j=1,2,3,4$ and $t \in [0, 2\tilde{L}]$
    \[
    |\Gamma _j(t) -\bar{\Gamma }_j(t)| \le \frac{r_{\rm bvp}}{\mu_{u_j}}.
    \]
\end{lem}
\begin{proof}
    We begin by ensuring that $\Gamma (t)$ is well-defined on $[0, 2\tilde{L}]$. For $j = 1,2,3,4$, and $t \in [0, 2\tilde{L}]$, we have
    \[
    \Gamma _j(t)  = (\tilde{u}_j)_0 + 2 \sum_{k \ge 1} (\tilde{u}_j)_k T_k\left(1- \frac{t}{\tilde{L}}\right),
    \]
    and
    \[
        |(\tilde{u}_j)_0| + \sum_{k \ge 1} |2(\tilde{u}_j)_k| \left|T_k\left(1- \frac{t}{L}\right)\right| \le \sum_{k \ge 0} |(\tilde{u}_j)_k| \omega_k = \norm{\tilde{u}_j}_{\cC_\nu} \le \norm{\tilde{u}_j - \bar{u}_j}_{\cC_\nu} + \norm{\bar{u}_j}_{\cC_\nu} < \infty.
    \]
    Therefore $\Gamma (t)$ is absolutely convergent and hence converges. By construction of $\cG$, we have that $y_j(t) \bydef (\tilde{u}_j)_0 + 2 \sum_{k \ge 1} (\tilde{u}_j)_k T_k(t)$ solves
    \[
    \begin{cases}
        y(1) \in \Pi_R, \\
        y(-1) = P(\tilde{\theta}), \\
        \dot{y} = -\tilde{L}f(y), \quad t \in [-1, 1].
    \end{cases}
    \]
    Then $\Gamma (t)$ simply rescales time back to the interval $[0, 2\tilde{L}]$ and reverses it, so $\Gamma (t)$ solves 
    \[
    \begin{cases}
        \Gamma (0) \in \Pi_R, \\
        \Gamma (2\tilde{L}) = P(\tilde{\theta}), \\
        \dot{\Gamma } = f(\Gamma ), \quad t \in [0, 2\tilde{L}].
    \end{cases}
    \]
    Finally, by Theorem~\ref{lem:parameterization_properties} we have $P(\tilde{\theta}) \in W^s_\text{loc}(\tilde{x}^{(+)})$ since $\tilde{\theta} \in [-1,1]^2$. As for the inequality, for $j = 1,2,3,4$ and $t \in [0, 2\tilde{L}]$, we have
    \[
    {\small
    \begin{split}
        |\Gamma _j(t) -\bar{\Gamma }_j(t)| &= \left| (\tilde{u}_j)_0 + 2 \sum_{k \ge 1} (\tilde{u}_j)_k T_k\left( 1 - \frac{1}{\tilde{L}}\right) - (\bar{u}_j)_0 - 2 \sum_{k = 1}^{N} (\bar{u}_j)_k T_k\left( 1 - \frac{1}{\tilde{L}}\right)\right| \\
        &= \left| (\tilde{u}_j)_0 - (\bar{u}_j) + 2 \sum_{k \ge 1} ((\tilde{u}_j)_k - (\bar{u}_j)_k) T_k\left( 1 - \frac{1}{\tilde{L}}\right)\right| \\
        &\le |(\tilde{u}_j)_0 - (\bar{u}_j)_0| + 2 \sum_{k \ge 1} \left|(\tilde{u}_j)_k - (\bar{u}_j)_k\right| \left| T_k\left( 1 - \frac{1}{\tilde{L}}\right)\right| \\
        &\le \sum_{k \ge 0} \left| (\tilde{u}_j)_k - (\bar{u}_j)_k \right| \omega_k \\
        &= \norm{\tilde{u}_j - \bar{u}_j}_{\cC_\nu} \le \frac{r_{bvp}}{\mu_{u_j}},
    \end{split}    
    }
    \]
    where the last inequality follows from Theorem~\ref{thm:bvp_proof}.
\end{proof}

\begin{figure}[H] \label{fig:original_bvp}
\centering
    \includegraphics[width=.6\textwidth]{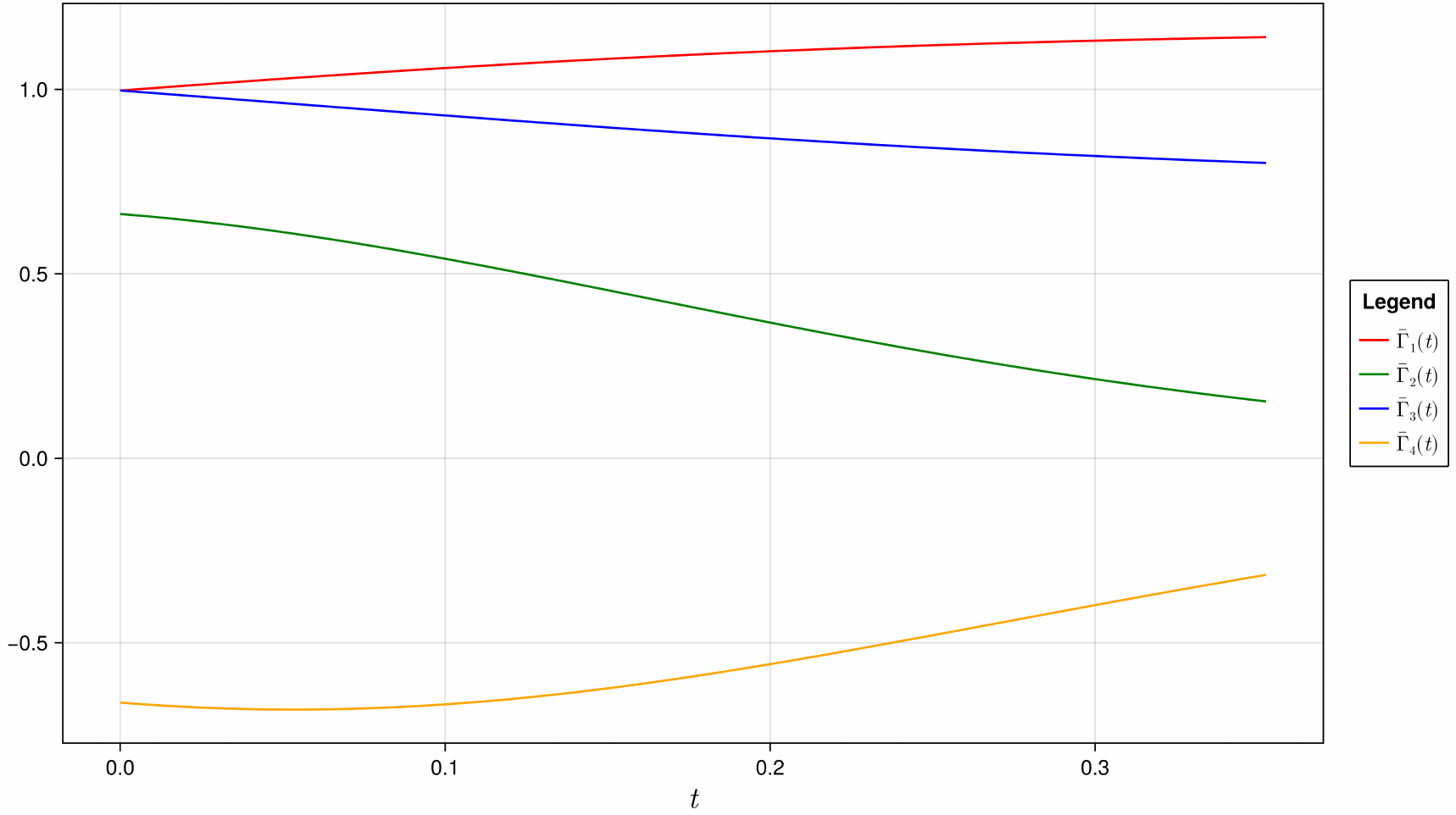}
    \vspace{-.3cm}
    \caption{Solution to the projected BVP when $\sigma=10$ and $\rho=2.2$.
    Plot of the four coordinates of $\bar{\Gamma}(t)$ for $t \in [0, 2\bar{L}]$. One notices that $\bar{\Gamma}(0) \approx R\bar{\Gamma}(0)$ as expected.}
\end{figure}

\subsection{Proof for the 2-cycle with the logisitc growth function}
\label{sec:results}
In this subsection, we present the main result of the thesis, that is, the verification of the spatially inhomogeneous two-cycle in \eqref{IDE1}. This will follow from the results obtained in Sections \ref{sec:manifold_proof} and \ref{sec:bvp proof}. To begin, let $C(\R)$ be the space of continuous functions from $\R$ to $\R$ with the supremum norm. Next, let $\bar{N}(t), \bar{M}(t) \in C(\R)$ be the following piecewise-defined functions.
\[
\bar{N}(t) \bydef \begin{cases}
    \bar{\zeta}_3(-t), & t < -2\tilde{L}, \\
    \bar{\Gamma }_3(-t), &-2\tilde{L} \le t < 0, \\
    \bar{\Gamma }_1(t),& 0 \le t \le 2\tilde{L}, \\
    \bar{\zeta}_1(t),& t > 2\tilde{L},   
\end{cases}
\quad
\bar{M}(t) \bydef \begin{cases}
    \bar{\zeta}_1(-t), & t < -2\tilde{L}, \\
    \bar{\Gamma }_1(-t), &-2\tilde{L} \le t < 0, \\
    \bar{\Gamma }_3(t),& 0 \le t \le 2\tilde{L}, \\
    \bar{\zeta}_3(t),& t > 2\tilde{L},   
\end{cases}
\]
where $\bar{\zeta}(t) \bydef \bar{P}\left(e^{\Lambda(t - 2\tilde{L})} \tilde{\theta}\right)$ and $\bar{P}, \tilde{L}, \tilde{\theta}, \bar{\Gamma }$ are defined in Theorems~\ref{thm:manif_proof} and \ref{thm:bvp_proof} and Lemmas~\ref{lem:parameterization_properties} and \ref{lem:bvp_domain_of_convergence}.
Notice that $\bar{N}(t), \bar{M}(t)$ satisfy symmetry $N(t) = M(-t)$ observed in the ordinary differential equation.
\begin{figure}[H] \label{fig:n_m}
\centering
    \includegraphics[width=.6\textwidth]{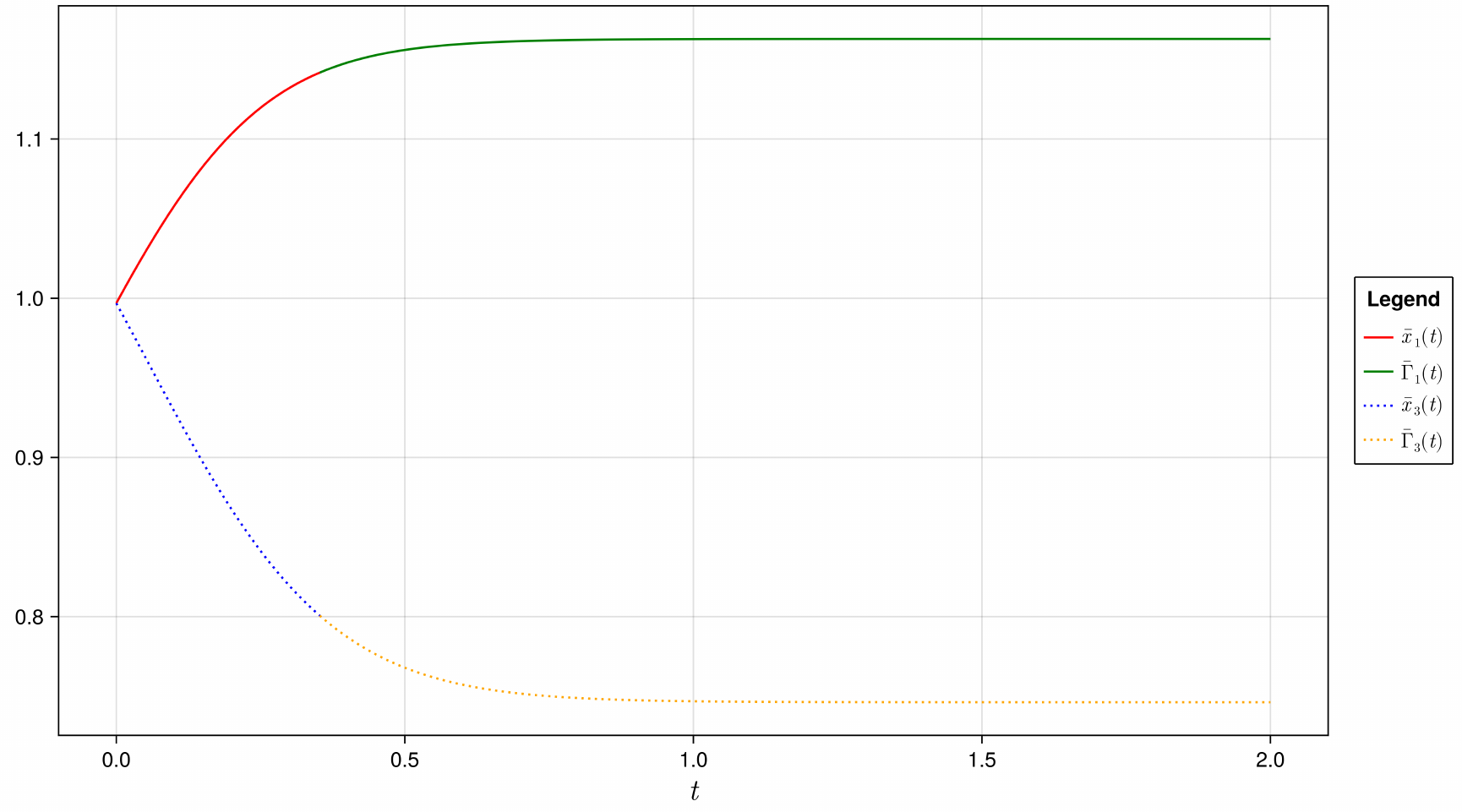}
    \vspace{-.3cm}
    \caption{Plot of $\bar{N}(t), \bar{M}(t)$ for $t \ge 0$. $\bar{N}(t)$ appears as a solid line while $\bar{M}(t)$ appears as a dotted line. The color and legend is to precisely show how these functions are constructed from their pieces. For negative values of $t$, recall the symmetry presented above.}
\end{figure}
\begin{theorem} \label{thm:two-cycle_proof}
    Let $\bar{N}(t), \bar{M}(t)$ be as above. There exist functions $N(t), M(t)$ that form a spatially inhomogeneous two-cycle of \eqref{IDE1} with parameter values $\sigma = 10$ and $\rho = 2.2$. Moreover,
    \[
    \begin{split}
        \norm{N - \bar{N}}_{C(\R)}, \norm{M - \bar{M}}_{C(\R)} &\le \max\left\{\frac{r_{\rm bvp}}{\mu_{u_1}}, r_{\rm manif} \right\} \\
        &= r_{\rm uniform} \bydef 4.189559197816045 \times 10^{-13}.
    \end{split}
    \]
\end{theorem}
\begin{proof}
    By Theorems~\ref{thm:manif_proof},\ref{thm:bvp_proof} and Lemmas~\ref{lem:parameterization_properties},\ref{lem:bvp_domain_of_convergence}, we have that
    \[
    z(t) \bydef \begin{cases}
        \Gamma (t), & 0 \le t \le 2\tilde{L},\\
        \zeta(t), &t > 2\tilde{L},
    \end{cases}
    \]
    is a solution to \eqref{EQ:4DODE} that satisfies \eqref{eq:reduced_problem} since $\Gamma(2\tilde{L}) = P(\tilde{\theta}) = \zeta(2\tilde{L})$. In other words,
    \[
    \begin{cases}
        z(0) = \Gamma (0), \\
        \lim_{t \to \infty} z(t) = \tilde{x}^{(+)}, \\
        \dot{z} = f(z), &t \ge 0.
    \end{cases}
    \]
    By the definition of the reversor $R$ we extend $z(t)$ for negative values of time as $Rz(-t)$ is also a valid solution to $\dot{z} = f(z)$. This extension of $z(t)$, which we also denote $z(t)$, is continuous as $z(0) \in \Pi_R$ and solves
    \[
    \begin{cases}
        \lim_{t \to \pm \infty} z(t) = \tilde{x}^{(\pm)}, \\
        \dot{z} = f(z), & t \in \R.
    \end{cases}
    \]    
    Now let $N(t) = z_1(t)$ and $M(t) = z_3(t)$, more precisely,
    \[
    N(t) \bydef \begin{cases}
        \zeta_3(-t), & t < -2\tilde{L}, \\
        \Gamma_3(-t), &-2\tilde{L} \le t < 0, \\
        \Gamma_1(t),& 0 \le t \le 2\tilde{L}, \\
        \zeta_1(t),& t > 2\tilde{L},   
    \end{cases}
    \quad
    M(t) \bydef \begin{cases}
        \zeta_1(-t), & t < -2\tilde{L}, \\
        \Gamma_1(-t), &-2\tilde{L} \le t < 0, \\
        \Gamma_3(t),& 0 \le t \le 2\tilde{L}, \\
        \zeta_3(t),& t > 2\tilde{L}. 
    \end{cases}
    \]
    By Theorem~\ref{THM:ODE_equivalence}, we have that $N,M$ define a two-cycle of \eqref{IDE1}. Finally we have that
    \[
    \norm{\bar{N}-N}_{C(\R)} = \max\left\{\sup_{0 \le t \le 2\tilde{L}}|\bar{N}(t)-N(t)|, \sup_{t > 2\tilde{L}}|\bar{N}(t)-N(t)|\right\},
    \]
    by symmetry. For $0 < t \le 2\tilde{L}$, we have
    \[
    |\bar{N}(t)-N(t)| = |\Gamma (t) - \bar{\Gamma }(t)| \le \frac{r_\text{bvp}}{\mu_{u_1}},
    \]
    by Theorem~\ref{lem:bvp_domain_of_convergence}. Now for $t > 2\tilde{L}$, letting $\vartheta_i \bydef \pi_ie^{\Lambda(t - 2L)} \tilde{\theta}$ and using multi-index notation
    \[
    |N(t) - \bar{N}(t)| = |\zeta_1(t) - \bar{\zeta}_1(t)| = |P_1(\vartheta) - \bar{P}_1(\vartheta)| \le r_{\rm manif},
    \]
    by Lemma~\ref{lem:parameterization_properties} since $|\tilde{\theta}_i| < 1$ and the matrix $\Lambda$ is diagonal with negative values. The same can be done to obtain these bounds for $\norm{M - \bar{M}}_{C(\R)}$.    
\end{proof}
\begin{figure}[H]
    \centering
    \includegraphics[width=1\textwidth]{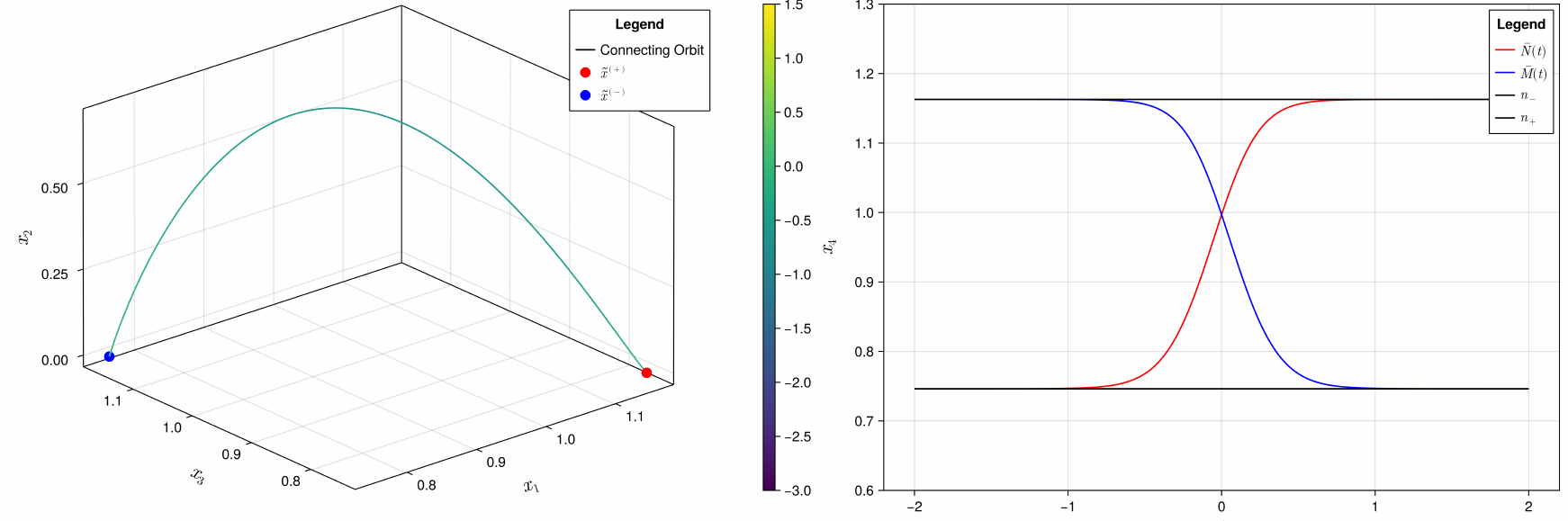}
    \caption{Left: Plot of approximate stable and unstable manifolds along with the connection between the two. In red, we have the two fixed points of interest. The colored sheets are the manifolds and the colored line is the connection between them. The four-dimensional objects are projected to $\R^3$ and the fourth dimension is represented by color. Right: The corresponding spatially inhomogeneous two-cycle of \eqref{IDE1}. In red, we have $\bar{N}(t)=N_0(t)$ and in blue $\bar{M}(t)=N_1(t)$. In black, horizontal lines corresponding to $n_\pm$.}
\end{figure}

\section{Spectral stability of 2-cycles}\label{sec:stability}
In this section, we will assume without loss of generality that $\sigma=1$. Recall that the parameter $\sigma$ can be scaled out of \eqref{IDE1} by a reparameterization of space. We provide a proof of Theorem \ref{thm-stability} and provide some theoretical background on stabilty using the Evans function, before moving on to numerical results.
\subsection{Eigenvalues of the linearized second-iterate map}
\begin{lem}\label{lem-S-frechet}
The second-iterate map $S=Q\circ Q$ is Fr\'echet differentiable, with
\begin{align}\label{S-frechet-derivative}
DS[N]h(x) = \int_\R K(x-y)F'(QN(y))\int_\R K(y-z)F'(N(z))h(z)dzdy,
\end{align}
\end{lem}

\begin{lem}\label{lem-stability-ode-1}
Let $N_t$ be a 2-cycle. Define $\mathcal{M}=DS[N_0]$. If $\mathcal{M}h=\lambda h$ for some $h\in C(\R)$, then $h$ is twice-differentiable. Defining $w(x)=\int_\R K(x-y)F'(N_0(y))h(y)dy$, the pair $(h,w)$ satisfies the following system of ordinary differential equations.
\begin{align}
\label{stability-ode-1-1}\lambda\ddot h&=\lambda h - F'(N_1(x))w\\
\label{stability-ode-1-2}\ddot w&=w-F'(N_0(x))h
\end{align}
\end{lem}


\begin{lem}\label{lem-stability-ode-2}
Let $\lambda\in\C$ with $|\lambda|\geq 1$, and $a,b\in\R$. If $0<ab<1$. No eigenvalue of
\begin{align}
\label{stability-ode-2-1}
\lambda\ddot h&=\lambda h -aw\\
\label{stability-ode-2-2}
\ddot w&=w-bh
\end{align}
is imaginary. More generally, this holds so long as $0<ab<|\lambda|$. The eigenvalues $\mu$ satisfy $\mu^2\in 1\pm\sqrt{ab/\lambda}$.
\end{lem}

\begin{proof}
Defining $v=\dot h$ and $u=\dot w$, we have that $X=(h,v,w,u)$ satisfies
\begin{align*}
\dot X=\left(\begin{array}{cccc}
0&1&0&0\\
1&0&\frac{-a}{\lambda}&0\\
0&0&0&1\\
- b&0&1&0\end{array}\right)X.
\end{align*}
The characteristic polynomial is $\mu^4 - 2\mu^2 + (1-ab/\lambda)$, and the eigenvalues therefore satisfy  $\mu^2\in 1\pm\sqrt{ab/\lambda}.$ Suppose $\mu=i\gamma$ is imaginary. Then $-\gamma^2\in 1\pm\sqrt{ab/\lambda}$, which implies $|ab/\lambda|>1$, contradicting the inequality assertions on $|\lambda|$ and $ab$ from the statement of the lemma.
\end{proof}

%

\begin{proof}[Proof of Theorem \ref{thm-stability}]
We first majorize $\mathcal{M}h(x)$.
\begin{align*}
|Mh(x)|&\leq \int_\R K(x-y)|F'(N_1(y)|\left(\int_\R K(y-z)|F'(N_0(z))|\cdot|h(z)|dz\right)dy\\
&\leq \int_\R K(x-y)|F'(N_1(y)|\left(\int_\R K(y-z)dz\right)||F'\circ N_0||_\infty\cdot ||h||_\infty\\
&\leq \left(\int_\R K(u)du\right)^2||F'\circ N_1||_\infty||F'\circ N_0||_\infty||h||_\infty\\
&=||F'\circ N_1||_\infty||F'\circ N_0||_\infty||h||_\infty,
\end{align*}
It follows that $||\mathcal{M}||_\infty\leq \mu$, so the eigenvalues are contained in the disc of radius $\mu$, as claimed. The remaining parts follow directly from Lemma \ref{lem-stability-ode-1} and Lemma \ref{lem-stability-ode-2}.
\end{proof}

\subsection{Spectral stability}
Our objective in this section is to provide a means of studying what we refer to as \emph{spectral stability} of 2-cycles.
\begin{dfn}
The 2-cycle $N_t$ is \emph{spectrally stable} if $DS[N_0]$ has a single eigenvalue with multiplicity 1 on the unit circle, and all other eigenvalues have absolute value less than unity.
\end{dfn}
\noindent
$DS[N_0]$ always has at least one eigenvalue on the unit circle, corresponding to spatial translation; one can verify that $DS[N_0]\dot N_0=\dot N_0$. Note that spectral stability does not necessarily imply (local) nonlinear stability of the 2-cycle under the dynamics of \eqref{IDE1}; this line of questioning is beyond the scope of our current work. However, necessarily, spectral stability is a requirement for nonlinear stability.

Using the usual transformation of \eqref{stability-ode-thm-1-1}--\eqref{stability-ode-thm-1-2} to a first-order system of ordinary differential equations, Theorem \ref{thm-stability} guarantees that the eigenvalues $\lambda$ of $\mathcal{M}=DS[N_0]$ are associated with bounded solutions of 
\begin{align}\label{evans-ode-z}
\dot z(t)&=A(t;\lambda)z(t),\hspace{10mm}A(t;\lambda)=\left(\begin{array}{cccc}0&1&0&0\\ 1&0&-F'(N_1(t))\lambda^{-1}&0\\ 0&0&0&1 \\ -F'(N_0(t)) &0&1&0 \end{array}\right)
\end{align}
Let us catalogue some properties of this linear system.
\begin{enumerate}
\item Since each of $N_0(t)$ and $N_1(t)$ are dominated by the dynamics on stable/unstable manifolds as $|t|\rightarrow\infty$, the convergence $\lim_{t\rightarrow\pm\infty} A(t;\lambda)\bydef A_{\pm\infty}(\lambda)$ is exponential. Moreover, this convergence is clearly uniformly exponential in compact subsets of $\lambda$ bounded away from zero.
\item The limiting matrices $A_{\pm\infty}(\lambda)$ are analytic on the punctured complex plane.
\item If $0<F'(n_+)F'(n_-)<|\lambda|$, then each of $A_{\pm}(\lambda)$ are hyperbolic, with two unstable and two stable eigenvalues; this follows from Lemma \ref{lem-stability-ode-2}.
\item If the computer-assisted proof of the connecting orbit is successful, we claim that $t\mapsto A(t;\lambda)$ is analytic on the real line. More strongly, it is analytic on a strip of positive width in the complex plane, containing the real line. This can be proven with a careful application of the identity theorem of complex analysis. As a rigorous proof of stability is not the focus of this work, we will not provide the details at this time. Instead, we will from this point on assume analyticity of $t\mapsto A(t,\lambda)$.
\item A bounded solution of \eqref{evans-ode-z} must necessarily be tangent to the unstable subspace of $A_{-\infty}(\lambda)$ as $t\rightarrow -\infty$, and be tangent to the stable subspace of $A_{+\infty}(\lambda)$ as $t\rightarrow +\infty$.
\end{enumerate}
As consequence of the above observations (and assumptions on analyticity), we are in a position to define an Evans function (c.f.\ \cite{JonesGardnerAlexander,BarkerHumpherys} ), whose zeroes will coincide precisely (counting multiplicity) with eigenvalues  $\lambda$ of $M$ outside of the disc of radius $F'(n_+)F'(n_-)$. To investigate spectral stability of the 2-cycle, we recall from Theorem \ref{thm-stability} that we have ruled out eigenvalues with absolute value greater than $\mu=||F'\circ N_0||_\infty||F'\circ N_1||_\infty$. Taking this into account together with complex conjugacy, to establish spectral stability, it is enough to count the eigenvalues of $M$ in 
\begin{align}\label{stablity-Lambda-set}
\Lambda_{\epsilon_1,\epsilon_2}&=\left\{ z\in\C : -\epsilon_1\leq\mbox{Arg}(z)\leq \pi+\epsilon_1,\hspace{1mm}1-\epsilon_2\leq |z|\leq \mu+\epsilon_2\right \}
\end{align}
for some suitably-chosen $0<\epsilon_1<\pi$ and $0<\epsilon_2<1-F'(n_+)F'(n_-)$; see Figure \ref{fig:Lambda-contour}. For $\lambda\in\Lambda_{\epsilon_1,\epsilon_2}$, as consequence of point 3 above, the eigenvalues of $A_{\pm\infty}(\lambda)$ are bounded away from the imaginary axis, and the stable/unstables subspaces are each two-dimensional. 

\begin{figure}
\centering
\includegraphics[scale=0.6]{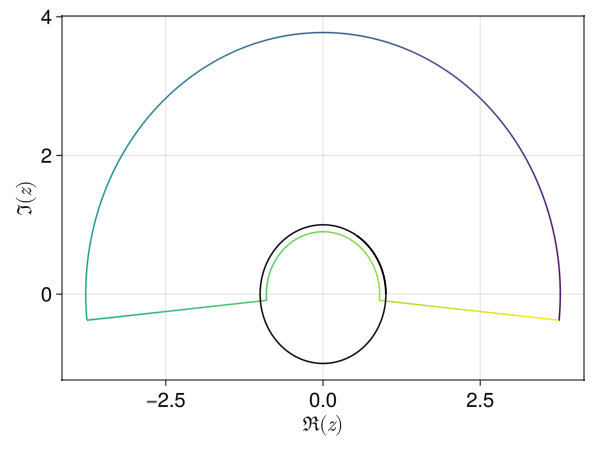}
\caption{Plot of the boundary of $\Lambda_{\epsilon_1,\epsilon_2}$ as parameterized by $\gamma:[0,1]\rightarrow\Lambda_{\epsilon_1,\epsilon_2}$ in the positive orientation (i.e.\ the interior of $\Lambda_{\epsilon_1,\epsilon_2}$ is to the ``left'' of its boundary), starting from the right endpoint of the larger circle; colour of the curve at each point corresponds to $t\in[0,1]$. Unit circle (black) for scale. Parameters $\epsilon_1=\epsilon_2=0.1$ and $\mu=3.6731$ are used here, which are the parameters used for the logistic 2-cycle calculations; see Section \ref{sec:logistic-evans}.}\label{fig:Lambda-contour}
\end{figure}

We follow the construction of the Evans function provided in \cite{BarkerHumpherys}. While that work is focused on stability of traveling waves, fundamentally, the identification with the multiplicity of a zero of the Evans function and the dimension of the space of bounded solutions of \eqref{evans-ode-z} is a consequence of exponential dichotomies, which is agnostic to the evolution equation from which the system $\dot z(t)=A(t;\lambda)z(t)$ is derived. See \cite{JonesGardnerAlexander} for details. It therefore applies to our situation. To our knowledge, this may be a first application of the Evans function to the stability of solutions of integrodifference equations.

In the following sections, for brevity, we will say that $\epsilon_1$ and $\epsilon_2$ satisfy the \emph{requisite inequalities} if $0<\epsilon_1<\pi$ and $0<\epsilon_2<1-F'(n_+)F'(n_-)$.

\subsubsection{Analytic basis construction at $\pm\infty$}
The first step is to construct two analytic bases, as suggested by point 5 above, characterizing the behaviour of bounded solutions at $\pm\infty$: 
\begin{itemize}
\item $V_-(\lambda)$: a basis for the unstable subspace of $A_{-\infty}(\lambda)$.
\item $V_+(\lambda)$: a basis for the stable subspace of $A_{+\infty}(\lambda)$.
\end{itemize}
To avoid confusion between the eigenvalues of the operator $M=DS[N_0]$ and the eigenvalues of the matrices $A_{\pm\infty}(\lambda)$, we will refer to the latter as \emph{growth modes} from this point onward. The \emph{stable growth modes} are those with negative real part, and the \emph{unstable growth modes} are those with positive real part. 

Theorem \ref{thm-stability} provides an explicit formula for the growth modes, and these will be needed in the basis calculations. Denoting the unstable growth modes of $A_{-\infty}(\lambda)$ by $\xi_{-\infty,k}(\lambda)$ for $k=0,1$, and the stable growth modes of $A_{+\infty}(\lambda)$ by $\xi_{+\infty,k}(\lambda)$, we can compactly write
\begin{align*}
\xi_{\pm\infty,k}(\lambda)=\mp\sqrt{1+(-1)^k\sqrt{F'(n_+)F'(n_-)/\lambda}}.
\end{align*}
However, to ensure that the growth modes are represented as a functions that are analytic on $\Lambda_{\epsilon_1,\epsilon_2}$, we must be careful to specify the branches of square roots. 

\begin{lem}\label{lem-analytic-growth}
Let $\Log_{\pi/2}$ be a complex logarithm that is analytic off the positive imaginary axis, and let $\Log$ be the principal logarithm. Define $\Sqrt_{\pi/2}(z)=\exp\left(\frac{1}{2}\Log_{\pi/2}(z)\right)$ and $\Sqrt(z)=\exp\left(\frac{1}{2}\Log(z)\right)$. Suppose $0<F'(n_+)F'(n_-)<1$ and $\epsilon_1,\epsilon_2$ satisfy the requisite inequalities. The function
\begin{align}\label{analytic-growth-modes}\tilde\xi_{\pm\infty,k}(\lambda)=\mp \Sqrt\left(1+(-1)^k \Sqrt_{\frac{\pi}{2}}\left(F'(n_+)F'(n_-)/\lambda\right)\right)
\end{align}
is analytic on $\Lambda_{\epsilon_1,\epsilon_2}$, and $\{\tilde\xi_{\pm\infty,k}(\lambda),\hspace{1mm}k=0,1\}=\{\xi_{\pm\infty,k}(\lambda),\hspace{1mm}k=0,1\}$.
\end{lem}
\begin{proof}
For $\lambda\in\Lambda_{\epsilon_1,\epsilon_2}$, we have $F'(n_+)F'(n_-)/\lambda\notin i\R^+$, so $\Sqrt_{\pi/2}(\cdot)$ is an analytic square root of  $F'(n_+)F'(n_-)/\lambda$. Since the growth modes are not imaginary -- see Lemma \ref{lem-stability-ode-2} -- we have $1+(-1)^k\Sqrt_{\pi/2}(F'(n_+)F'(n_-)/\lambda)\notin\R^-$, and it follows that $\Sqrt(\cdot)$ is an analytic square root on the range of $\Lambda_{\epsilon_1,\epsilon_2}\ni\lambda\mapsto1+(-1)^k\Sqrt_{\pi/2}(F'(n_+)F'(n_-)/\lambda)$. Since complex square roots are unique up to the reflection $z\mapsto -z$ and we have taken $\Sqrt$ to be the principal square root (which maps the upper and lower half-planes into themselves), it follows that $\tilde\xi_{\pm\infty,k}(\lambda)$ is equal to $\xi_{\pm\infty,k}(\lambda)$ up to a shift $k\mapsto k+1\mod 2$.
\end{proof}

\begin{rem}
A suitable choice of $\Log_{\pi/2}$, which we use in our numerical implementation, is given by $z\mapsto \log|z| + i\left(\Arg(iz) - \frac{\pi}{2}\right)$, where $\Arg$ denotes the principal argument.
\end{rem}

\begin{lem}\label{lem-analytic-bases}
In the following, $V_-(\lambda)$ is a basis for the unstable subspace of $A_{-\infty}(\lambda)$, and $V_{+}(\lambda)$ is a basis for the stable subspace of $A_{+\infty}(\lambda)$.
\begin{align}\label{analytic-bases}
V_\pm(\lambda)=\left\{\left(\begin{array}{c}1-\xi_{\pm\infty,k}^2 \\ (1-\xi_{\pm\infty,k}^2)\xi_{\pm\infty,k} \\ F'(n_\pm) \\ F'(n_\pm)\xi_{\pm\infty,k} \end{array}\right),\hspace{2mm}k=0,1\right\}=\{V_{\pm,k}(\lambda),\hspace{2mm}k=0,1\}.
\end{align}
If the analytic representation \eqref{analytic-growth-modes} of the growth modes is selected, $0<F'(n_-)F'(n_+)<1$ and $\epsilon_1,\epsilon_2$ satisfy the requisite inequalities, then the bases $V_-(\lambda)$ and $V_+(\lambda)$ defined in \eqref{analytic-bases} are analytic.
\end{lem}

\begin{proof}
Analyticity of $V_{\pm}(\lambda)$ assuming the analytic growth mode follows directly from Lemma \ref{lem-analytic-growth}. That $V_{\pm}(\lambda)$ is a basis of $A_{\pm\infty}(\lambda)$ follows by explicitly computing the kernel of $A-\xi_{\pm\infty,k}I$. First, we have
{\small 
\begin{align*}
A-\xi I = \left(\begin{array}{cccc}
-\xi&1&0&0\\
 1&-\xi&-F'(n_{\mp})/\lambda&0\\
 0&0&-\xi&1\\
 -F'(n_\pm)&0&1&-\xi \end{array}\right) \sim 
\left(\begin{array}{cccc}
-\xi&1&0&0\\
 1-\xi^2&0&-F'(n_{\mp})/\lambda&0\\
 0&0&-\xi&1\\
 -F'(n_\pm)&0&1-\xi^2&0 \end{array}\right)
\end{align*}
}
where similarity $\sim$ is under elementary row operations, and we write $\xi=\xi_{\pm\infty,k}$ for compactness. Since $(1-\xi^2)^2 = F'(n_\pm)F'(n_\mp)/\lambda$, the second and fourth rows are dependent, so
{\small
\begin{align*}
A-\xi I \sim \left(\begin{array}{cccc}
-\xi&1&0&0\\
 -F'(n_\pm)&0&1-\xi^2&0\\
 0&0&-\xi&1\\
  0&0&0&0 \end{array}\right).
 \end{align*}
 }
Checking the kernel is now straightforward.
\end{proof}

\subsubsection{The Evans function and spectral stability}
Continuing to follow the presentation in \cite{BarkerHumpherys}, an Evans function can be defined theoretically by
\begin{align}
D(\lambda)&=\det([\begin{array}{cccc}W_0^-(t;\lambda)&W_1^-(t;\lambda)&W_0^+(t;\lambda)&W_1^+(t;\lambda)\end{array}])|_{t=0},
\end{align}
where the columns $W_k^{\pm}$ satisfy the following.
\begin{itemize}
\item $x\mapsto W_k^{-}(x;\lambda)$ for $k=0,1$ are two linearly independent solutions of \eqref{evans-ode-z} tangent as $t\rightarrow -\infty$ to the span of $V_-(\lambda)$.
\item $x\mapsto W_k^{+}(x;\lambda)$ for $k=0,1$ are two linearly independent solutions of \eqref{evans-ode-z} tangent as $t\rightarrow +\infty$ to the span of $V_+(\lambda)$.
\end{itemize}
Following our earlier preparations (Lemma \ref{lem-analytic-bases}), the Evans function is analytic \cite{BarkerHumpherys} on $\Lambda_{\epsilon_1,\epsilon_2}$ provided the requisite inequalities are satisfied. As consequence of Cauchy's argument principle, if $\gamma$ is a parameterization of the boundary $\partial\Lambda_{\epsilon_1,\epsilon_2}$ and $D\circ\gamma$ does not vanish, then the winding number of $D\circ\gamma$ around $0\in\C$ is equal to the number of zeroes (with multiplicity) of $D$ in $\Lambda_{\epsilon_1,\epsilon_2}$. By previous discussions, this is the same as the number of eigenvalues (with multiplicity) of $M$ in $\Lambda_{\epsilon_1,\epsilon_2}$, which includes all possible eigenvalues with absolute value greater than or equal to unity.

We defer computations of an Evans function with explicit error bounds sufficient to perform winding number calculations -- and subsequently to prove spectral stability -- to future work. For the time being, we content ourselves with the following approximation; see \cite{BarkerHumpherys} for details. We define, for $L>0$ fixed, an approximate Evans function $E_L$, as follows:
\begin{align}\label{approximate-Evans}
E_L(\lambda)=\det\left([\begin{array}{cc}X_{-L}(t;\lambda)V_-(\lambda)&X_{L}(t;\lambda)V_{+}(\lambda)\end{array}]\right)|_{t=0},
\end{align}
where $X_s(t;\lambda)$ is the Cauchy matrix associated with \eqref{evans-ode-z}. That is, $X_s(t)=U(t)U(s)^{-1}$ for any fundamental matrix solution $U$. Due to the exponential convergence of $A(t;\lambda)$ to $A_{\pm\infty}(\lambda)$ as $t\rightarrow\pm\infty$, the zeroes of $D$ can be approximated by those of $E_L$, for $L$ sufficiently large; see Section 3(f) of \cite{BarkerHumpherys} for details. In any case, we have the following lemma.

\begin{lem}\label{lem-approximate-Evans}
If the 2-cycle is analytic, $0<F'(n_+)F'(n_-)<1$ and $\epsilon_1,\epsilon_2$ satisfy the requisite inequalities, then the approximate Evans function $E_L$ defined in \eqref{approximate-Evans} by way of the bases $V_{\pm}(\lambda)$ from \eqref{analytic-bases}, is analytic on $\Lambda_{\epsilon_1,\epsilon_2}$ for any $L\geq0$.
\end{lem}

\subsection{Numerical spectral stability results}
In this section, we make use of the approximate Evans function \eqref{approximate-Evans} to investigate the spectral stability of 2-cycles of \eqref{IDE1}. We do this both for our proven 2-cycle, which uses the logistic growth function, as well as our yet-to-be-proven numerical candidate of the 2-cycle for the Ricker growth function.

Numerical implementation of the approximate Evans function $E_L$ defined in \eqref{approximate-Evans} is challenging. The numerical integration of the Evans system \eqref{evans-ode-z} is dominated by the most strongly unstable growth mode, which leads to poor resolution of the dynamics from slower modes. Since the dimension of  \eqref{evans-ode-z} is low, we resolve this issue using the compound matrix method \cite{Brin2000}, which is simple to implement. We briefly review that method now.

The Evans system \eqref{evans-ode-z} on $\C^4$ is lifted to an ordinary differential equation on the 2nd exterior power $\bigwedge^2(\C^4)$, which is isomorphic to $\C^6$. Specifically, if $e_1,\dots e_4$ denotes the standard ordered basis of $\C^4$, one constructs the six-element basis $B=\{e_i\wedge e_j : i<j,\hspace{1mm}i,j=1,\dots,4\}$ of $\bigwedge^2(\C^4)$. The isomorphism with $\C^6$ is realized by extending the following linearly:
\begin{align*}
B\ni e_i\wedge e_j\mapsto \left\{ \begin{array}{ll}
\tilde e_{j-1},&i=1\\
\tilde e_{j+1},&i=2\\
\tilde e_6,&i=3, \end{array}\right. 
\end{align*}
where $\tilde e_1,\dots,\tilde e_6$ is the standard ordered basis for $\C^6$. If we define $A^{(2)}(t;\lambda):\bigwedge^2(\C^4)\rightarrow \bigwedge^2(\C^4)$ by 
$$A^{(2)}(t;\lambda)v\wedge w = (A(t;\lambda)v)\wedge w + v\wedge (A(t;\lambda)w),$$ then one can check using anti-commutativity of the wedge product that, under the isomorphism $\bigwedge^2(\C^4)\cong\C^6$ defined previously, the matrix representation of $A^{(2)}$ is \footnotesize{
\begin{align}
A^{(2)}(t;\lambda)&=\left(\begin{NiceArray}{cccccc}[columns-width=20pt]
0&-F'(N_1(t))\lambda^{-1}&0&0&0&0\\
0&0&1&1&0&0\\
0&1&0&0&1&0\\
0&1&0&0&1&0\\
F'(N_0(t))&0&1&1&0&-F'(N_1(t))\lambda^{-1}\\
0&F'(N_0(t))&0&0&0&0
 \end{NiceArray}\right).
\end{align}}
\normalsize
See e.g. \cite{Blake-thesis} for details. One can then show \cite{Brin2000} that $E_L(\lambda)=\mathcal{W}^+(\lambda,0)\wedge\mathcal{W}^-(\lambda,0)$, where 
\begin{align}
\label{evans-BVP2-1}\partial_t\mathcal{W}^{\pm}(\lambda,t)&=A^{(2)}(t;\lambda)\mathcal{W}^{\pm}(\lambda,t),\\
\label{evans-BVP2-2}\mathcal{W}^{\pm}(\lambda,\pm L)&=V_{\pm,0}(\lambda)\wedge V_{\pm,1}(\lambda).
\end{align}
Note that in the equation $E_L(\lambda)=\mathcal{W}^+(\lambda,0)\wedge\mathcal{W}^-(\lambda,0)$, the latter wedge product is interpreted as an element of $\C$, since the wedge product of two elements in $\bigwedge^2(\C^4)$ is in $\bigwedge^4(\C^4)\cong\C$. The advantage of this formulation is that the system \eqref{evans-BVP2-1} has a single dominant asymptotic growth mode at $\pm\infty$, which allows it to be effectively isolated in each of the separate calculations of $\mathcal{W}^-$ and $\mathcal{W}^+$. Also, the growth modes are expressible in terms of the growth modes (or analytic representations; see \eqref{analytic-growth-modes}) of the original Evans system; see Lemma 3 of \cite{Brin2000}. At $\pm\infty$, the growth modes are $\zeta_{\pm}\bydef \xi_{\pm\infty,0} + \xi_{\pm\infty,1}.$ In our implementation, we attain good numerical stability by scaling out the dominant growth modes. Specifically, we modify \eqref{evans-BVP2-1}--\eqref{evans-BVP2-2} to
\begin{align}
\label{evans-BVP3-1}\partial_t\mathcal{U}^{\pm}(\lambda,t)&=(A^{(2)}(t;\lambda)- \zeta_{\pm}I)\mathcal{U}^{\pm}(\lambda,t)  ,\\
\label{evans-BVP3-2}\mathcal{U}^{\pm}(\lambda,\pm L)&=V_{\pm,0}(\lambda)\wedge V_{\pm,1}(\lambda).
\end{align}
One can show that $\mathcal{U}^+(\lambda,0)\wedge\mathcal{U}^-(\lambda,0)=kE_L(\lambda)$ for a positive real constant $k$. For winding number and root-finding calculations associated to the Evans function, this is sufficient. In subsequent sections, we therefore abuse notation and define $E_L(\lambda)=\mathcal{U}^+(\lambda,0)\wedge\mathcal{U}^-(\lambda,0)$.

We should mention that a MATLAB library has been developed \cite{Blake-thesis} to facilitate stability and Evans function calculations. In this work, we use Julia and implement the compound matrix method directly for our problem. We use Grassmann.jl \cite{Reed_Differential_geometric_algebra_2019} to facilitate the exterior algebra computations, and DifferentialEquations.jl with default tolerances for numerical integration. For the latter, the system automatically selected a stiff solver.

\subsubsection{Logistic growth function}\label{sec:logistic-evans}
For the logistic growth function, we can determine an explicit, tight bound for $\mu=||F'\circ N_0||_\infty||F'\circ N_1||_\infty$. Calculating this bound is fairly direct, since $F'(u)=1+\rho-2\rho u$ is affine-linear and the 2-cycle has been proven analytically, which facilitates a tight computation of the supremum norms. We find  $\mu \approx 3.6731$ and $F'(n_-)F'(n_+)\approx 0.1599$. Following Lemma \ref{lem-approximate-Evans}, we assume analyticity of the 2-cycle and select $\epsilon_1=\epsilon_2=0.1$, which satisfy the requisite inequalities. We numerically compute the winding number by evaluating $E_L(\lambda)$ for $L=8$ over a one-dimensional mesh with spacing $0.05$, parameterizing the boundary of $\Lambda_{\epsilon_1,\epsilon_2}$ in the positive orientation (i.e.\ the interior of the domain is on the ``left'' of the boundary). For the parameterization $\gamma$ of the boundary, we have plotted $E_L\circ\gamma$ in Figure \ref{fig:Evans-logistic}. It is difficult to resolve the winding number due to the proximity of the true zero of the Evans function $E(\lambda)$ at $\lambda=1$, so we have provided a zoomed-in view near zero. From there, we can verify from a visual inspection that the windng number is $+1$. Based on this, we conjecture that our proven 2-cycle of \eqref{IDE1}, with logistic growth, is spectrally stable. Also, we find that $E_L(1)=-4.486\times 10^{-7}$, reasonably close to numerical zero, as expected.

\begin{figure}
\includegraphics[scale=0.5]{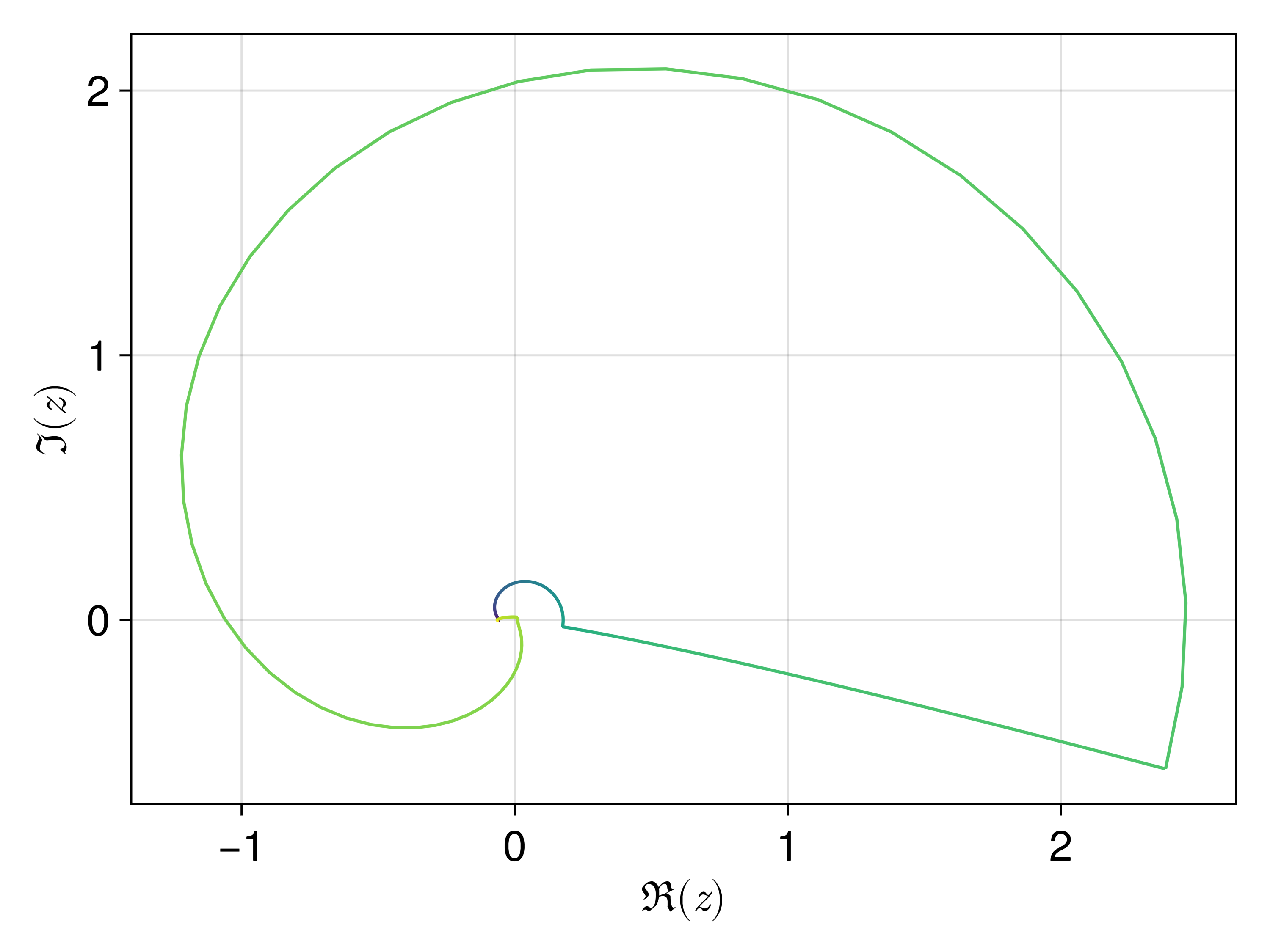}\includegraphics[scale=0.5]{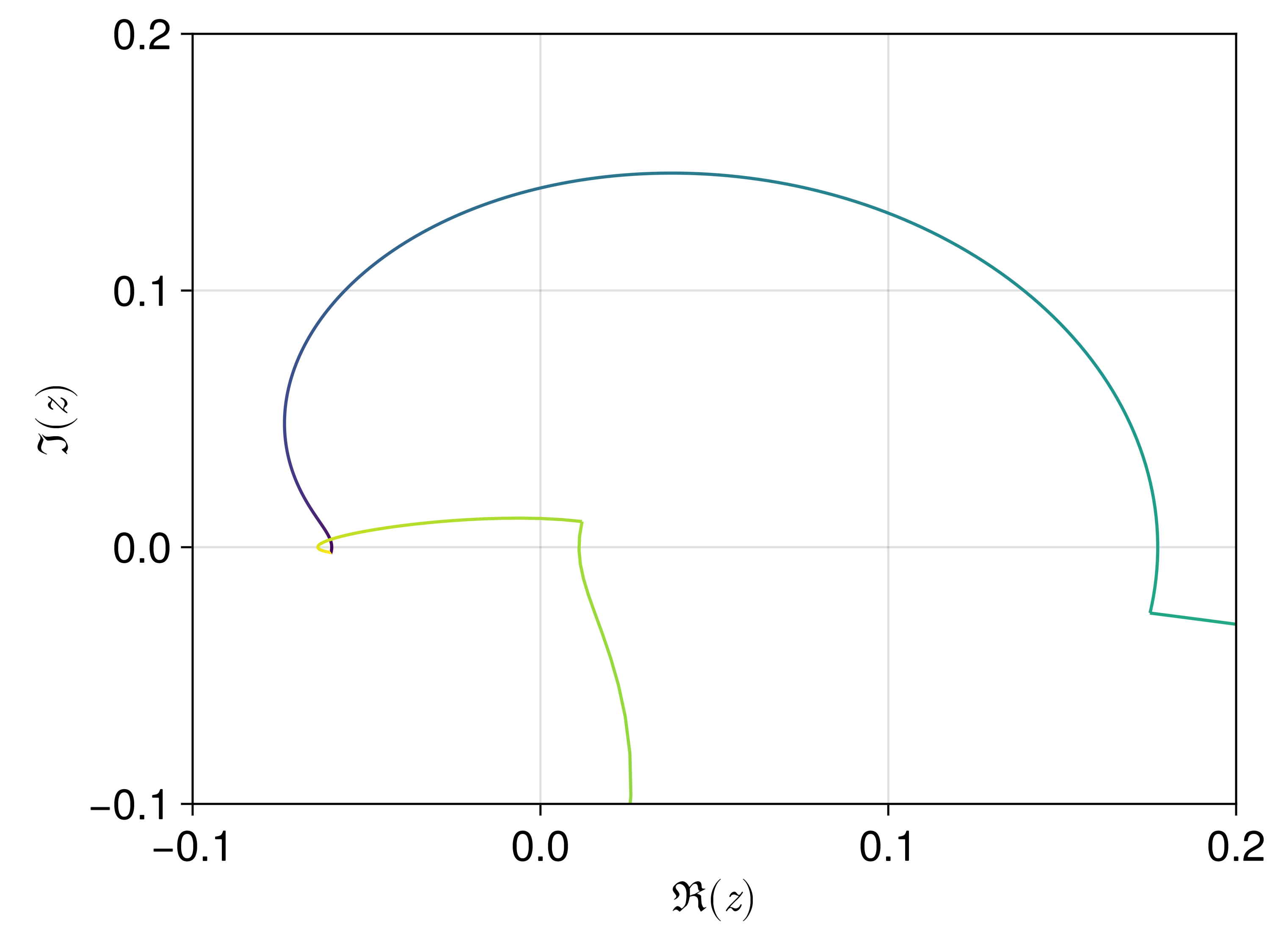}
\caption{Left: Plot the image of $[0,1]\ni t\mapsto E_L\circ\gamma(t)$, with $\gamma$ being a properly-oriented parameterization of the boundary of $\Lambda_{\epsilon_1,\epsilon_2}$, with $E_L$ calculated for the logistic 2-cycle. Bottom right: Zoomed-in near zero. In both cases, colour corresponds to $t\in[0,1]$; see Figure \ref{fig:Lambda-contour}.}\label{fig:Evans-logistic}
\end{figure}

\subsubsection{Ricker growth function}
The derivative of the Ricker map is $F'(u)=(1-\rho u)\exp(\rho(1-u))$. Since our 2-cycle is not yet proven, our bound for $\mu$ is subject to as-yet unquantified error, and we therefore rely solely on the numerically-computed finite-dimensional approximate solution. We find $\mu\approx 1.4918$ and $F'(n_+)F'(n_-)\approx 0.2157$. We select $\epsilon_1=\epsilon_2=0.3$, which satisfy the requisite inequalities. We once again use $L=8$ for the approximate Evans function calculation $E_L$, and we plot $E_L\circ\gamma$, for $\gamma$ parameterizing the boundary of $\Lambda_{\epsilon_1,\epsilon_2}$ in the positive orientation, with mesh spacing $0.05$. The result appears in Figure \ref{fig:Evans-ricker}. Visually inspecting, we can see that the winding number appears to be $+1$. We conjecture that \eqref{IDE1} with the Ricker growth function has a 2-cycle connecting $n_+$ and $n_-$, and that this 2-cycle is spectrally stable. Also, we find that $E_L(1)=-7.340\times 10^{-7}$, reasonably close to numerical zero, as expected.

\begin{figure}
\includegraphics[scale=0.5]{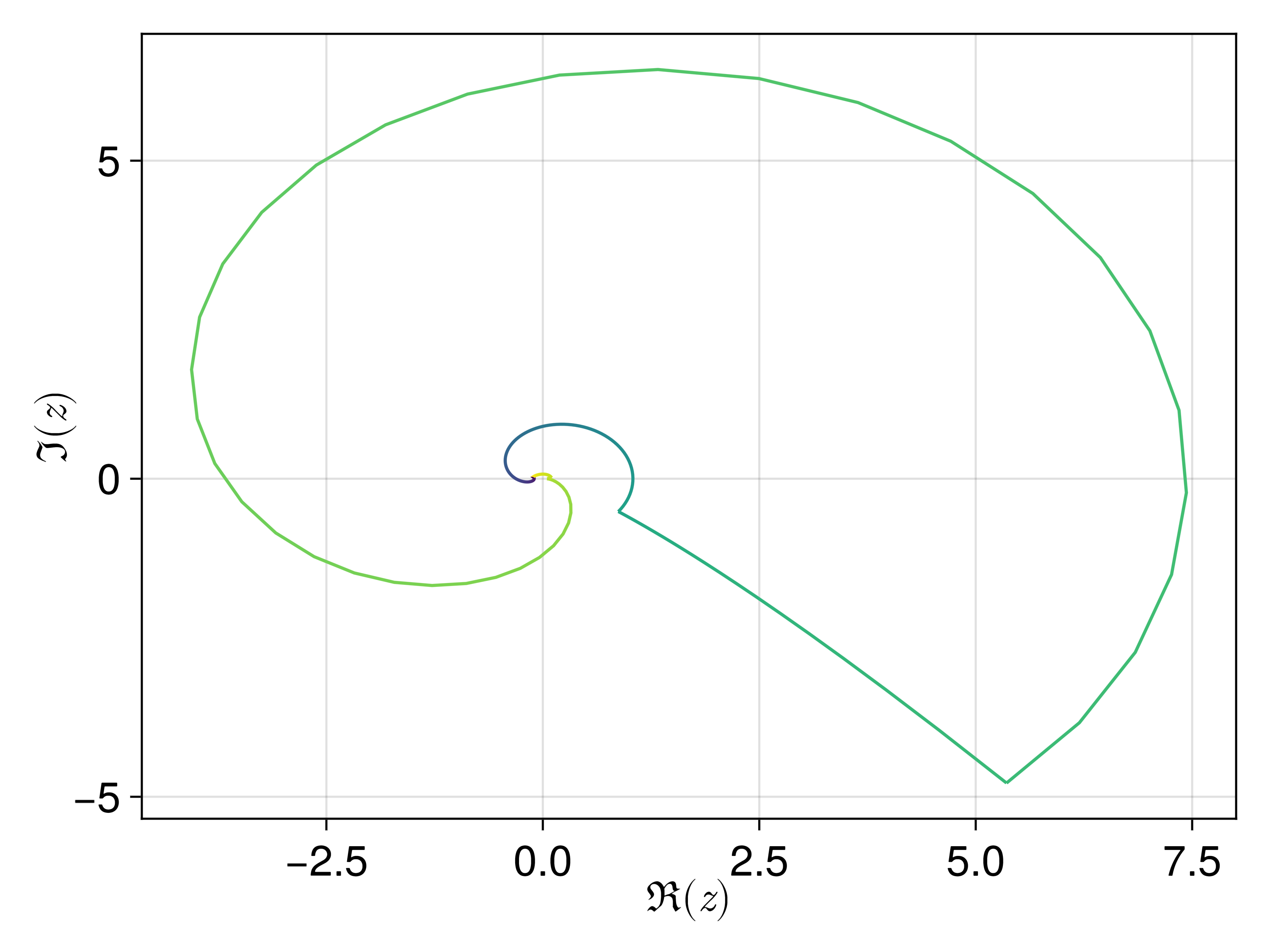}\includegraphics[scale=0.5]{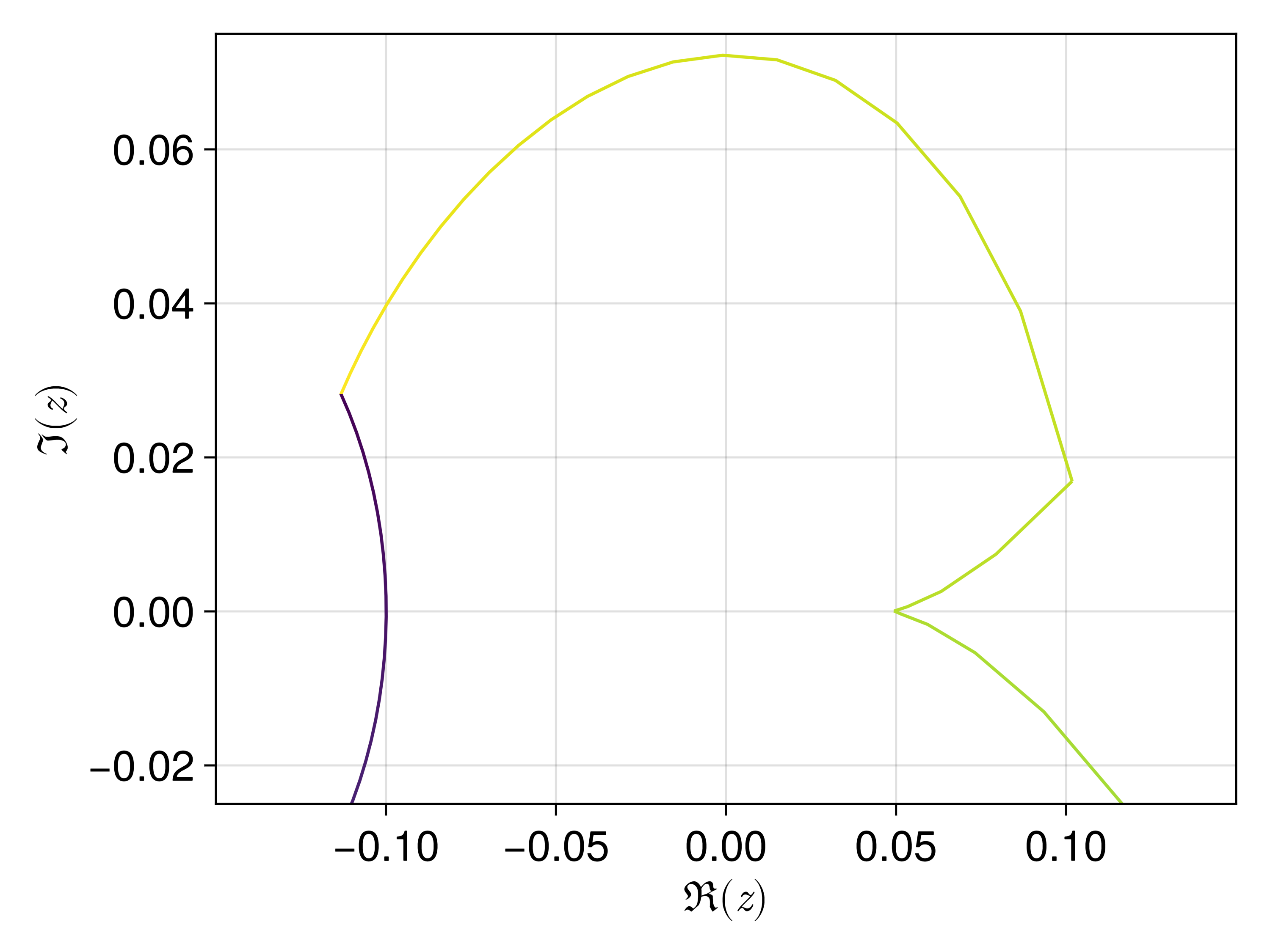}
\caption{Left: Plot the image of $[0,1]\ni t\mapsto E_L\circ\gamma(t)$, with $\gamma$ being a properly-oriented parameterization of the boundary of $\Lambda_{\epsilon_1,\epsilon_2}$, with $E_L$ calculated for the Ricker 2-cycle. Bottom right: Zoomed-in near zero. In both cases, colour corresponds to $t\in[0,1]$; see e.g.\ Figure \ref{fig:Lambda-contour}.}\label{fig:Evans-ricker}
\end{figure}

\section{Future work}\label{sec-future}
In this work, we have used the Evans function to count the number of unstable eigenvalues of the linearization $M=DS[N_0]$. While this does not constitute a proof of spectral stability, we are confident that a computer-assisted proof of such a result is feasible. 

It is all but guaranteed that Theorem \ref{thm-logistic} can be extended to accomodate the Ricker growth function. The approach taken in this work could be replicated for other growth functions. It is unclear whether our approach could apply for other dispersal kernels, since the transformation to an equivalent second-order ordinary differential equation \eqref{ODE1}--\eqref{ODE2} explicitly relies on a property of the Laplace kernel.

In 1992, based on some numerical simulations, Kot conjectured \cite{Kot1992} that the \eqref{IDE1} supports travelling 2-cycles. These solutions are essentially a combination of a 2-cycle and a travelling wave solution: they satisfy $N_{t+2}(x)=N_t(x-\tau)$ for some shift $\tau$, such that $N_{t+1}$ is not lateral shift of $N_t$. Later numerical work by Bourgeois, Leblanc and Lutscher \cite{Bourgeois2018} further supports this claim, although the wave profiles each set of authors found were different. One can show that a travelling 2-cycle of \eqref{IDE1} with the Laplace kernel is equivalent to a connecting orbit in the second-order system of delay (respectively advanced, if $\tau>0$) differential equations
\begin{align*}
\ddot x(t)&=\sigma^2(x(t)-F(y(t+\tau)))\\
\ddot y(t)&=\sigma^2(y(t)-F(x(t))).
\end{align*} 
There is some hope that the parameterization methods for stable and unstable manifolds of delay differential equations \cite{Groothedde2017,Hnot2023} could be used to also prove the existence of these travelling 2-cycles. 

\bibliographystyle{unsrt}
\bibliography{papers,library}

\end{document}